\documentclass{ws-procs9x6}

\usepackage{bm}
\bmdefine{\BA}{A}
\bmdefine{\BB}{B}
\bmdefine{\BI}{I}
\bmdefine{\Bi}{i}
\bmdefine{\Bp}{p}
\bmdefine{\Bt}{t}
\bmdefine{\BX}{X}
\bmdefine{\BY}{Y}
\bmdefine{\Bb}{b}
\bmdefine{\Bx}{x}
\bmdefine{\By}{y}
\bmdefine{\Bz}{z}

\newcommand{\Z}{{\mathbb Z}}
\newcommand{\R}{{\mathbb R}}

\newcommand{\iid}{\stackrel{\mathrm{iid}}{\sim}}

\newcommand{\cB}{{\cal B}}
\newcommand{\cF}{{\cal F}}

\newcommand{\cI}{{\cal I}}
\newcommand{\rank}{\operatorname{rank}}
\newcommand{\kerz}{\operatorname{\ker_{\mathbb Z}}}
\usepackage{amssymb}
\usepackage{graphicx}
\makeatletter
\newcommand*\dashline{\rotatebox[origin=c]{90}{$\dabar@\dabar@\dabar@$}}
\makeatother

\begin{document}

\title{RUNNING MARKOV CHAIN WITHOUT MARKOV BASIS}

\author{H. HARA}

\address{Faculty of Economics, Niigata University,\\
Niigata, 950-2181, Japan\\
E-mail: hara@econ-niigata-u.ac.jp}

\author{S. AOKI}

\address{
Department of Mathematics and Computer Science, 
Kagoshima University\\ 
Kagoshima, 890-0065, Japan\\
E-mail: aoki@sci.kagoshima-u.ac.jp\\
JST, CREST}

\author{A. TAKEMURA}

\address{Department of Mathematical Informatics, \\
Graduate School of Information Science and Technology, 
University of Tokyo, \\
Tokyo, 113-8656, Japan\\
E-mail: takemura@stat.t.u-tokyo.ac.jp\\
JST, CREST}

\begin{abstract}
The methodology of Markov basis initiated by Diaconis and Sturmfels\cite{diaconis-sturmfels}
stimulated active research on Markov bases for more than ten years.
It also motivated improvements of algorithms for Gr\"obner basis computation  for toric ideals,
such as those implemented in 4ti2\cite{4ti2}.
However at present explicit forms of Markov bases are known only 
for some relatively simple models, such as the decomposable models of contingency tables.
Furthermore general algorithms for Markov bases computation often fail to
produce Markov bases even for moderate-sized  models
in a practical amount of time. 
Hence so far we could not perform exact tests based on Markov basis methodology for
many important practical problems. 

In this article we propose to use lattice bases for performing exact tests, in the case where
Markov bases are not known.  Computation of lattice bases is much easier than that of Markov bases.
With many examples we show that the approach with lattice bases is practical. 
We also check that its performance is comparable to Markov bases for
the problems where Markov bases are known.
\end{abstract}


\bodymatter

\section{Introduction}
\label{sec1}
Since Diaconis and Sturmfels\cite{diaconis-sturmfels} introduced a
Markov basis and proposed an algorithm of exact test by
sampling contingency tables sharing a sufficient statistic, 
the algebraic and statistical properties of Markov bases for
toric models have been extensively studied. 
Once a Markov basis is given, we can perform an exact test
by using the basis.
There exist algebraic algorithms to compute a Markov basis and 
a Markov basis of models for relatively small contingency tables can be 
computed by a computer algebra system such as 4ti2\cite{4ti2}.
However the computational cost of these algorithms 
is very high and at present it is difficult to
compute a Markov basis for even moderate-sized models by such softwares
in a practical amount of time. 

For important models for applications  we can investigate the structure of Markov
bases for the model. 
In general, however, the structure is complicated and explicit forms
Markov bases are known only for a few models  
such as the decomposable model\cite{dobra-2003bernoulli}, no-three-factor 
interaction model for relatively small tables\cite{aoki-takemura-2003anz}.   
Considering the fact that an exact test is needed especially when the
sample size is relatively small for the degrees of freedom of the model
and the chi-square approximation of a test statistic is not accurate, 
these results at this point are not satisfactory from a practical viewpoint. 

The set of contingency tables sharing a sufficient statistic is called a fiber. 
Markov basis is defined as a set of moves connecting every fiber. 
One reason for the complexity of Markov bases is that they guarantee the
connectivity of every fiber. 
In practice, we only need to connect a fiber which a given data set belongs to. 
Sometimes we can find a useful subset of a Markov basis which has a
simple structure and guarantees the connectivity of particular fibers
\cite{aoki-takemura-2005jscs,chen-dinwoodie-dobra-huber2005,
hara-takemura-yoshida-logistic}.   
However, again,  
such a subset is not easy to obtain in general \cite{chen-dinwoodie-yoshida-2008}.

In view of these difficulties with Markov bases,  for performing exact tests
we propose to use a lattice basis, which is a basis of the
integer kernel of a configuration matrix,  instead of a Markov basis. 
Computation of lattice bases is much easier than computation of Markov basis.
With many examples we show that the proposed approach is practical. 
%
Note that a lattice basis itself does not guarantee the connectivity of every fiber.
However every move is written as an integer combination of elements of 
a lattice basis.   
Hence,  if we generate moves in such a way that every integer  combination of elements of
a lattice basis has a positive probability, then we can indeed 
guarantee the connectivity of every fiber.

When we run a Markov chain over a fiber, the transition probabilities
can be easily adjusted by the standard Metropolis-Hastings procedure.  Hence 
we can use any probability distribution for generating the moves,
as long as every integer combination of elements of
a lattice basis has a positive probability.

Based on the above observations, in this paper we discuss sampling of
contingency tables by using a lattice basis.  We propose  simple
algorithms for generating moves such that every move is generated with
a positive probability by using a lattice basis.  We can apply the
proposed method to models whose Markov basis is not easy to compute
and we show the usefulness of the proposed method through numerical
experiments.

The organization of the this paper is as follows.
In Section \ref{sec:def-MBLB} we give a brief review on a Markov basis
and lattice basis.
In Section \ref{sec:LB} we propose algorithms for generating moves by
using lattice basis and in Section \ref{sec:sim} we 
show the practicality and usefulness of the
proposed method through numerical experiments.

\section{Markov basis and lattice basis}
\label{sec:def-MBLB}
In this section we give a brief review on a Markov basis and a lattice
basis. 
Let $\Bx = \{x(\Bi), \Bi \in \cI\}$ denote a contingency table, where 
$x(\Bi)$ is a cell frequency for a cell $\Bi$ and $\cI$ is the set of
cells. 
When we order the elements of $\Bx$ appropriately, $\Bx$ is considered as an
$\vert \cI \vert$ dimensional column vector. 
Let $\Bt$ denote the vector of the sufficient statistic for a toric model.
In a toric model there exists an integer matrix $A$ satisfying 
\[
A \Bx = \Bt.
\]
$A$ is called a configuration matrix associated with the model. 
The set of contingency tables sharing $\Bt$ is called a fiber and
denoted by $\cF_\Bt$.

Consider a goodness-of-fit test for the model.
When $\Bt$ is fixed, $\Bx$ is distributed exactly as a hypergeometric
distribution over the fiber $\cF_\Bt$.
If we can enumerate the elements of the fiber, it is possible to
evaluate a test statistic based on the exact hypergeometric distribution.
In general the enumeration is infeasible and the evaluation of the distribution of
a test statistic is done by sampling contingency tables.

Let
\[
\kerz A = \ker A \cap \Z^{|{\cal I}|} = \{ \Bz\in \Z^{|{\cal I}|} \mid A \Bz=0\}
\]
denote the integer kernel of $A$.
An element of $\kerz A$ is called a move for the model. 
By adding or subtracting a move $\Bz = \{z(\Bi), \Bi \in\cI\} \in \kerz
A$, a contingency table  
$\Bx$ is transformed to a table in the same fiber $\By = \Bx + \Bz$, as long
as $\By$ does not contain a negative cell. 
A finite set of moves $\cB=\{\Bz_1, \dots, \Bz_M\}$ is called a Markov
basis if for every fiber all the states become mutually accessible by moves
in $\cB$.
Consider an undirected graph $G_{\Bt, \cB}$  whose vertices are the
elements of a fiber $\cF_\Bt$. 
We draw an edge between $\Bx \in \cF_\Bt$ and $\By \in \cF_\Bt$ if there
exists $\Bz\in \cB$ such that $\By=\Bx + \Bz$  or $\By = \Bx - \Bz$. 
$\cB=\{\Bz_1, \dots, \Bz_M\}$ is a Markov basis if and only if 
$G_{\Bt, \cB}$ is connected for all $\Bt$. 
In this way a Markov basis guarantees the connectivity of every fiber. 
Combined with the standard Metropolis-Hastings procedure, 
the connectivity enables us to sample contingency tables from an
irreducible Markov chain whose stationary distribution is the
hypergeometric distribution by Markov chain Monte Carlo (MCMC) method. 
Therefore once a Markov basis is obtained, we can evaluate the distribution
of a test statistic of a conditional test based on the exact distribution.

A move $\Bz$ is written as a difference of its positive part and
negative part as $\Bz = \Bz^+ - \Bz^-$, where $z^+(\Bi)=\max(z(\Bi),0)$ and
$z^-(\Bi)=\max(-z(\Bi),0)$, $\Bi \in {\cal I}$.
Consider a binomial $\Bp^{\Bz^+} - \Bp^{\Bz^-}$ corresponding to $\Bz$, 
where $\Bp^{\Bz^\pm} = \prod_{\Bi \in \cI} p(\Bi)^{z^\pm(\Bi)}$ and $p(\Bi)$
are indeterminates. 
The degree of the binomial $\Bp^{\Bz^+} - \Bp^{\Bz^-}$ is called the 
degree of $\Bz$. 
Let $I_A$ be the toric ideal associated with a configuration $A$. 
Then $\Bp^{\Bz^+} - \Bp^{\Bz^-} \in I_A$ if and only if 
$\Bz =\Bz^+ -\Bz^-$ is a move. 
Algebraically a Markov basis is defined as a generator of the toric
ideal $I_A$. 
A Gr\"obner basis of $I_A$ forms a Markov basis\cite{diaconis-sturmfels}.
A Markov basis or a Gr\"obner basis of models for relatively small
contingency tables can be computed by a computer algebra system such as
4ti2\cite{4ti2}.   
However the computational cost is very high and for even moderate-sized
models it is difficult to compute a Markov basis or Gr\"obner basis in a
practical amount of time.  

Let $d=\dim \ker A = \vert \cI \vert - \rank A$ be the dimension of
linear space spanned by the elements of $\ker A$ in $\R^{|{\cal I}|}$.  
It is a standard fact that the integer lattice $\kerz A$ possesses a lattice
basis
${\cal L}=\{\Bz_1, \dots,\Bz_d\}$, such that 
every $\Bz \in \kerz A$ is a unique integer combination of 
$\Bz_1, \dots, \Bz_d$\cite{schrijver}.  
Given $A$, it is relatively easy to compute such a basis of $\kerz A$ using the
Hermite normal form of $A$.  

Usually a lattice basis contains exactly  $d$ elements. 
In this paper we allow {\it redundancy} of a lattice basis
and call a finite set $\cal L$ of moves a lattice basis if every move is
written by an integral combination of the elements of $\cal L$.
As we mentioned it is relatively easy to compute a lattice
basis for a given $A$.  Also, for many statistical models, where a Markov basis
is hard to obtain, we can more easily identify a lattice basis.  An example of this
is the Lawrence lifting discussed in Section \ref{subsec:lawrence}.
 
Let $S$ be a polynomial ring and let $I_{\cal L}=\langle \Bp^\Bz \mid \Bz\in {\cal L}\rangle$ 
be the ideal generated by a lattice basis $\cal L$. 
The toric ideal $I_A$ is obtained from $I_{\cal L}$ by taking saturation
\cite{sturmfels1996, miller-sturmfels} 
\begin{align*}
 I_A &= \left(
 I_{\cal L} : \langle \prod_{\Bi \in \cI} x(\Bi) \rangle 
 \right)\\
 & := 
 \left\{ 
 y \in S \mid (\prod_{\Bi \in \cI} x(\Bi))^m y \in I_{\cal L}
 \text{ for some } 
 m > 0 
 \right\}.
\end{align*}
Intuitively this fact shows that 
when the frequency of each cell is sufficiently 
large, the fiber is connected by the lattice basis $\cal L$. 

\section{Sampling contingency tables with a lattice basis}
In this section we propose algorithms to generate a move based on a lattice bases.
We also give  lattice bases for higher Lawrence configurations.

\label{sec:LB}
\subsection{Generating moves by using a lattice basis}
Assume that ${\cal L} = \{\Bz_1,\ldots,\Bz_K\}$, $K \ge d$, is a lattice
basis. 
Then any move $\Bz \in \kerz A$ is expressed as
\[
\Bz = \alpha_1 \Bz_1 + \cdots + \alpha_K \Bz_K, 
\quad \alpha_1,\ldots,\alpha_K \in \mathbb{Z}.
\]
Then we can generate a move $\Bz$ by generating the integer coefficients
$\alpha_1, \dots, \alpha_K$.
In the numerical experiments in the next section we use the following
two methods to generate $\alpha_1,\ldots,\alpha_K$.
Both methods generate all integer combinations of elements of $\cal L$ 
with positive probabilities and hence guarantee the connectivity of all
fibers. 

\begin{algorithm}
 \begin{description}
  \item[Step 1] Generate $\vert \alpha_1 \vert, \ldots, \vert \alpha_K
	     \vert$ from Poisson distribution 
	     with mean $\lambda$, 
	     \[
	     \vert \alpha_k \vert \iid Po(\lambda)
	     \]
	     and exclude the case 
	     $\vert \alpha_1 \vert = \cdots = \vert \alpha_K \vert =
	     0$.
  \item[Step 2] $\alpha_k \leftarrow \vert \alpha_k \vert$ or 
	     $\alpha_k \leftarrow - \vert \alpha_k \vert$ with
	     probability $1/2$ for $k=1,\ldots,K$.
 \end{description}
\end{algorithm}

\smallskip

\begin{algorithm}
 \begin{description}
  \item[Step 1] Generate 
	     $\vert \alpha \vert = \sum_{i=1}^K \vert \alpha_i \vert$
	     from geometric distribution with parameter $p$
	     \[
	     \vert \alpha \vert \sim Geom(p)
	     \]
	     and allocate $\vert \alpha \vert$ to
	     $\alpha_1,\ldots,\alpha_K$ according to  
	     multinomial distribution
	     \[
	      \alpha_1,\ldots,\alpha_K \sim 
	     Mn(\vert \alpha \vert;1/K,\ldots,1/K)
	     \]
  \item[Step 2] $\alpha_k \leftarrow \vert \alpha_k \vert$ or 
	     $\alpha_k \leftarrow - \vert \alpha_k \vert$ with
	     probability $1/2$ for $k=1,\ldots,K$.
 \end{description}
\end{algorithm}

\subsection{A lattice basis for higher Lawrence configuration}
\label{subsec:lawrence}
Consider a configuration matrix of the form
\[
 \Lambda(A) = 
 \begin{pmatrix}
  A & 0\\
  I & I
 \end{pmatrix},
\]
where $I$ is an identity matrix. 
$\Lambda(A)$ is called the Lawrence lifting of $A$ or a Lawrence
configuration\cite{santos-sturmfels-2003}. 
More generally the $r$-th Lawrence configuration is defined by
\vspace{0.3cm}
\begin{equation}
 \label{eq:rth-Lawrence}
\Lambda^{(r)}(A)=
 \begin{pmatrix}
 \smash{\rlap{$\overbrace{\phantom{A \hspace{0.17cm} 0
 \hspace{0.17cm} \cdots 
 \hspace{0.17cm} 0}}^{r-1}$}}A & 0 & \cdots 
 & 0 & 0\\
 0 & A & 0 & \cdots  & 0\\
 \vdots & \ddots & \ddots & \ddots & \vdots\\
 0 & \cdots & 0 & A  & 0\\
 I &  I &  \cdots & I & I 
\end{pmatrix}.
\end{equation}
Many practical statistical models including the no-three-factor
interaction model and the discrete logistic regression model discussed
in the following section have 
Lawrence configurations. 
In general a Markov basis for the Lawrence configuration is 
very difficult to
compute\cite{chen-dinwoodie-dobra-huber2005,hara-takemura-yoshida-logistic}.    
On the other hand it is easy to compute a lattice basis and the proposed
method is available even for such models. 
We can compute a lattice basis of $\Lambda^{(r)}(A)$
by the following propositions. 

\begin{proposition}
 \label{prop:lawrence}
 Let the column vectors of a matrix $B$ form a lattice basis of $A$. 
 Then the column vectors of 
$\begin{pmatrix} B \\ -B \end{pmatrix}$ 
form a
 lattice basis of $\Lambda(A)$.
\end{proposition}
\begin{proof}
 Let $\Bx$ and $\By$ be two contingency tables in the same fiber for 
 $\Lambda(A)$.
 Let $\vert \cI \vert =2n$ be the number of cells.
 Then we note that $n$ is the number of columns of $A$. 
 Write $\Bx = (\Bx_1^\prime, \Bx_2^\prime)^\prime$, 
 where $\Bx_1$ and $\Bx_2$ are $n \times 1$ column vectors and 
${}^\prime$ denotes the transpose. 
 In the same way, write 
 $\By = (\By_1^\prime, \By_2^\prime)^\prime$.
 Let 
 \[
 \Bz = 
 \begin{pmatrix}
  \Bz_1\\
  \Bz_2
 \end{pmatrix}
 = \Bx - \By 
 = 
 \begin{pmatrix}
  \Bx_1 - \By_1\\ 
  \Bx_2 - \By_2
 \end{pmatrix}
 \]
 be a move of $\Lambda(A)$. 
 Since $A \Bz_1=0$, 
 $\Bz_1$ is written by an integer linear combination of $B$ as 
 $\Bz_1 = B \bm{\alpha}$, where 
 $\bm{\alpha}$ is an $l \times 1$ integer vector. 
 $\Bz_0 + \Bz_1 =0$ implies that $\Bz_1 = -B \bm{\alpha}$ and
 therefore
 \[
 \Bz = 
 \begin{pmatrix}
  \Bz_1\\
  \Bz_2
 \end{pmatrix}
 = 
 \begin{pmatrix}
  B\\
  -B
 \end{pmatrix}
 \bm{\alpha}
 \]
 Hence 
$\begin{pmatrix}
  B\\
  -B
\end{pmatrix}$
form a lattice basis of $\Lambda(A)$.
\end{proof}

\begin{proposition}
 \label{prop:rth-Lawrence}
 Let the column vectors of $B$ form a lattice basis of $A$. 
 Then the column vectors of 
\begin{equation}
\label{eq:Br}
 B^{(r)}=
\begin{pmatrix}
 \smash{\rlap{$\overbrace{\phantom{A \hspace{0.6cm} 0
 \hspace{0.6cm} \cdots 
 }}^{r-1}$}}B & 0 & \cdots 
 & 0\\
 0 & B & \ddots & \vdots\\
 \vdots & \ddots & \ddots & 0\\
 0 & \cdots & 0 & B\\
 -B &  -B &  \cdots & -B 
\end{pmatrix}.
\end{equation}
form a lattice basis of higher Lawrence configuration $\Lambda^{(r)}(A)$.  
\end{proposition}
\begin{proof}
We can interpret the $r$-th Lawrence lifting as $r$ slices 
of the original contingency table corresponding to $A$.
The number of the cells for $\Lambda^{(r)}(A)$ is $|{\cal I}|=rn$, where
$n$ is the number of cells (columns) of $A$.
 Let 
 \[
 \Bz = 
 \begin{pmatrix}
  \Bz_1\\ \vdots\\
  \Bz_r
 \end{pmatrix}
 = \Bx - \By 
 = 
 \begin{pmatrix}
  \Bx_1 - \By_1\\  \vdots\\
  \Bx_r - \By_r
 \end{pmatrix}
 \]
 be a move of $\Lambda^{(r)}(A)$.   We can express $\Bz_1 = B \bm{\alpha}_1$.
Then using the $r$-th slice as ``pivots'' we can write
\[
 \Bz =  \begin{pmatrix}
  B\\ 0\\ \vdots \\  0 \\
  -B
 \end{pmatrix} \bm{\alpha}_1 
+ 
\begin{pmatrix}
0 \\ \Bz_2 \\ \vdots \\ \Bz_{r-1} \\ \Bz_r + B\bm{\alpha}_1 
\end{pmatrix}.
\]
Note that the first block of $\Bz$ is now eliminated.  Performing the
same operation recursively to other blocks we are left with 
the $(r-1)$-th slice and $r$-th slice, which is the same as the previous proposition.
\end{proof}

In  this proposition  we only used the last slice as pivots.  More symmetric
lattice basis can be obtained by columns of all pairwise differences of slices, 
for example for $r=3$ 
\[
\begin{pmatrix}
B & B & 0 \\
-B & 0 & B\\
0 & -B & -B
\end{pmatrix}.
\]

The lattice bases in the above propositions may contain redundant
elements.   
However the set of moves including redundant elements are sometimes
preferable for moving  around the fiber. 
In general the computation of a lattice basis of $A$ is easier than
the  computation of a lattice basis of $\Lambda^{(r)}(A)$.
Sometimes we can compute a Markov basis for $A$ even when 
it is difficult to compute a Markov basis of $\Lambda^{(r)}(A)$. 
If a Markov basis for $A$ is known, we can use it as a lattice basis for $A$ and
apply the above propositions for obtaining a lattice basis of $\Lambda^{(r)}(A)$.
In the following numerical experiments 
we compute a lattice basis by using the above propositions.

\section{Numerical experiments}
\label{sec:sim}
In this section we apply the proposed method to the no-three-factor
interaction model and the discrete logistic regression model and 
show the usefulness of the proposed method.

\subsection{No-three-factor interaction model}
No-three-factor interaction model is a model for three-way contingency
tables.
Let $x_{i_1 i_2 i_3}$ and $p_{i_1 i_2 i_3}$ denote a cell frequency
and a cell probability of a cell $\Bi=(i_1,i_2,i_3)$ of a three-way
contingency table, respectively. 
Then the model is described as
\[
\log p_{i_1 i_2 i_3} = \mu_{12}(i_1 i_2) + \mu_{23}(i_2 i_3) +
\mu_{31}(i_3 i_1),
\]
where $\mu_{12}$, $\mu_{23}$ and $\mu_{31}$ are free parameters.
Aoki and Takemura\cite{aoki-takemura-2003anz} discussed the structure of
Markov basis for $3 \times 3 \times K$ table in detail and showed that 
there exists a Markov basis such that the largest degree of moves is 10.
In general, however, the structure of Markov bases for this model is
known to be complicated and the closed form expression of Markov bases
for this model of general tables is not yet obtained at present. 
Even by using 4ti2, it is difficult to compute a Markov basis for
contingency tables larger than $5 \times 5 \times 5$ tables within a
practical amount of time.  

This model has the higher Lawrence configuration in (\ref{eq:rth-Lawrence}) such
that $A$ is a configuration for the two-way complete independence
model. 
The set of basic moves of form
\[
 \begin{array}{|c|c|c|}
  \hline
   & i_1 & i^\prime_1 \\ \hline
 i_2  & 1 & -1 \\ \hline
 i^\prime_2  & -1 & 1 \\ \hline
 \end{array}
\]
is known to be a Markov basis for the two-way complete independence
model.
By using this fact and Proposition \ref{prop:rth-Lawrence}, we can
compute a lattice basis as a set of degree four moves,
\[
\begin{array}{ccc}
 i_3  &\quad & i^\prime_3\\
 \begin{array}{|c|c|c|}
  \hline
   & i_1 & i^\prime_1 \\ \hline
 i_2  & 1 & -1 \\ \hline
 i^\prime_2  & -1 & 1 \\ \hline
 \end{array} & \qquad &
 \begin{array}{|c|c|c|}
  \hline
   & i_1 & i^\prime_1 \\ \hline
 i_2  & -1 & 1 \\ \hline
 i^\prime_2  & 1 & -1 \\ \hline
 \end{array}\\
\end{array}.
\]
In this experiment we compute an exact distribution of 
the log-likelihood ratio (LR) statistic of 
the goodness-of-fit test for no-three-factor interaction model against
the three-way saturated model  
\[
\log p_{i_1 i_2 i_3} = \mu_{123}(i_1 i_2 i_3).
\]
We computed sampling distribution of the LR statistic for 
$I \times I \times I$, $I=3,5,10$ three-way contingency tables.
Then the degrees of freedom of the asymptotic $\chi^2$ distribution of
LR statistic  is $(I-1)^3$.  
We set the sample size as $5I^3$.
For $3 \times 3 \times 3$ tables, the number of burn-in samples and
iterations are $(\text{burn-in},\text{iteration})=(1000,10000)$.
In $3 \times 3 \times 3$ tables, a minimal Markov basis is known
\cite{aoki-takemura-2003anz} and we also compute a sampling 
distribution by a Markov basis. 
In other cases, we set 
$(\text{burn-in},\text{iteration})=(10000,100000)$.

Figure \ref{fig:no3factor1} presents the results for 
$3 \times 3 \times 3$ tables. 
Left, center and right figures are histograms, 
paths and correlograms of the LR statistic, respectively. 
Solid lines in the left figures are asymptotic $\chi^2$ distributions
with degrees of freedom $8$. 
$\alpha_k$ is generated from $Po(\lambda)$, $\lambda=1,10,50$.

We can see from the figures that the proposed methods show comparative
performance to the sampling with a Markov basis.
Although the sampling distribution and the path is somewhat unstable for
$\lambda=50$, in other cases the sampling distributions are similar and
the paths are stable after burn-in period.
Unless we set $\lambda$ as extremely high, the proposed method is robust
against the distribution of $\alpha_k$. 

Figure \ref{fig:no3factor2} presents the results for 
$5 \times 5 \times 5$ and $10 \times 10 \times 10$ tables.
In these cases Markov basis cannot be computed via 4ti2 within a
practical amount of time by an Intel Core 2 Duo 3.0 GHz CPU machine. 
So we compute sampling distributions by the proposed method.
For $5 \times 5 \times 5$ tables, $\alpha_1,\ldots,\alpha_K$ are
generated from $Geom(p)$, $p=0.1,0.5$.
The degrees of freedom of the asymptotic $\chi^2$ distribution is $64$. 
Also in this case we can see that the proposed methods perform well.
The approximation of the sampling distributions to the asymptotic
$\chi^2$ distribution is good and the paths are stable after burn-in
period. 

For $10 \times 10 \times 10$ tables, $\alpha_1,\ldots,\alpha_K$ are
generated from $Po(\lambda)$, $\lambda=10,50$.
The degrees of freedom of the asymptotic $\chi^2$ distribution is $729$. 
In this case the performances of the proposed methods look less stable.
We also compute the cases where the sample sizes are $10I^3$ and
$100I^3$ but the results are similar. 
This is considered to be because the size of fibers of 
$10 \times 10 \times 10$ tables is far larger than those of 
$3 \times 3 \times 3$ or $5 \times 5 \times 5$ tables and it is more
difficult to move around all over a fiber.
Even if we use a Markov basis, the result might not be improved.
Increasing the number of iterations might lead to a better performance. 

Comparing the paths with $\lambda=10$, the path with $\lambda=50$ looks
relatively more stable. 
For larger tables, larger $\lambda$ might be preferable to move around a
fiber.

\begin{figure}[htbp]
\centering
 \includegraphics[scale=0.20]{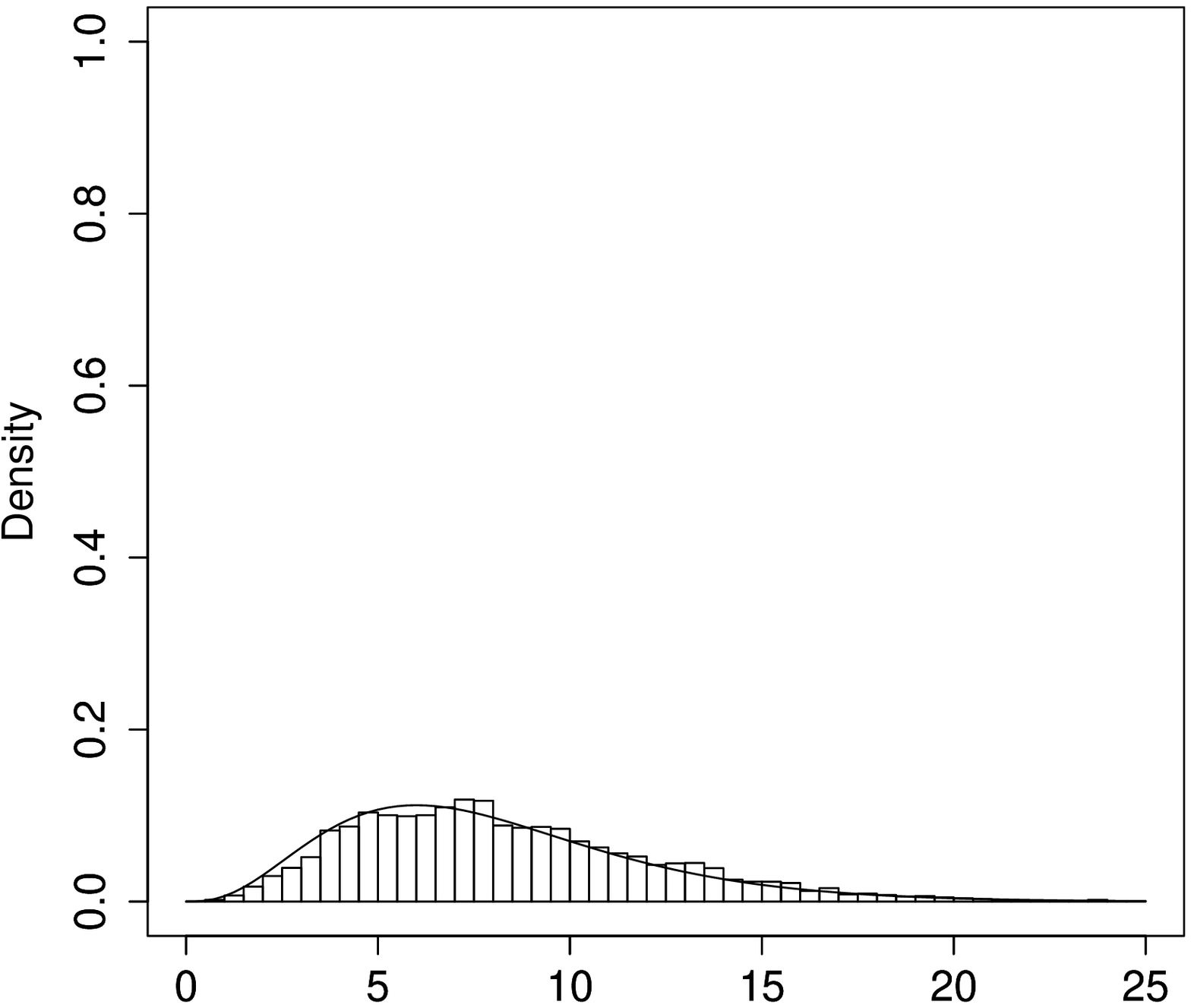}
 \includegraphics[scale=0.20]{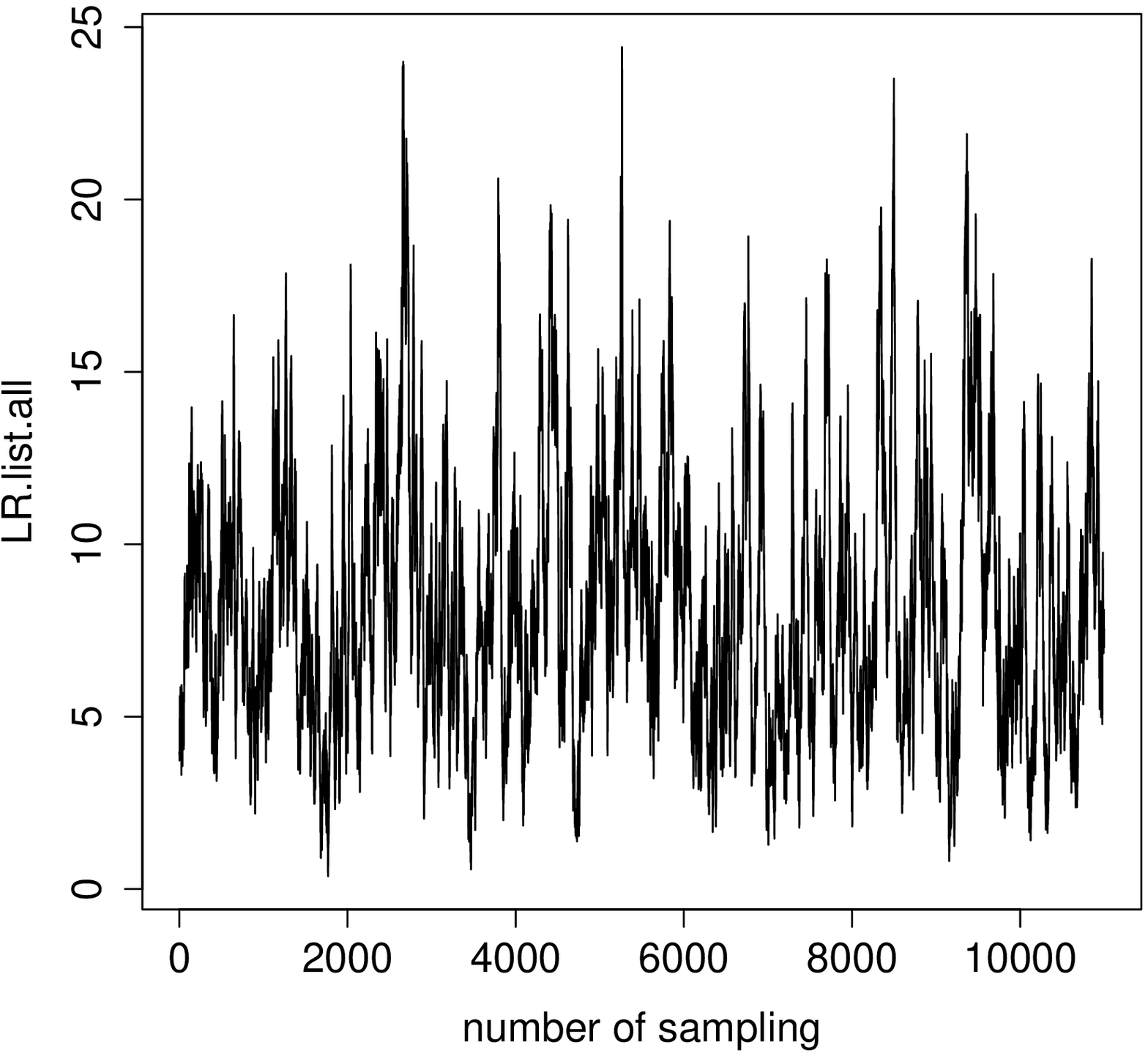}
 \includegraphics[scale=0.20]{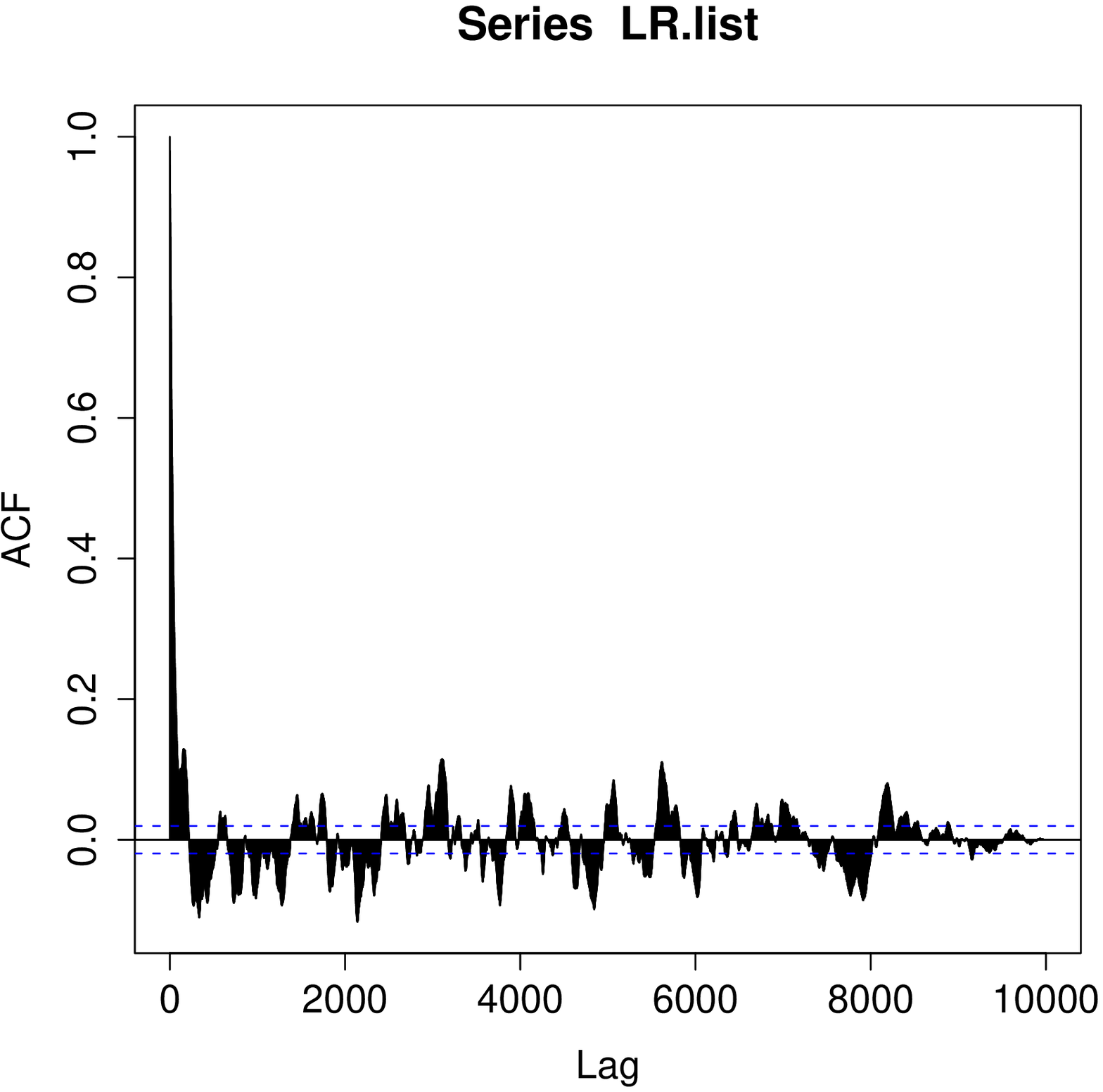}\\
 (a) a Markov basis\\
 \noindent
 \includegraphics[scale=0.20]{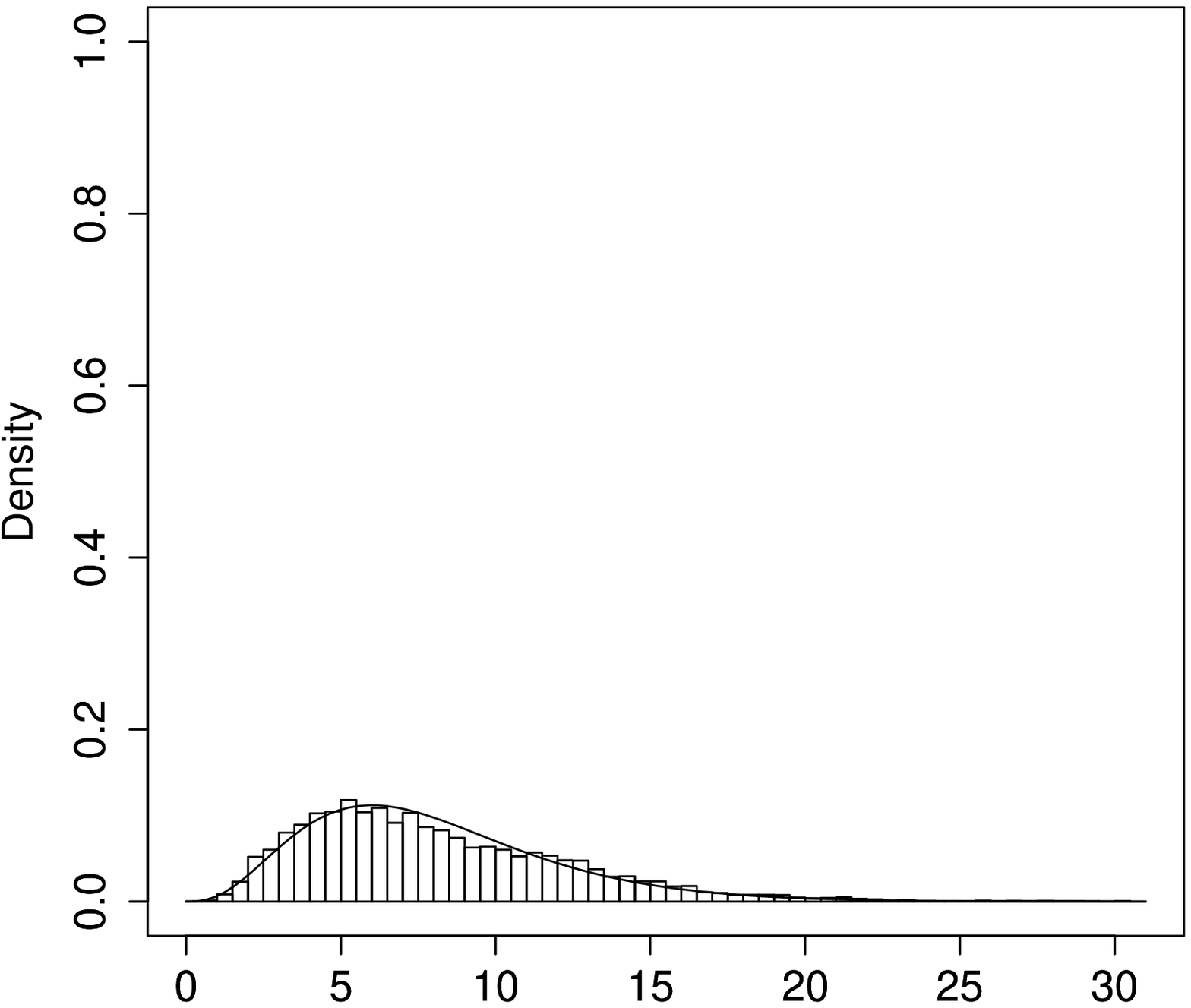}
 \includegraphics[scale=0.20]{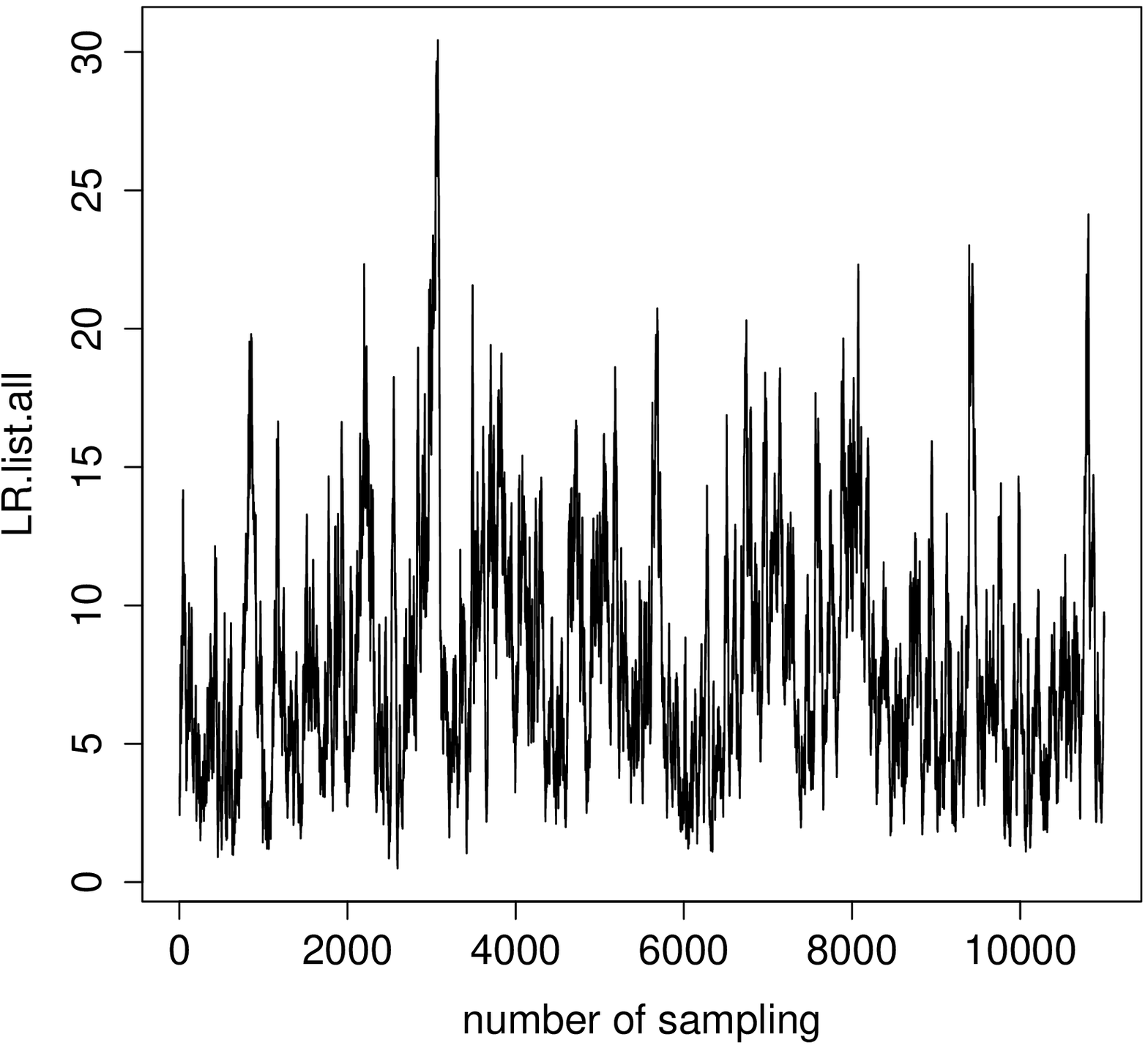}
 \includegraphics[scale=0.20]{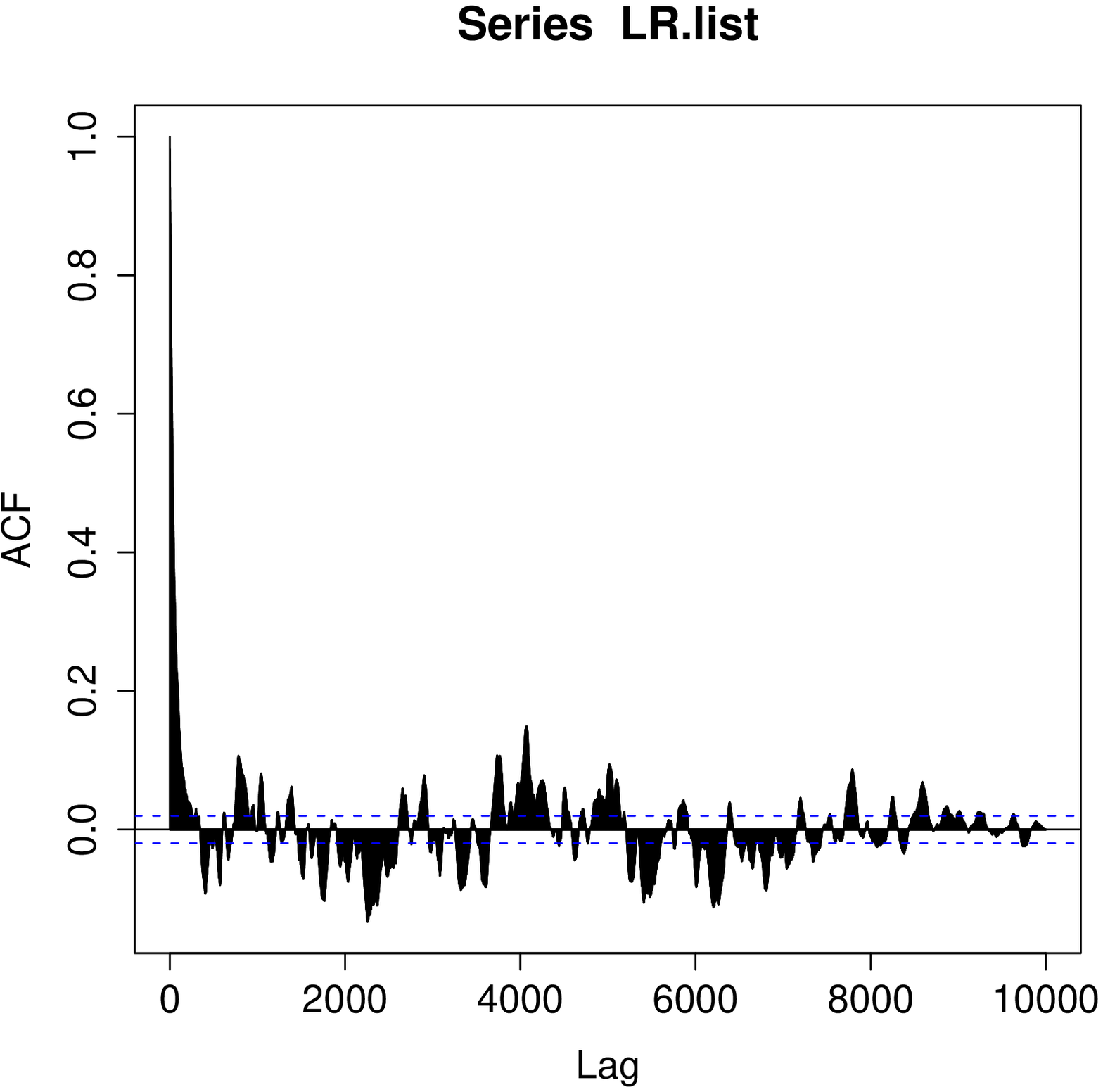}\\
 (b) a lattice basis with $Po(1)$\\
 \noindent
 \includegraphics[scale=0.20]{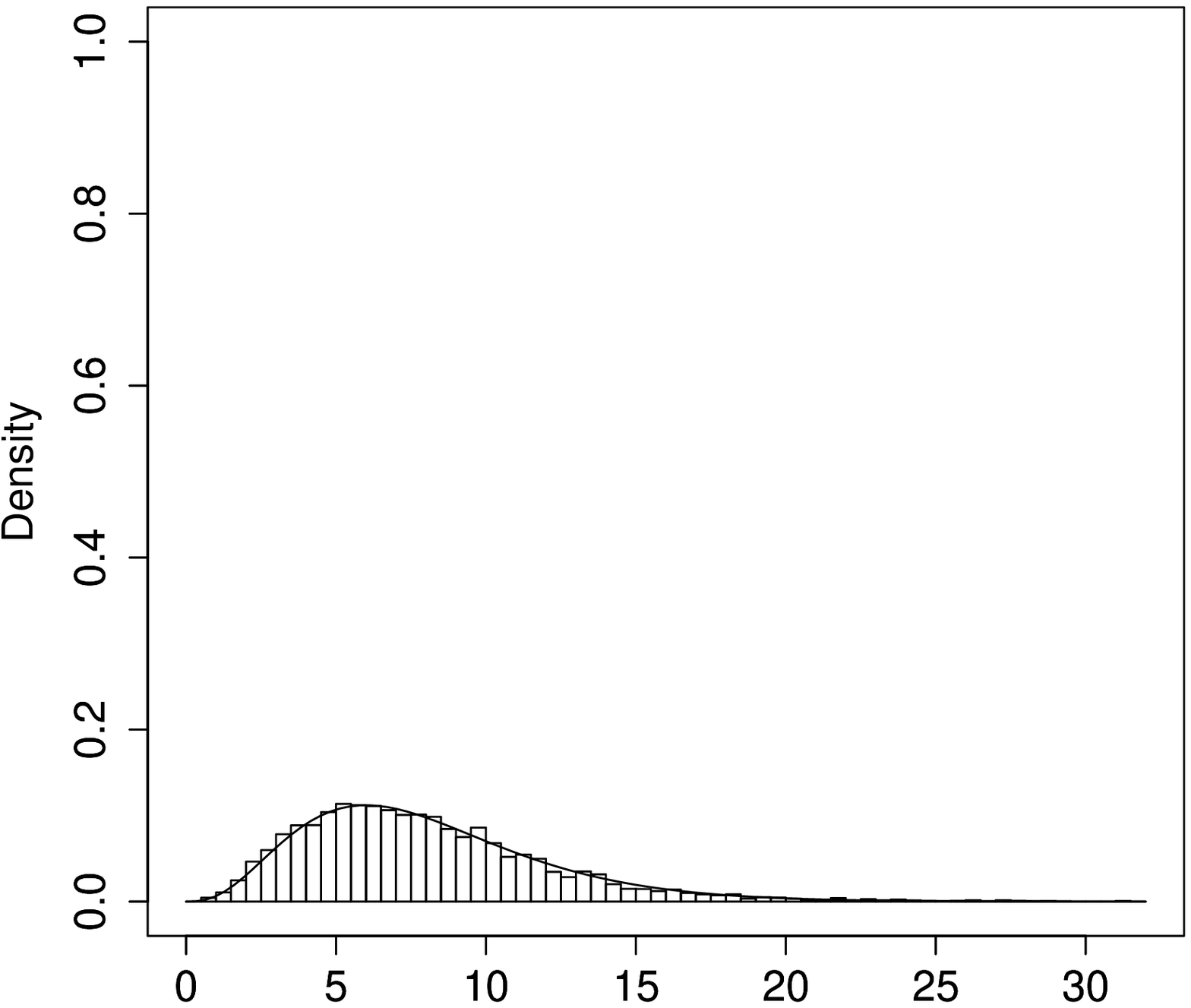}
 \includegraphics[scale=0.20]{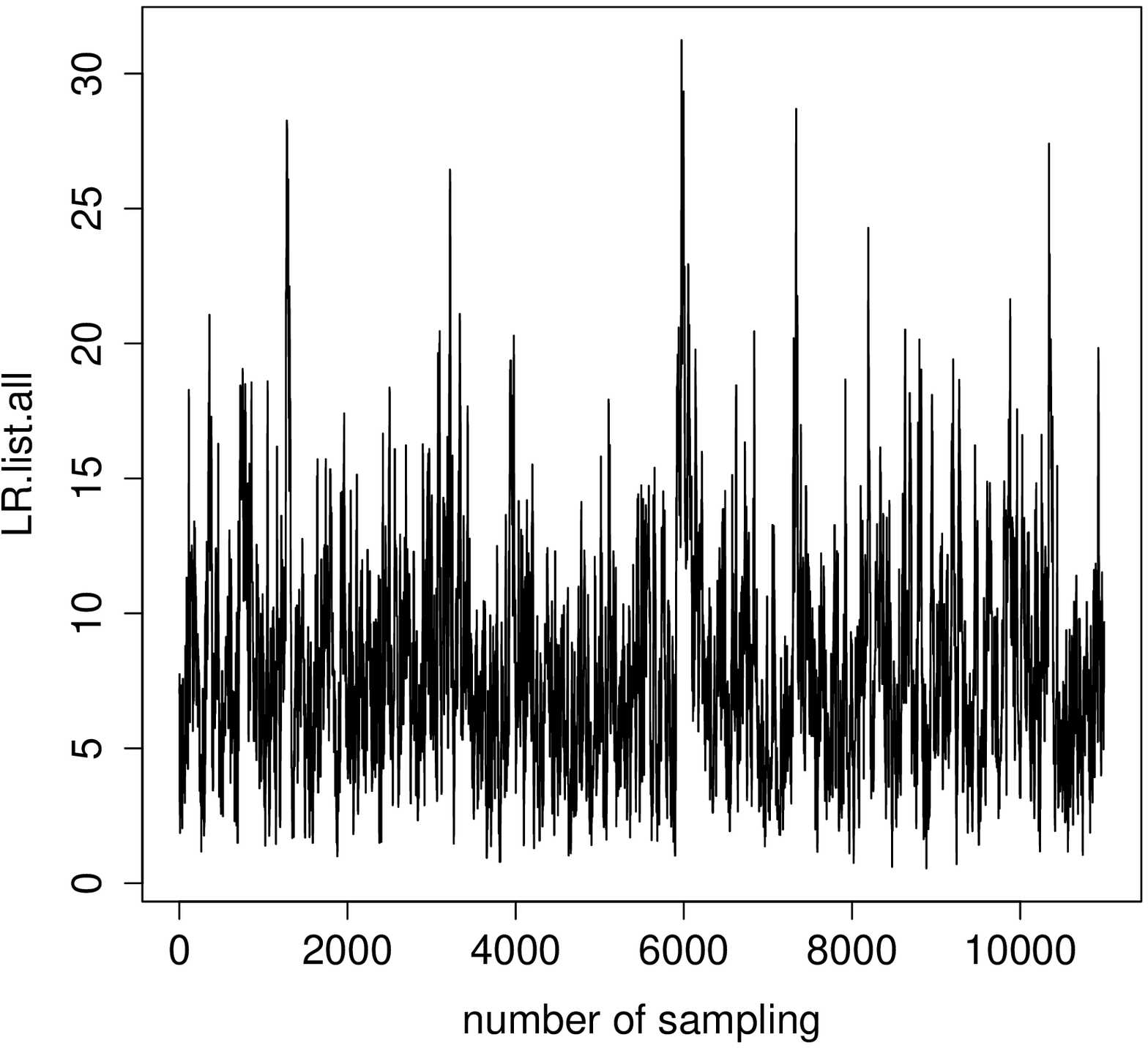}
 \includegraphics[scale=0.20]{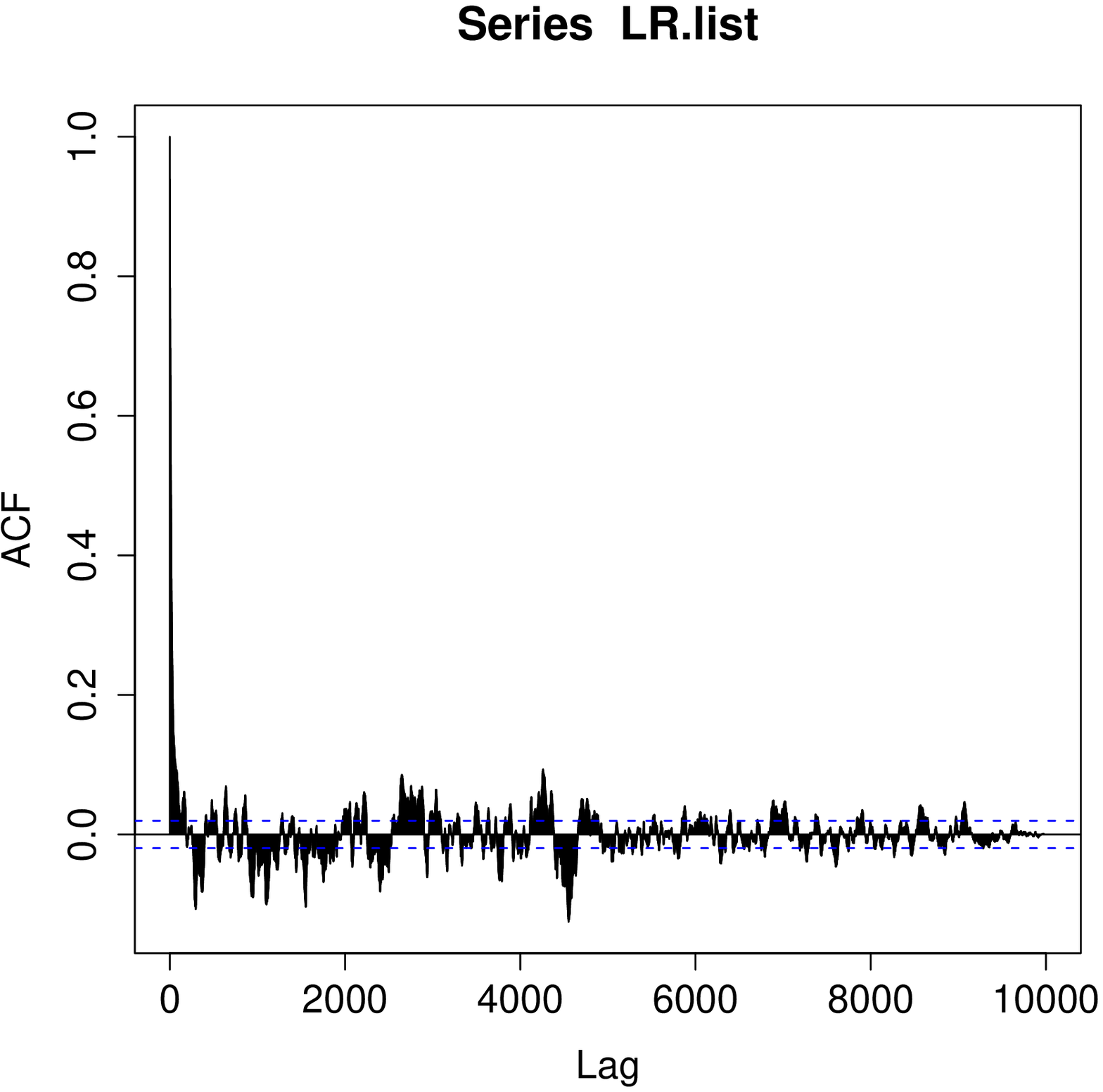}\\
 (c) a lattice basis with $Po(10)$\\
 \noindent
 \includegraphics[scale=0.20]{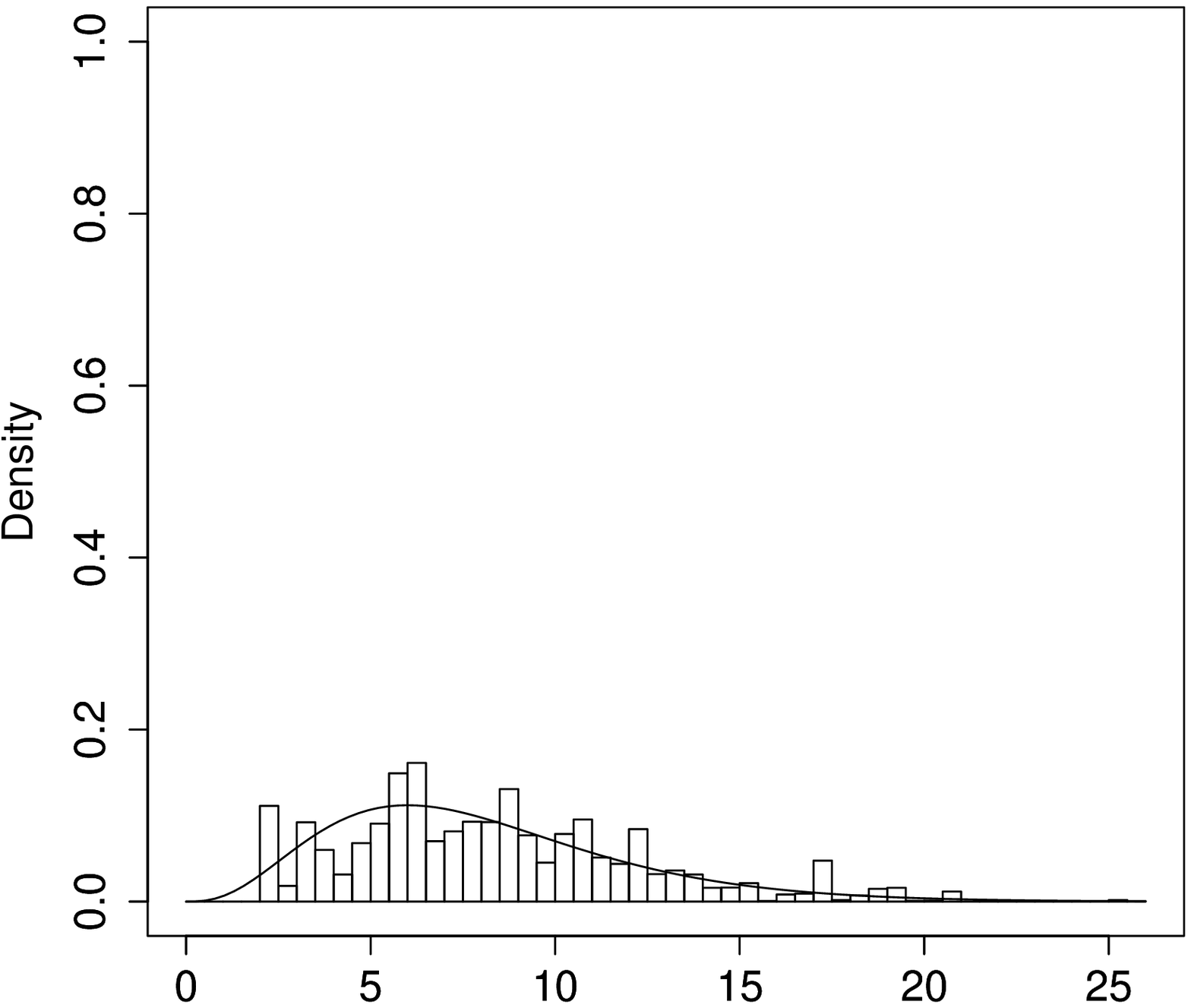}
 \includegraphics[scale=0.20]{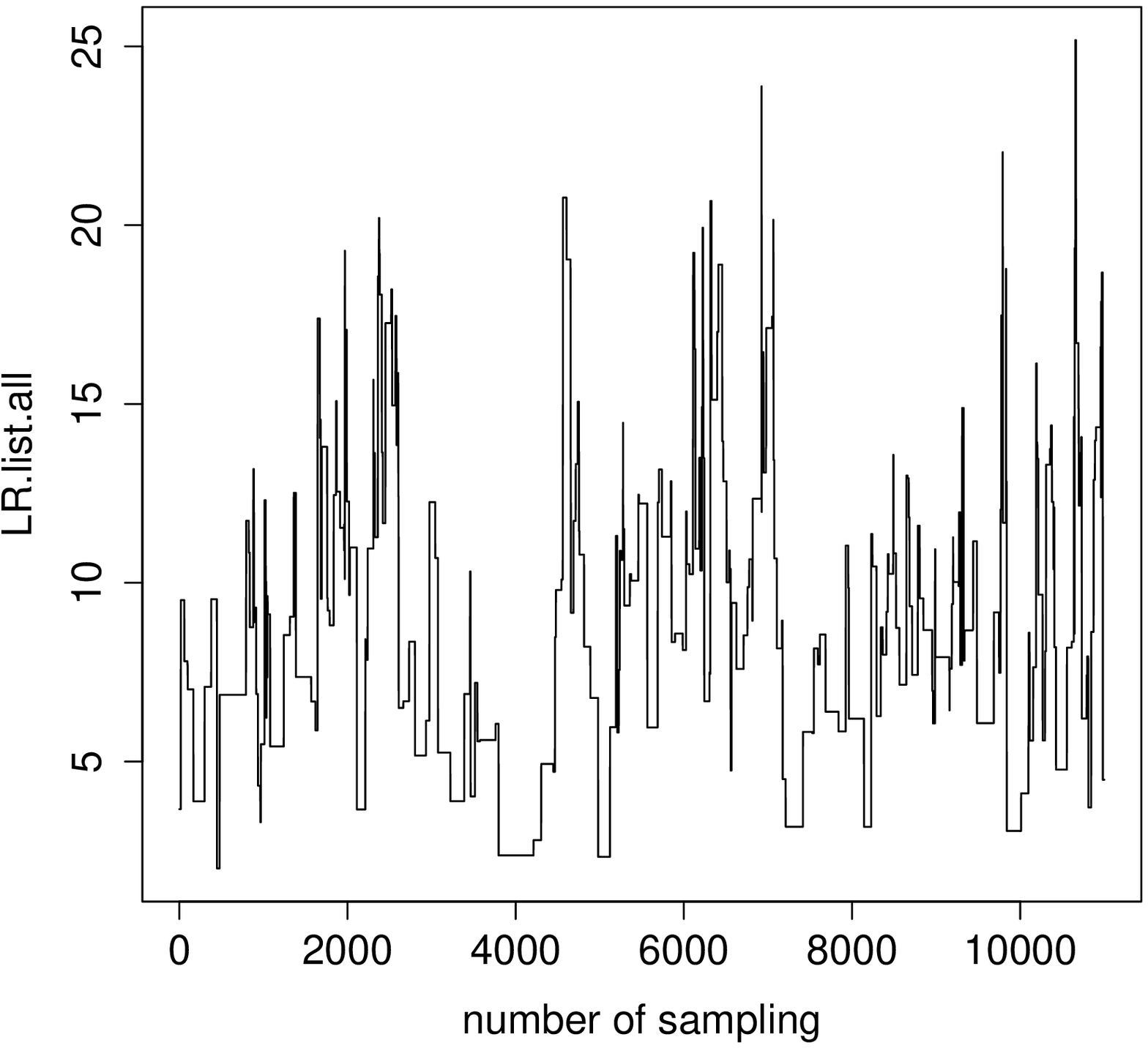}
 \includegraphics[scale=0.20]{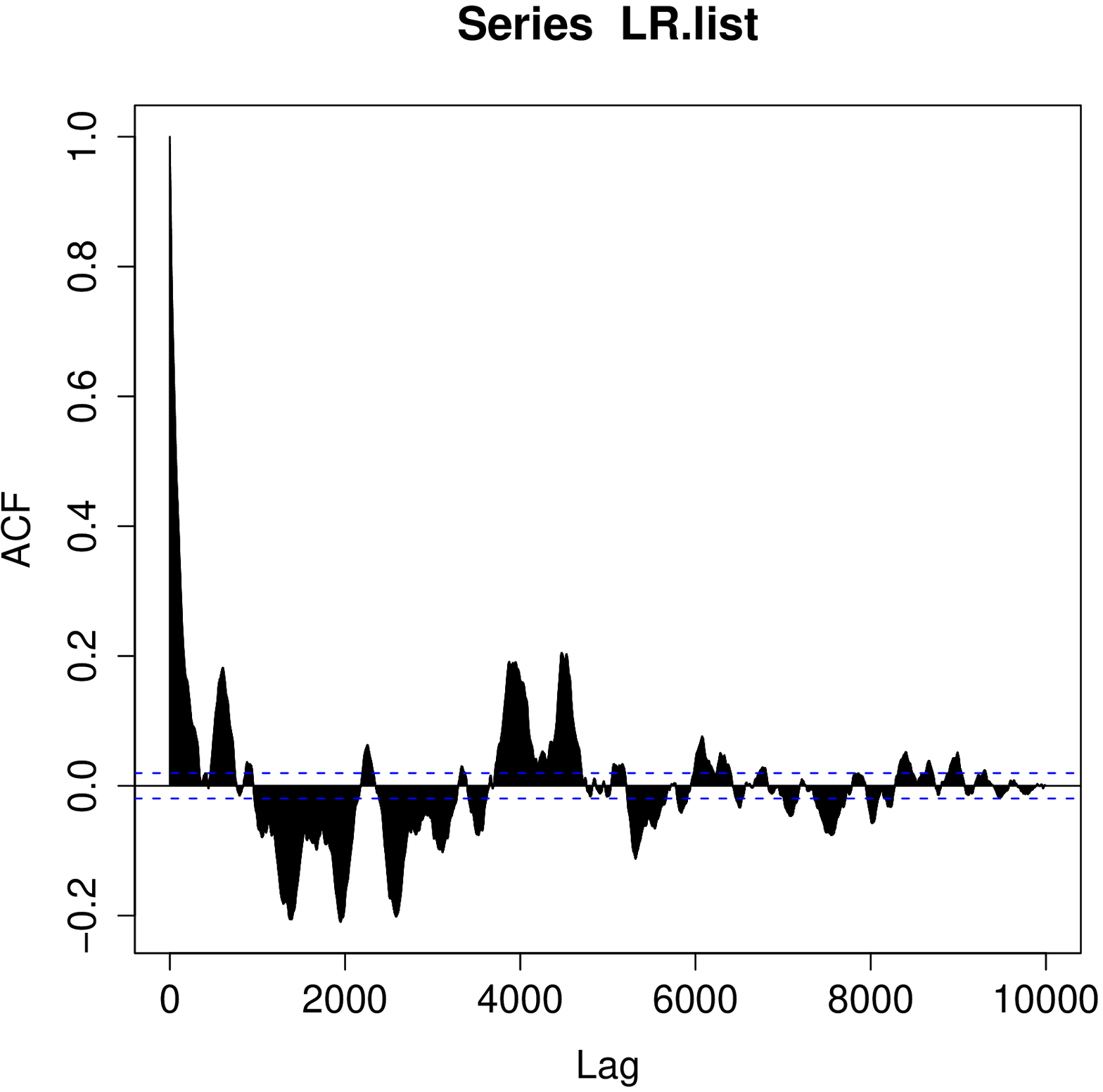}\\
 (d) a lattice basis with $Po(50)$\\
 \caption{Histograms, paths of LR statistic and correlograms for $3
 \times 3 \times 3$ no-three-factor interaction model 
 ((burn in,iteration) $= (1000,10000)$)}
\label{fig:no3factor1}
\end{figure}

\begin{figure}[htbp]
 \centering
 \noindent
 \includegraphics[scale=0.20]{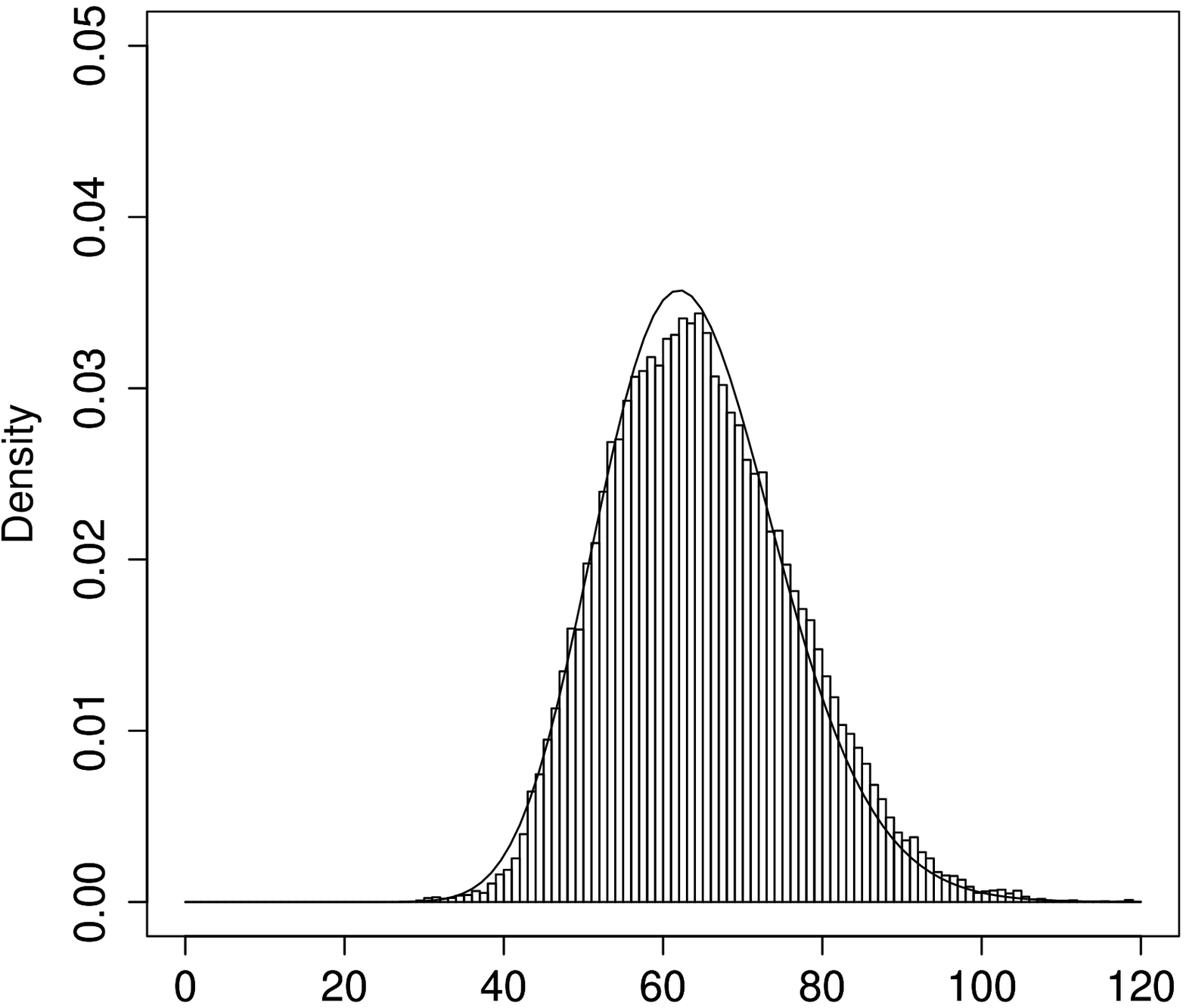}
 \includegraphics[scale=0.20]{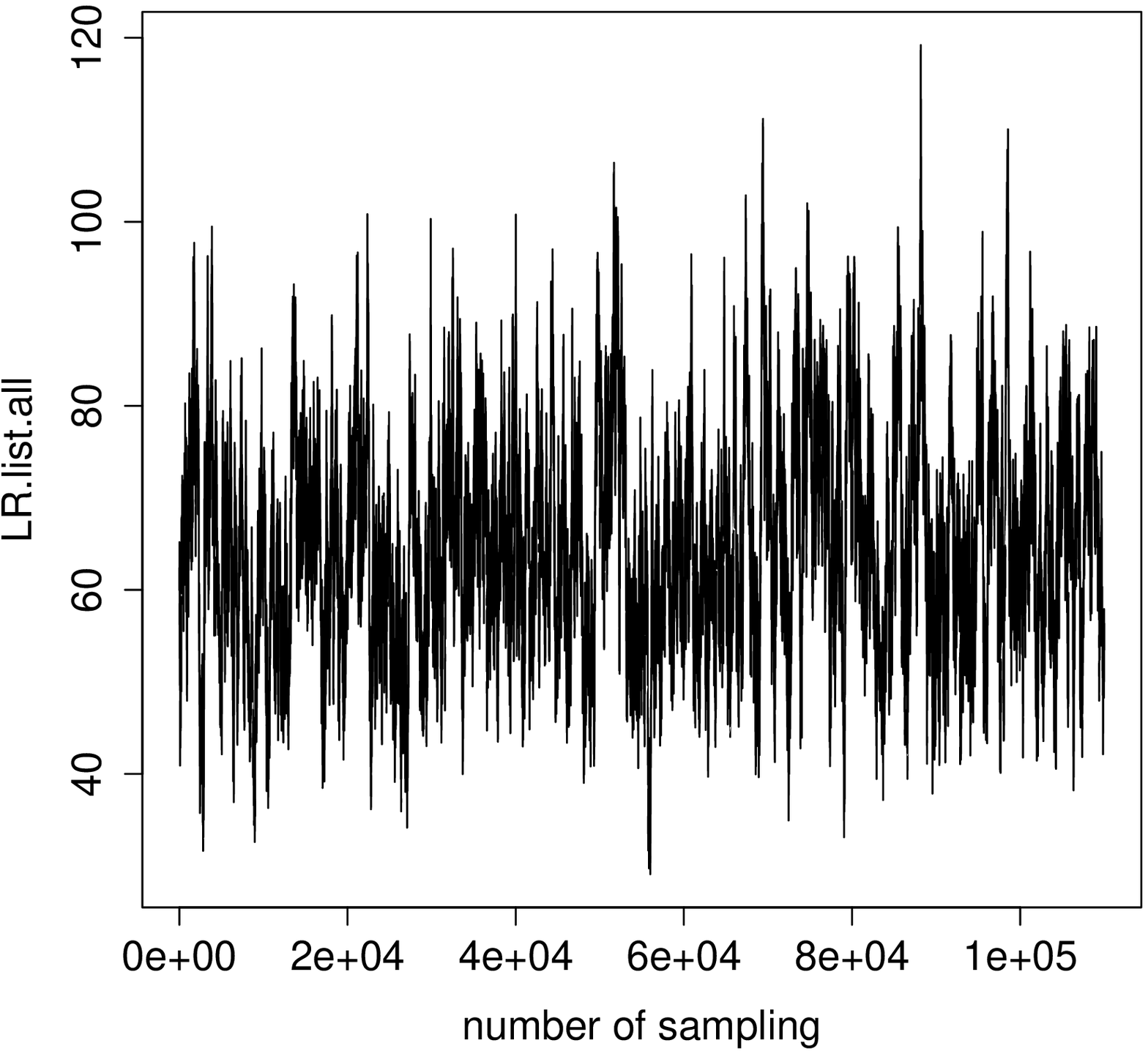}
 \includegraphics[scale=0.20]{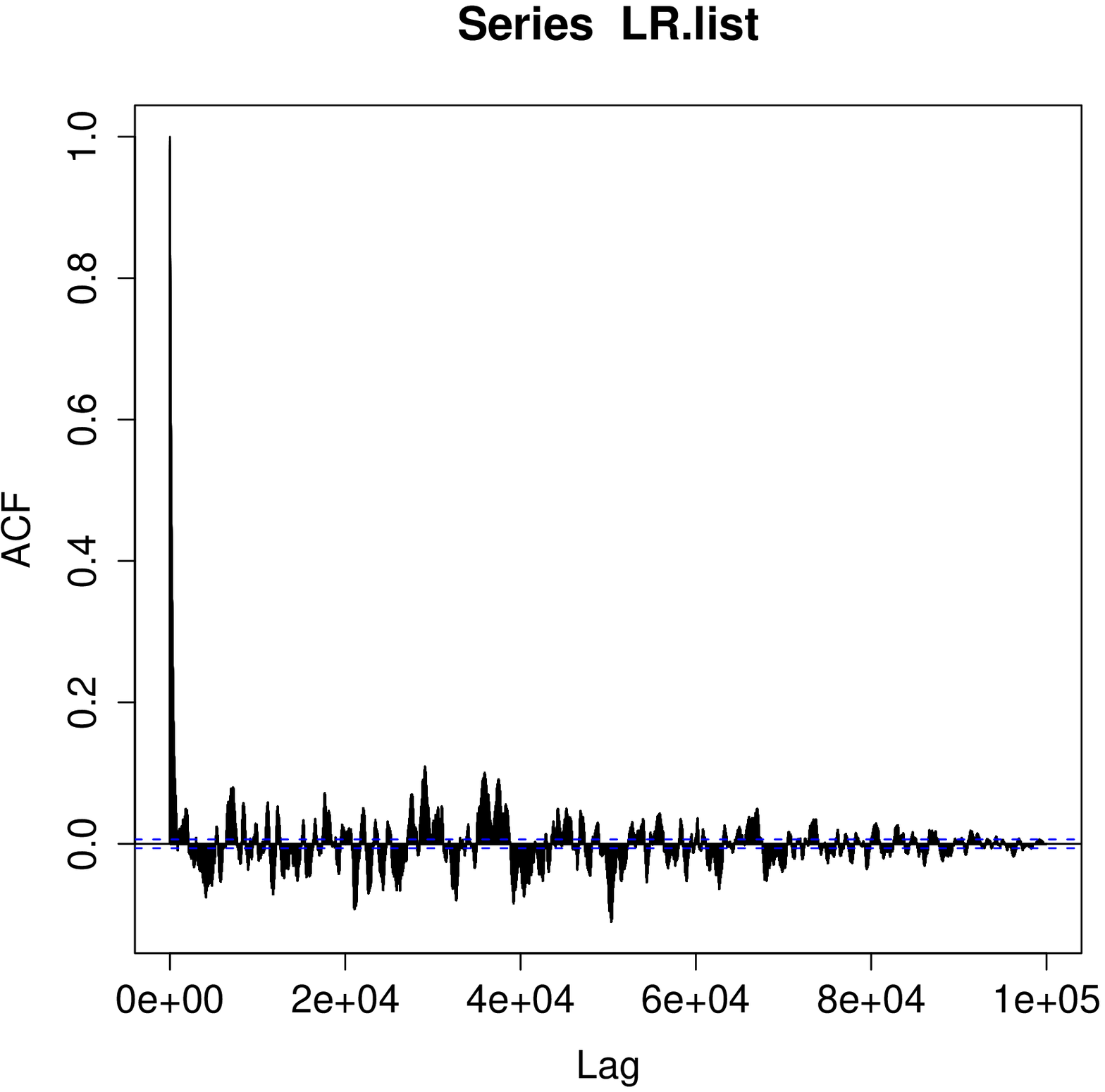}\\
 (a) $5\times 5 \times 5$, a lattice basis with $Geom(0.1)$\\
 \noindent
 \includegraphics[scale=0.20]{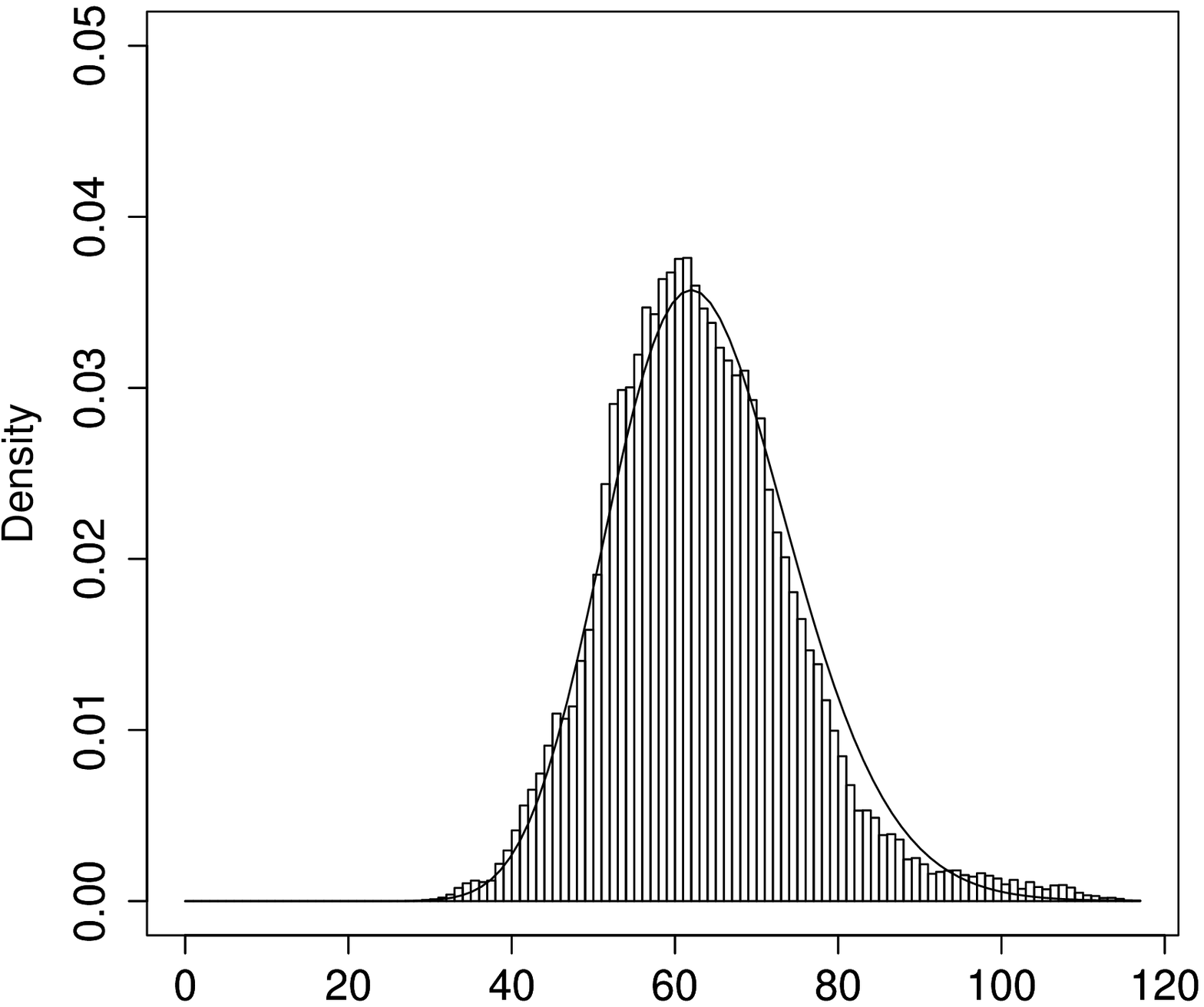}
 \includegraphics[scale=0.20]{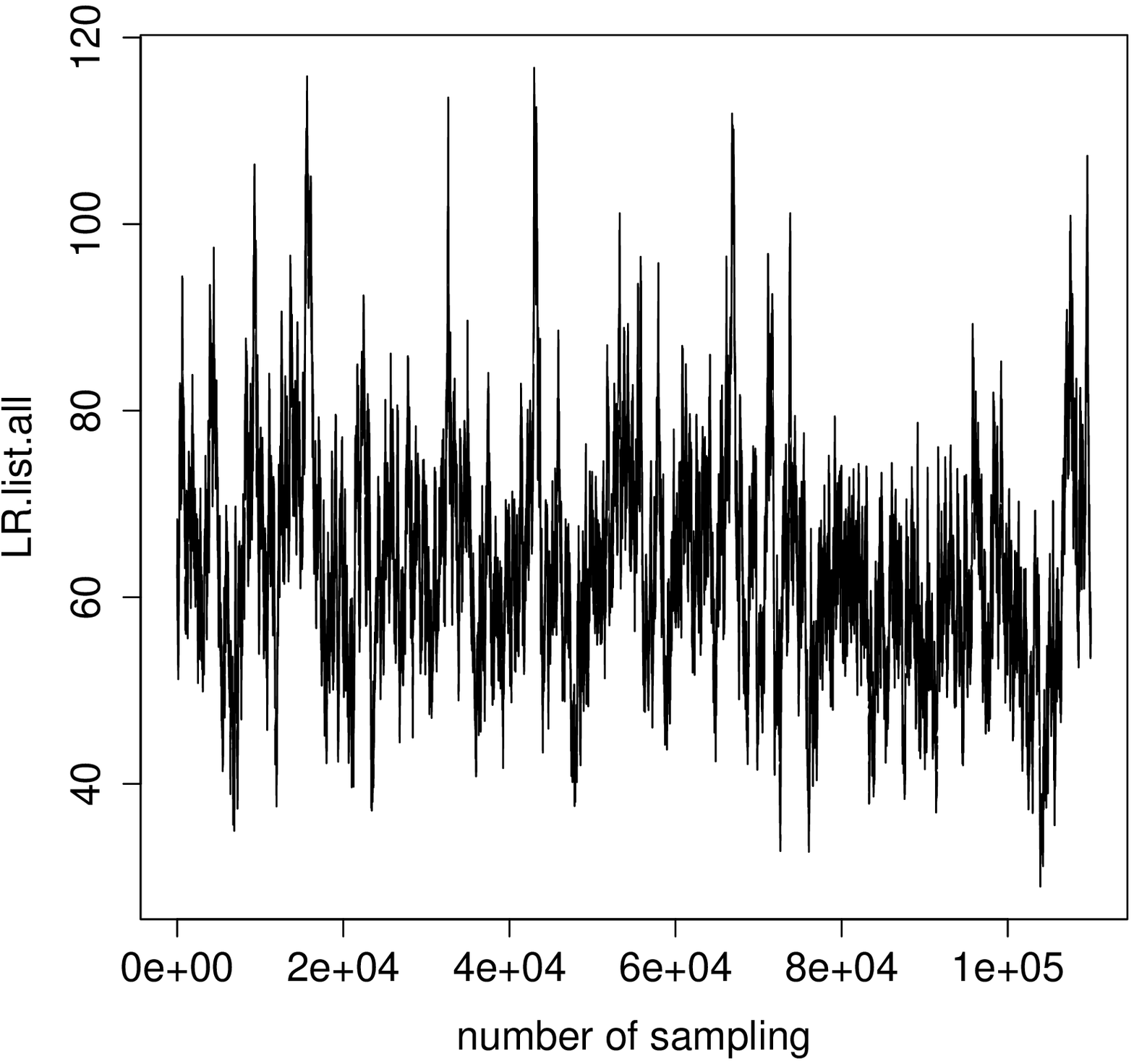}
 \includegraphics[scale=0.20]{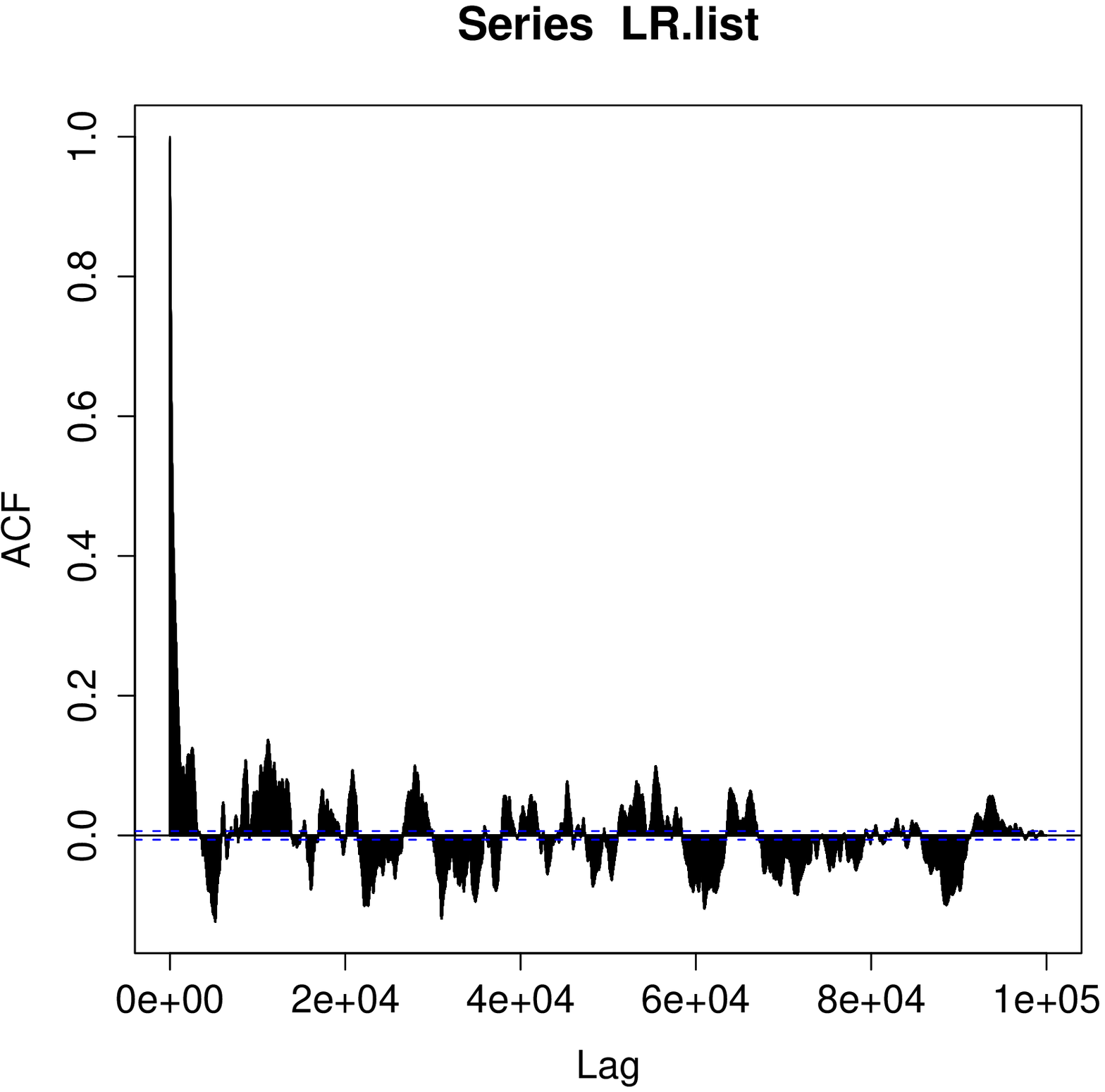}\\
 (b) $5\times 5 \times 5$, a lattice basis with $Geom(0.5)$\\
 \noindent
 \includegraphics[scale=0.20]{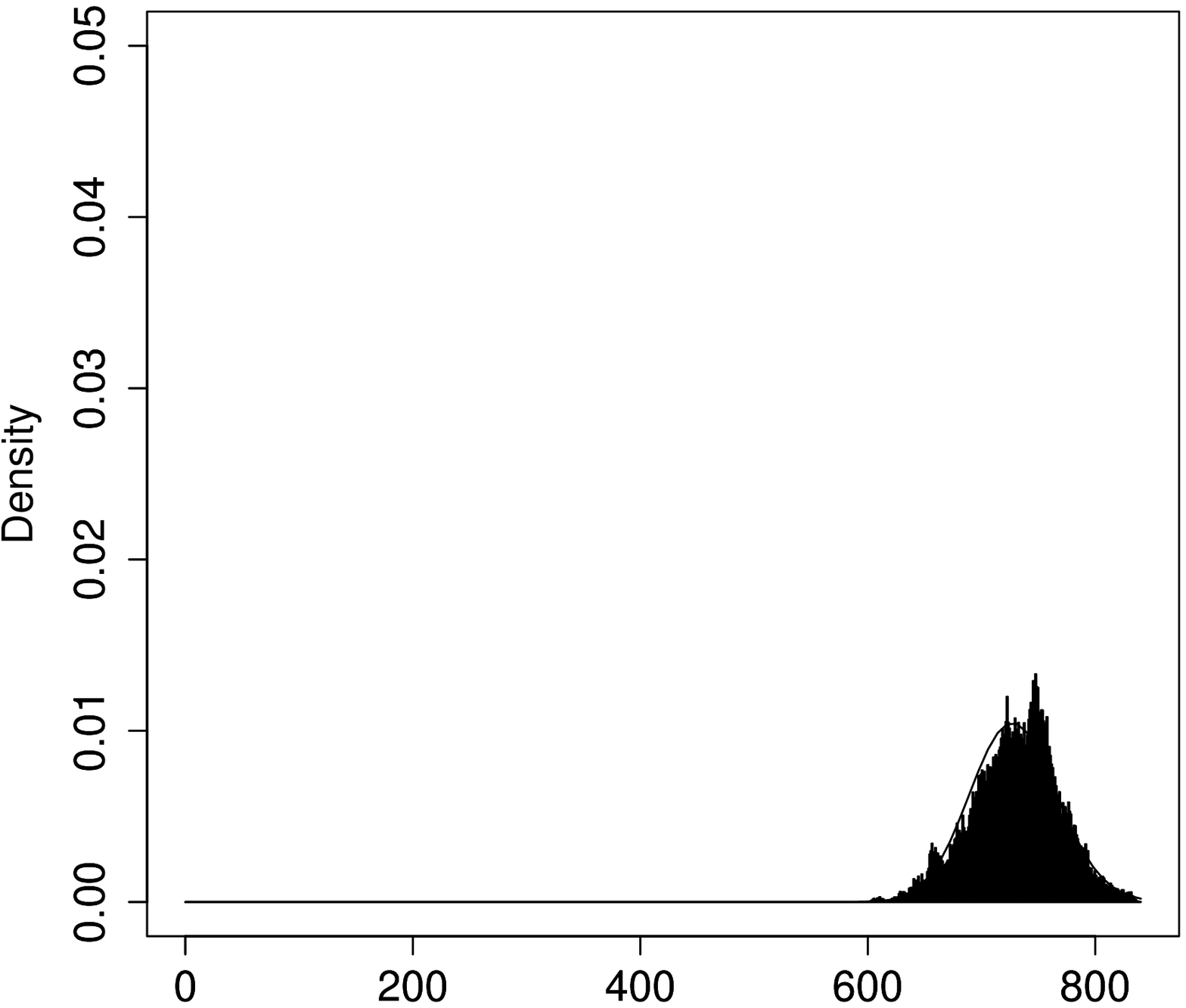}
 \includegraphics[scale=0.20]{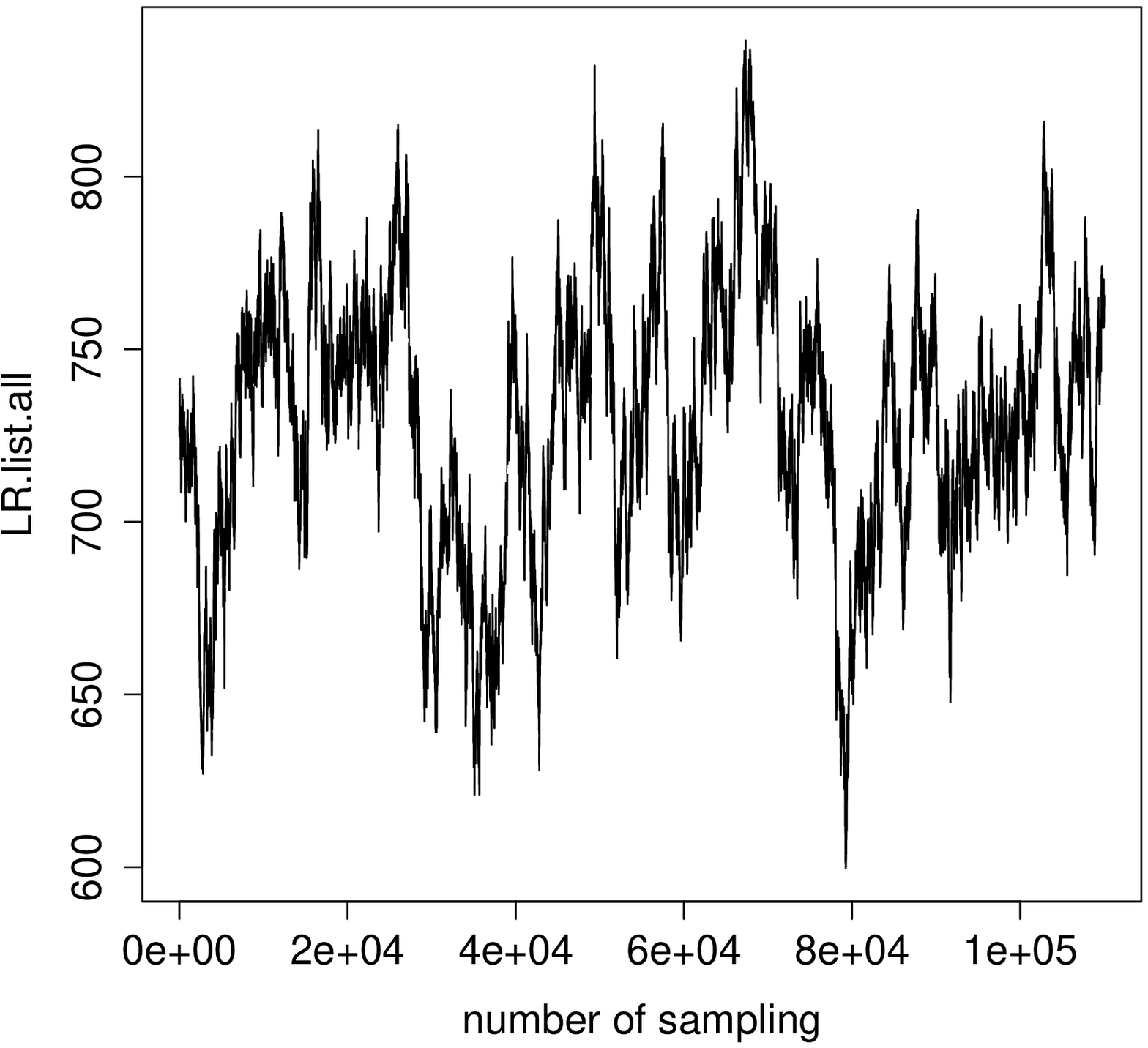}
 \includegraphics[scale=0.20]{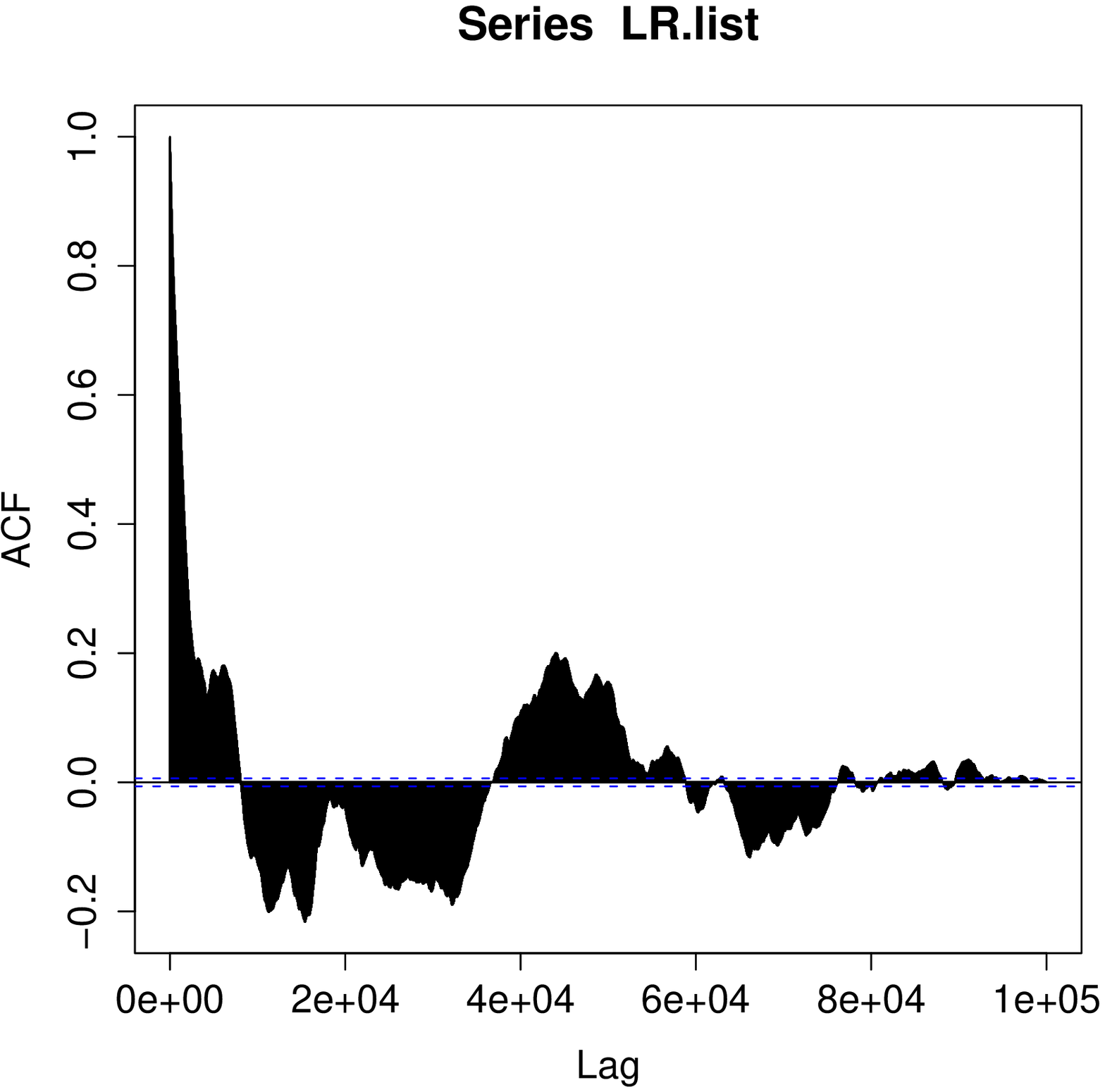}\\
 (c) $10\times 10 \times 10$, a lattice basis with $Po(10)$\\
 \noindent
 \includegraphics[scale=0.20]{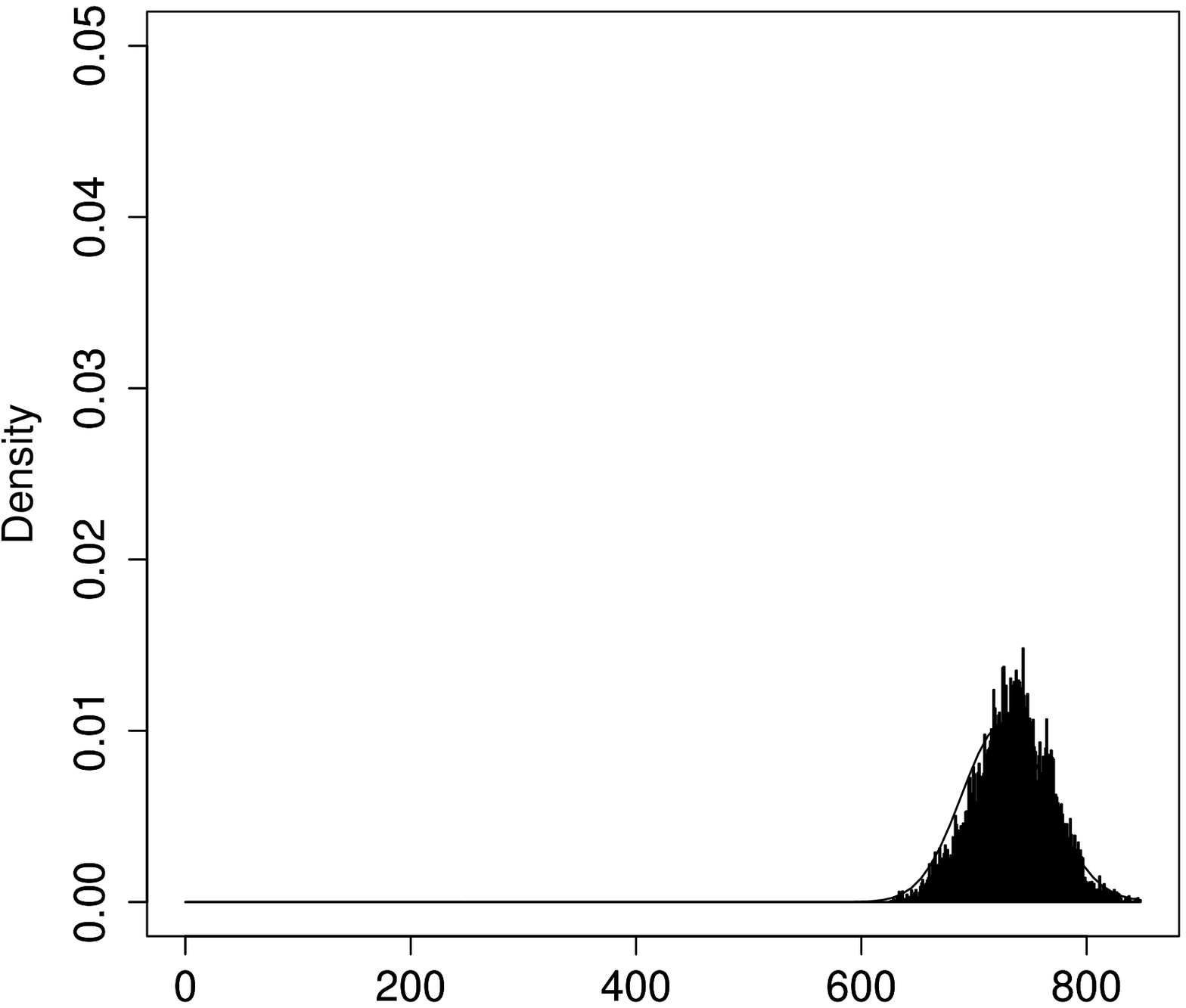}
 \includegraphics[scale=0.20]{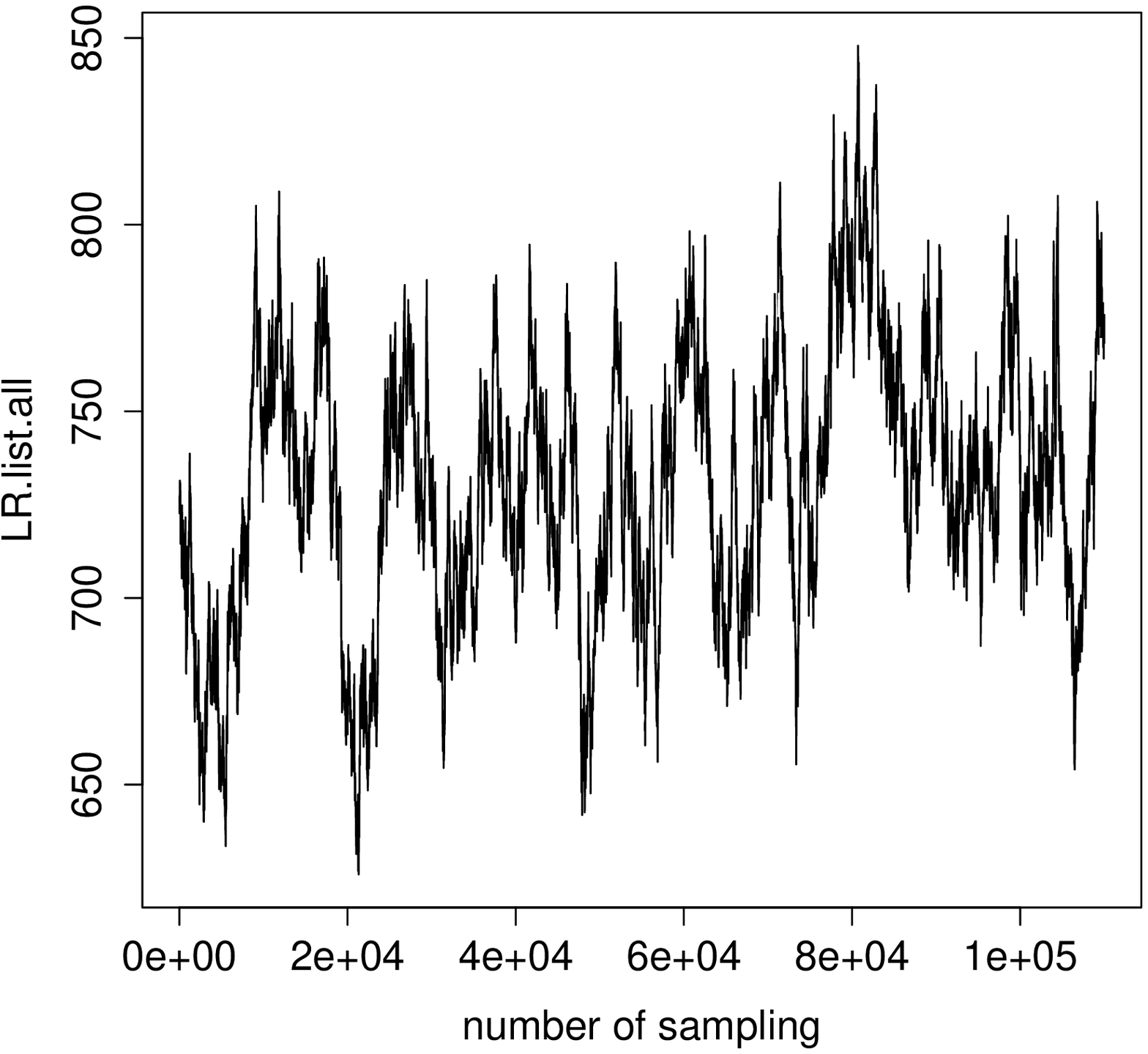}
 \includegraphics[scale=0.20]{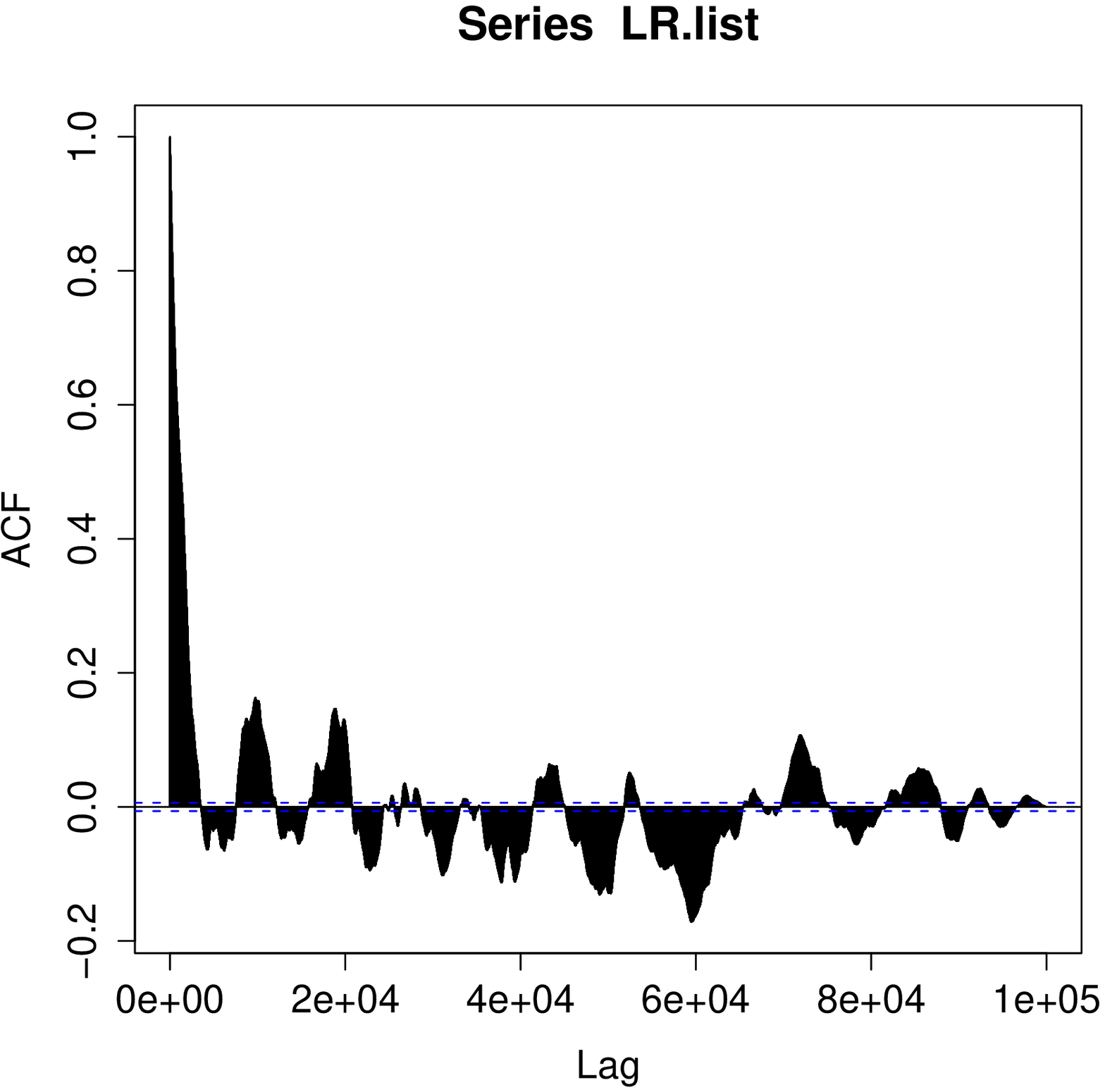}\\
 (d) $10\times 10 \times 10$, a lattice basis with $Po(50)$\\
 \caption{Histograms, paths of LR statistic and correlograms of paths
 for no-three-factor interaction model ((burn in,iteration) $=
 (10000,100000)$)} 
\label{fig:no3factor2}
\end{figure}

\subsection{Discrete logistic regression model}
The logistic regression model with discrete covariates is 
considered as a model for contingency tables.  The model
is defined by the conditional probability for the response variable.
The model with one covariate and the model with two covariates are
described as 
\[
p_{i_1\mid i_2}	=
\left\{
\begin{array}{ll}
 \displaystyle{
 \frac{\exp (\mu_{i_1} + \alpha_{i_1} i_2)}
  {1+ \sum_{i'_1=1}^{I_1-1} \exp (\mu_{i'_1} + \alpha_{i'_1} i_2)},} & \quad 
  i_1 = 1,\ldots,I_1-1,\\
 \displaystyle{
 \frac{1}
  {1+ \sum_{i'_1=1}^{I_1-1} \exp (\mu + \alpha_{i'_1} i_2)},} & \quad 
  i_1 = I_1,
\end{array}
\right.
\]
where $i_2 \in \cI_2$ and 
\[
p_{i_1\mid i_2i_3}	=
\left\{
\begin{array}{ll}
 \displaystyle{
 \frac{\exp (\mu_{i_1} + \alpha_{i_1} i_2 + \beta_{i_1} i_3)}
  {1+ \sum_{i'_1=1}^{I_1-1} 
  \exp (\mu_{i'_1} + \alpha_{i'_1} i_2 + \beta_{i'_1} i_3)},} 
  & \quad i_1 = 1,\ldots,I_1-1,\\
 \displaystyle{
 \frac{1}
  {1+ \sum_{i'_1=1}^{I_1-1} \exp (\mu_{i'_1} + \alpha_{i'_1} i_2
   + \beta_{i'_1} i_3)},} & \quad 
  i_1 = I_1,
\end{array}
\right.
\]
where $(i_2,i_3) \in \cI_2 \times \cI_3$, 
respectively. 
$p_{i_1 \mid i_2}$ and $p_{i_1 \mid i_2 i_3}$ are conditional  probabilities that 
the value of the response variable equals $i_1$ given the covariates $i_2$ and $(i_2,i_3)$,
respectively. 
$\cI_2$ and $\cI_2 \times \cI_3$ are designs for covariates.
The structure of Markov bases for discrete logistic regression model is
also known to be complicated even for the case of binary responses
$I_1=2$ \cite{chen-dinwoodie-dobra-huber2005,hara-takemura-yoshida-logistic}. 
Chen et al. \cite{chen-dinwoodie-dobra-huber2005} and 
Hara et al. \cite{hara-takemura-yoshida-logistic} discussed the model
with one covariate which is discrete and equally spaced and showed that
the set of degree four moves of form 
\[
 \begin{array}{c|rrrr|}
  \multicolumn{1}{c}{} & \multicolumn{1}{c}{i_2} & i_2 + k & i^\prime_2
   - k & \multicolumn{1}{c}{i^\prime_2}\\ \cline{2-5}
 i_1 &  1 & -1 & -1 & 1\\ 
 i^\prime_1 & -1 & 1 & 1 & -1\\ \cline{2-5}
 \end{array}
\]
connects all fibers.
Hara et al\cite{hara-takemura-yoshida-logistic} generalized the argument
to the model with two covariates both of which are equally spaced. 
However it seems to be difficult to generalize these arguments to
the models with more than two covariates or with more
than two responses $I_1 > 2$ at this point.    
A Markov basis connecting all designs has to contain higher degree moves
and the number of moves in a Markov basis is very large.  
Table 1 presents the highest degrees and the numbers of moves in
the minimal Markov bases of binomial logistic regression models with one 
covariate computed by 4ti2. 
Even for models with one covariate, if a covariate has more than 20
levels, it is difficult to compute Markov bases of models via 4ti2
within a practical amount of time by a computer with a 32-bit processor.

The logistic regression model with $r$ responses 
has the $r$-th Lawrence configuration (\ref{eq:rth-Lawrence}) where 
$A$ is a configuration for Poisson regression model. 
The computation of Markov bases of Poisson regression model is
relatively easy.
Therefore a lattice basis can be computed by Proposition
\ref{prop:rth-Lawrence} and we can apply the proposed method to these
models.  

\begin{table}
 \tbl{The highest degrees and the number of moves in a minimal Markov
 basis for binomial logistic regression models with one covariate}
 {\begin{tabular}{lccccccc}\hline 
   & \multicolumn{7}{c}{number of levels of a covariate}\\
   & 10 & 11 & 12 & 13 & 14 & 15 & 16 \\ \hline
  maximum degree & 18 & 20 & 22 & 24 & 26 & 28 & 30 \\ 
  number of moves  & 1830 & 3916 & 8569 & 16968 & 34355 & 66066 &
			      123330 \\ \hline
 \end{tabular}}
\end{table}

In the experiment we considered the goodness-of-fit test of binomial or
trinomial logistic regression model with two covariates against a model
with three covariates 
\[
p_{i_1\mid i_2i_3i_4}	=
\left\{
\begin{array}{ll}
 \displaystyle{
 \frac{\exp (\mu_{i_1} + \alpha_{i_1} i_2 + \beta_{i_1} i_3) +
 \gamma_{i_1} i_4}
  {1+ \sum_{i'_1=1}^{I_1-1} 
  \exp (\mu_{i'_1} + \alpha_{i'_1} i_2 + \beta_{i'_1} i_3 
  + \gamma_{i'_1} i_4)},} 
  & \quad i_1 = 1,\ldots,I_1-1,\\
 \displaystyle{
 \frac{1}
  {1+ \sum_{i'_1=1}^{I_1-1} \exp (\mu_{i'_1} + \alpha_{i'_1} i_2
   + \beta_{i'_1} i_3 + \gamma_{i'_1} i_4)},} & \quad 
  i_1 = I_1,
\end{array}
\right.
\]
where $(i_2,i_3) \in \cI_2 \times \cI_3$, $i_4 \in \cI_4$.
We use the LR statistic as a test statistic.
We assume that $\cI_2 \times \cI_3$ are $4 \times 4$ and $10 \times 10$
checkered designs 
as described in the following figure 
for the $4\times 4$ case,
where only $(i_2,i_3)$ in dotted patterns have positive frequencies. 
\begin{center}
\includegraphics[scale=0.60]{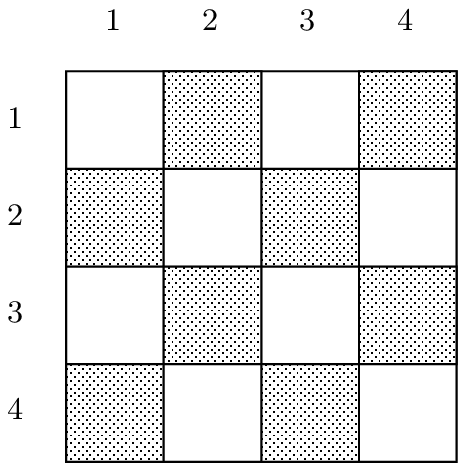}
\end{center}
We also assume that $\cI_4 = \{1,2,3,4,5\}$.
The degrees of freedom of the asymptotic $\chi^2$ distribution of
the LR statistic is $1$.
We set the sample sizes for $4 \times 4$ and $10 \times 10$ designs
are $200$ and $625$, respectively.
We also set $(\text{burn-in},\text{iteration})=(1000,10000)$.


Figures \ref{fig:logit1} and \ref{fig:logit2} present the results for a
binomial and a trinomial logistic regression models with $4 \times 4$ 
checkered pattern, respectively.   
Solid lines in the left figures are asymptotic $\chi^2$ distributions.
$\alpha_k$ is generated from $Po(\lambda)$, $\lambda=1,10,50$.
We can compute Markov bases in these models.
So we also present the results for Markov bases.
We can see from the figures that the proposed methods show comparative
performance to a Markov basis also in these models.
We note that the paths are also stable even for the case where 
$\alpha_1,\ldots,\alpha_K$ are generated from $Po(50)$. 

Figure \ref{fig:logit3} presents the results for $10 \times 10$
checkered pattern.    
In this case Markov bases cannot be computed via 4ti2 by our machine.
$\alpha_k$ is generated from $Geom(p)$, $\lambda=0.1,0.5$.
Also in these cases the results look stable.
These results shows that the proposed method is useful for the logistic
regression models for which that it is difficult to compute a Markov basis.

\begin{figure}[htbp]
\centering
\noindent
\includegraphics[scale=0.20]{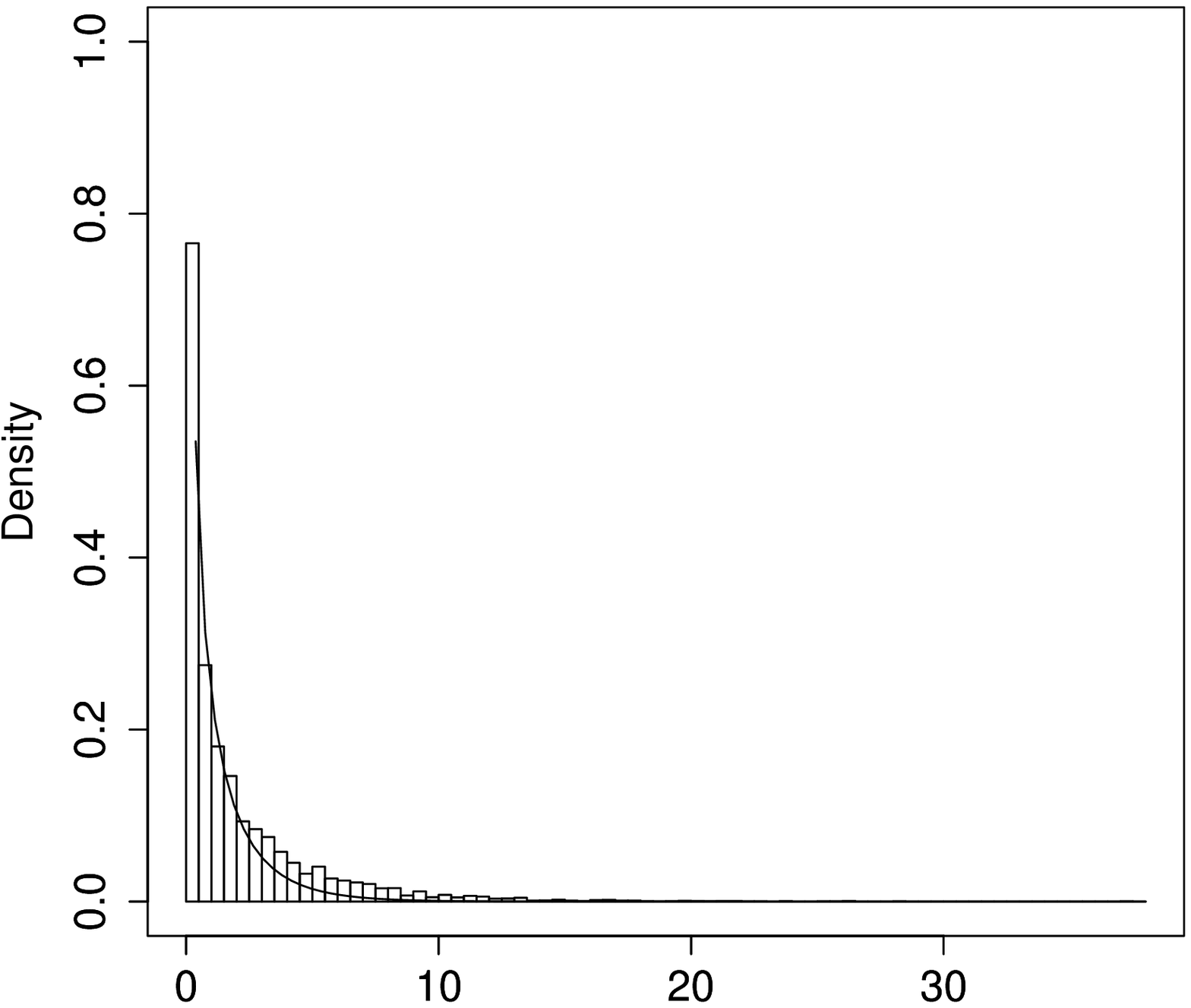}
\includegraphics[scale=0.20]{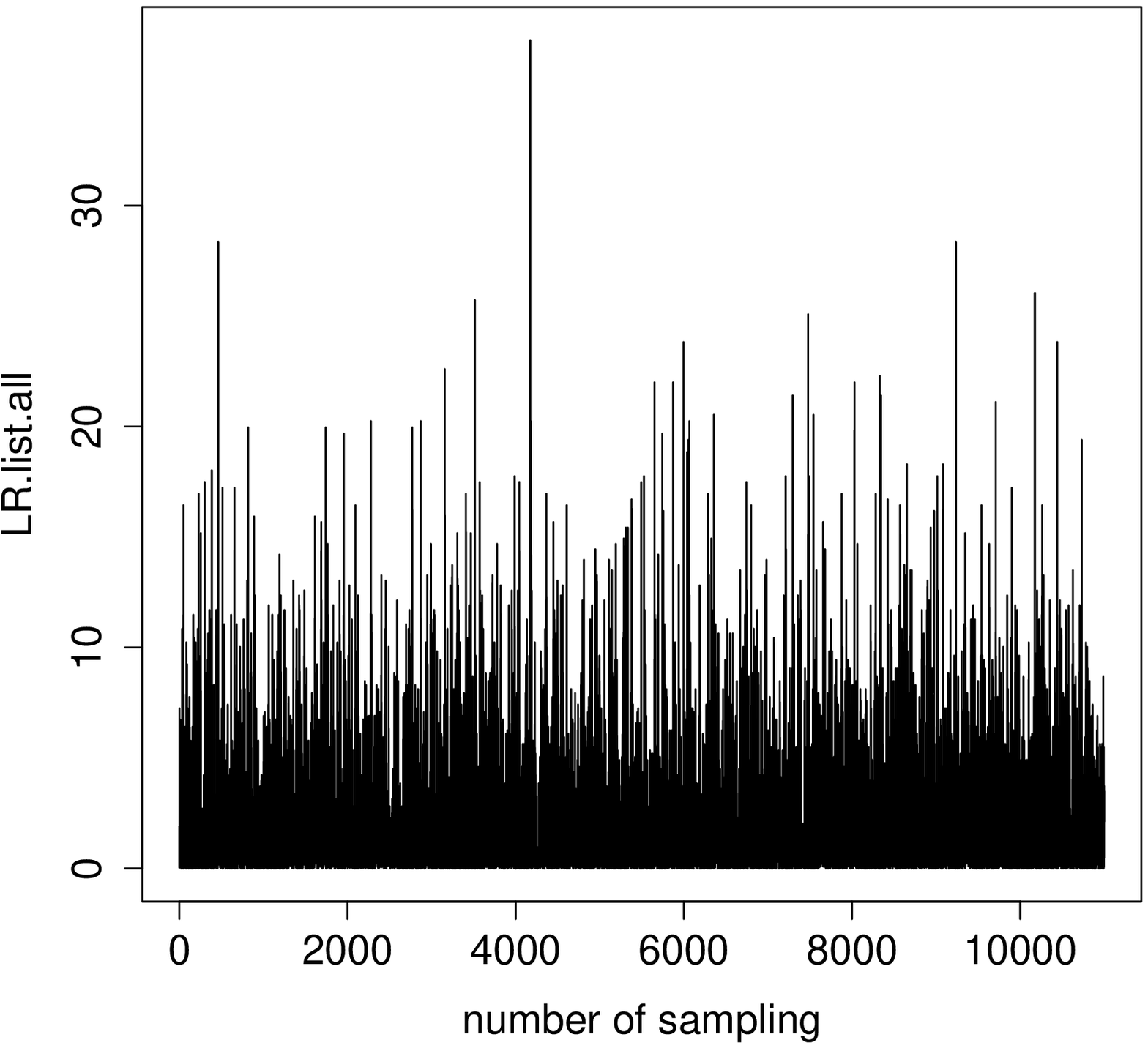}
\includegraphics[scale=0.20]{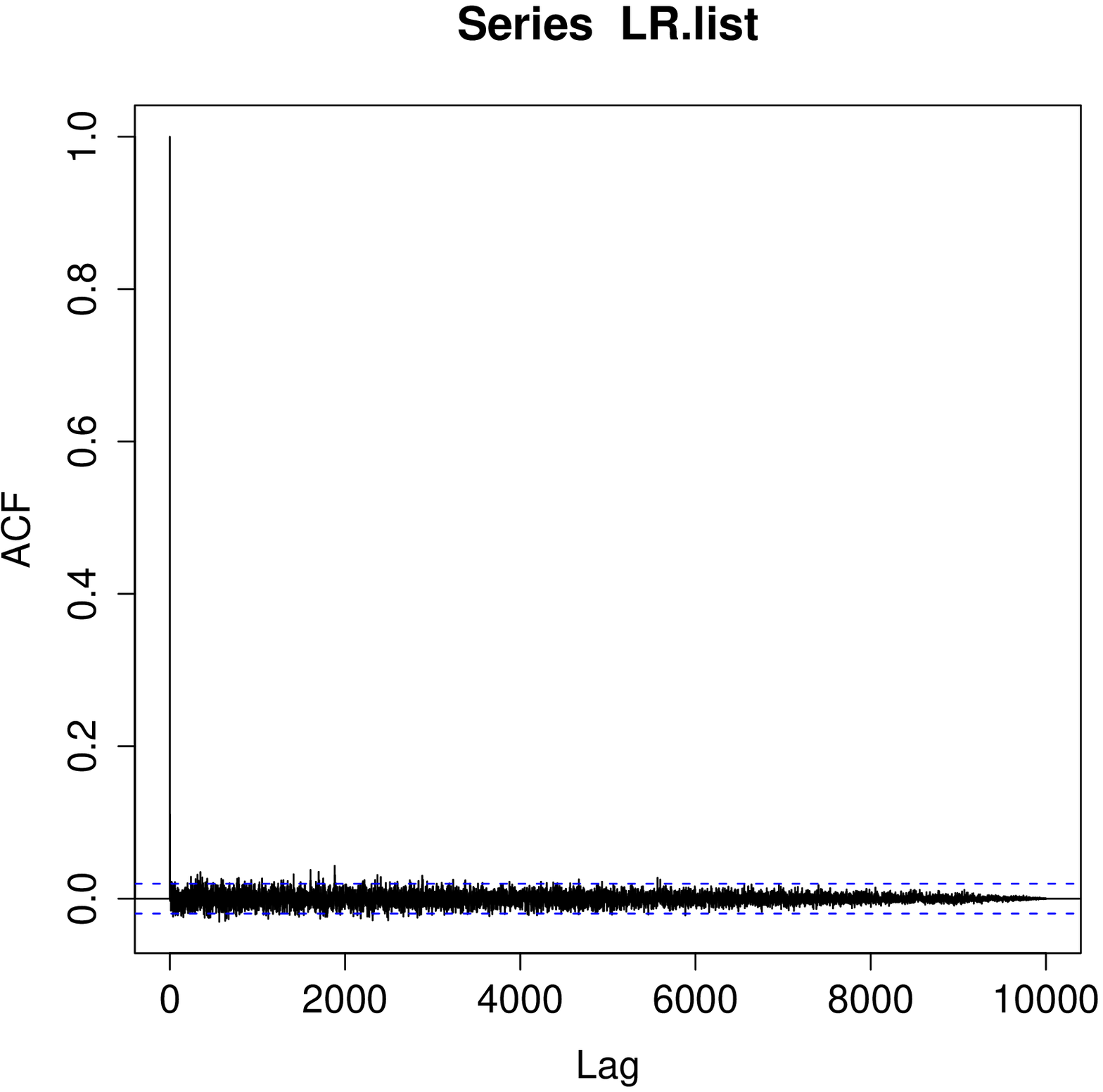}\\
 (a) a Markov basis\\
\includegraphics[scale=0.20]{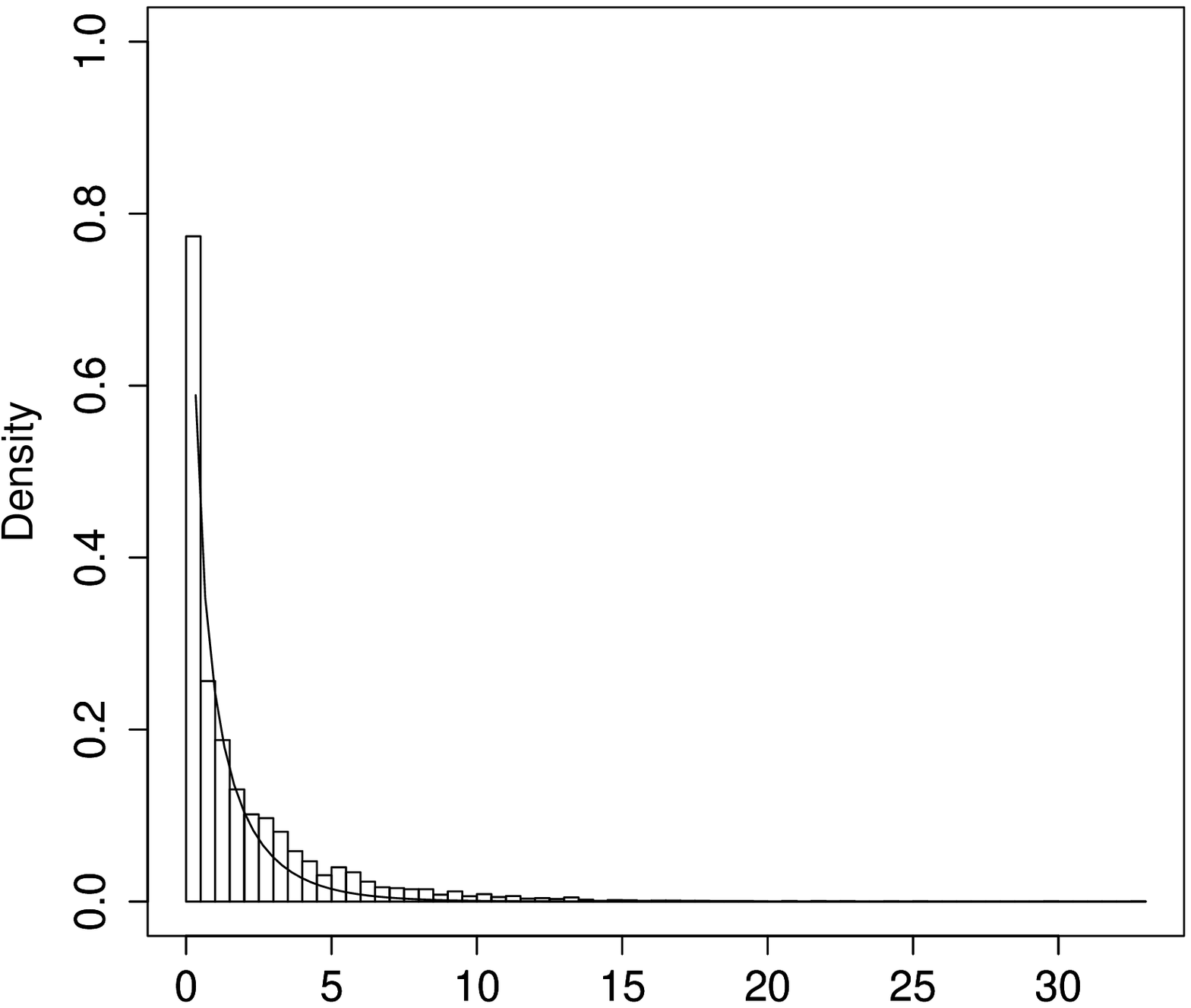}
\includegraphics[scale=0.20]{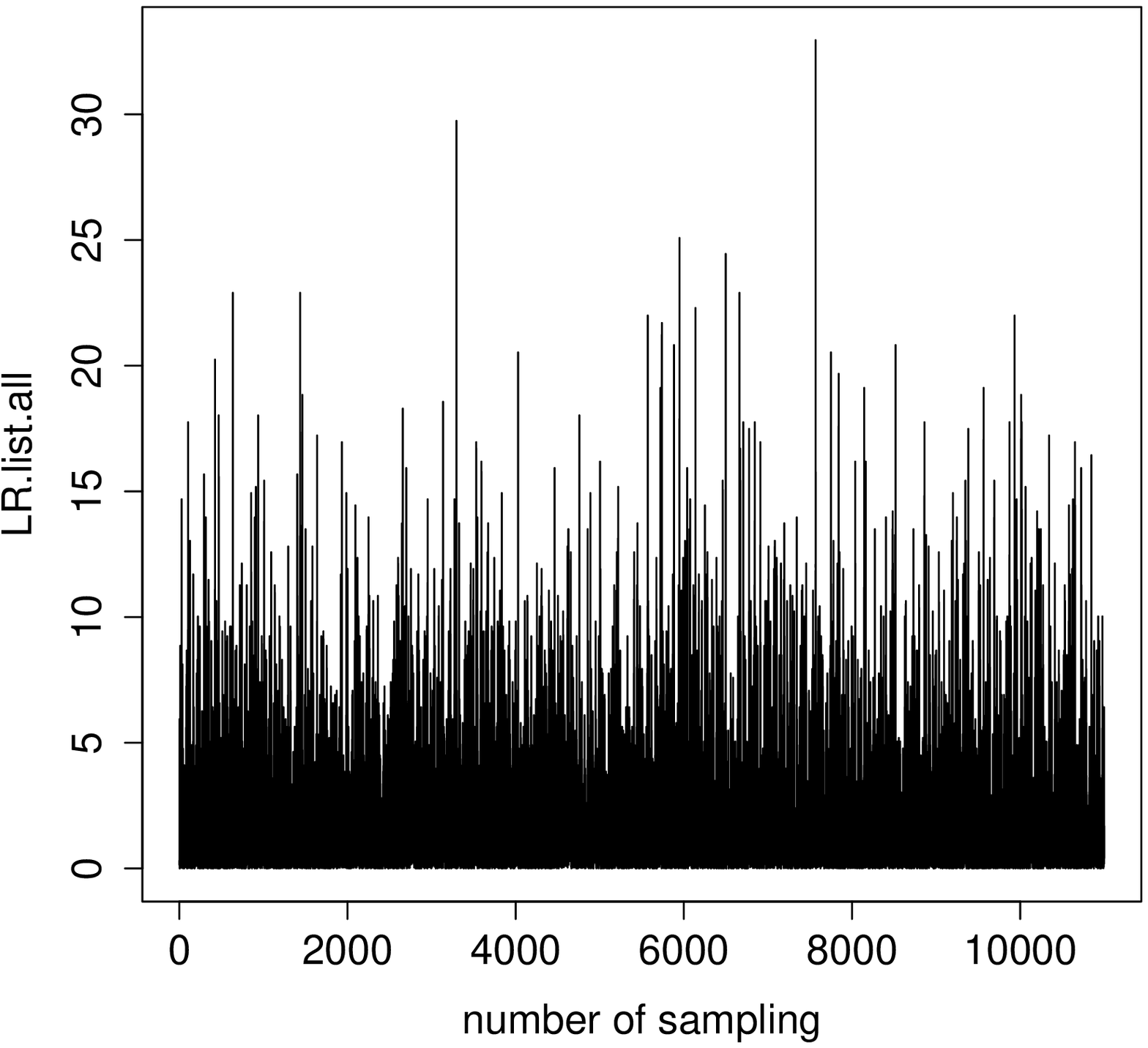}
\includegraphics[scale=0.20]{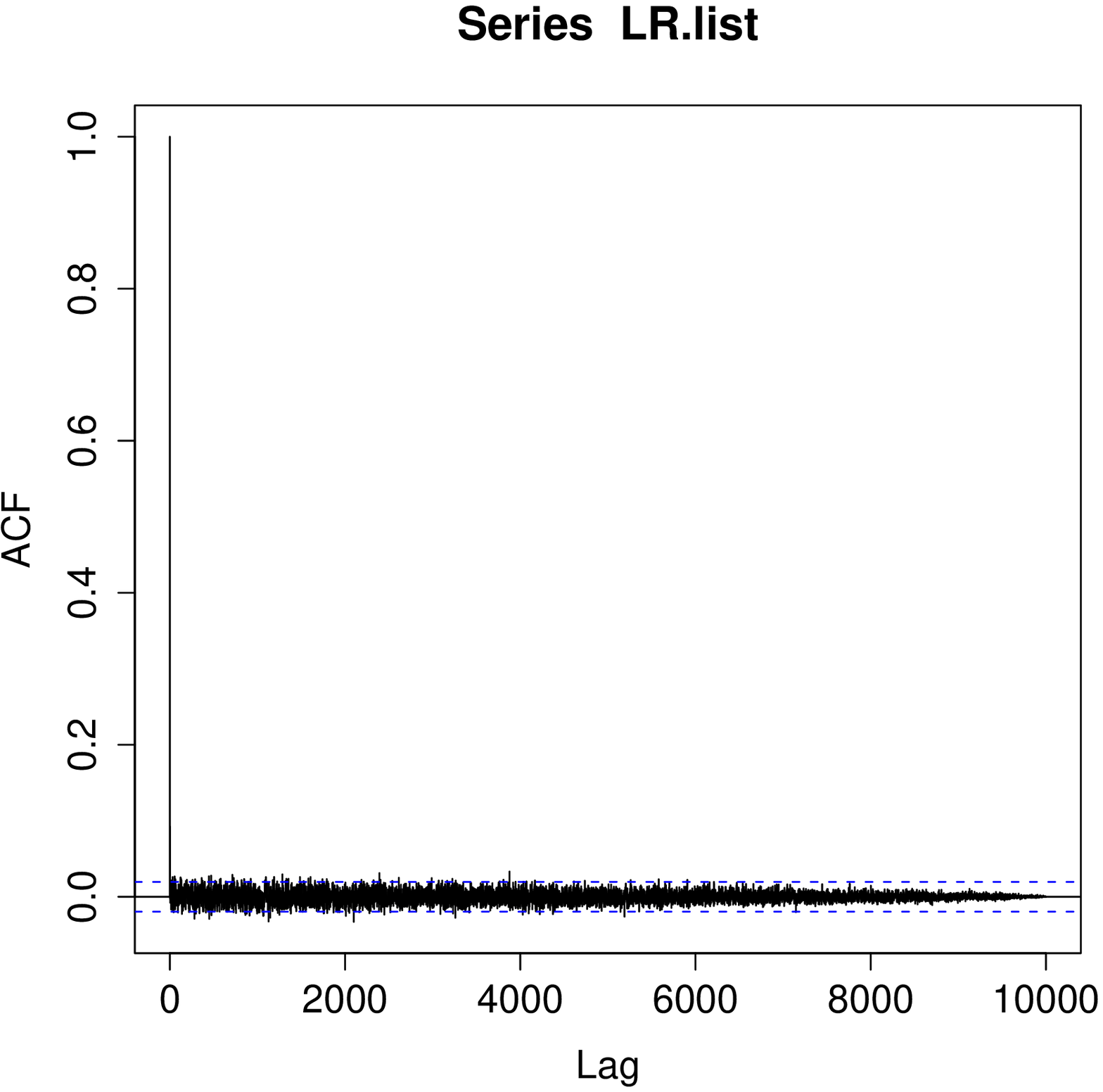}\\
 (b) a lattice basis with $Po(1)$\\
\includegraphics[scale=0.20]{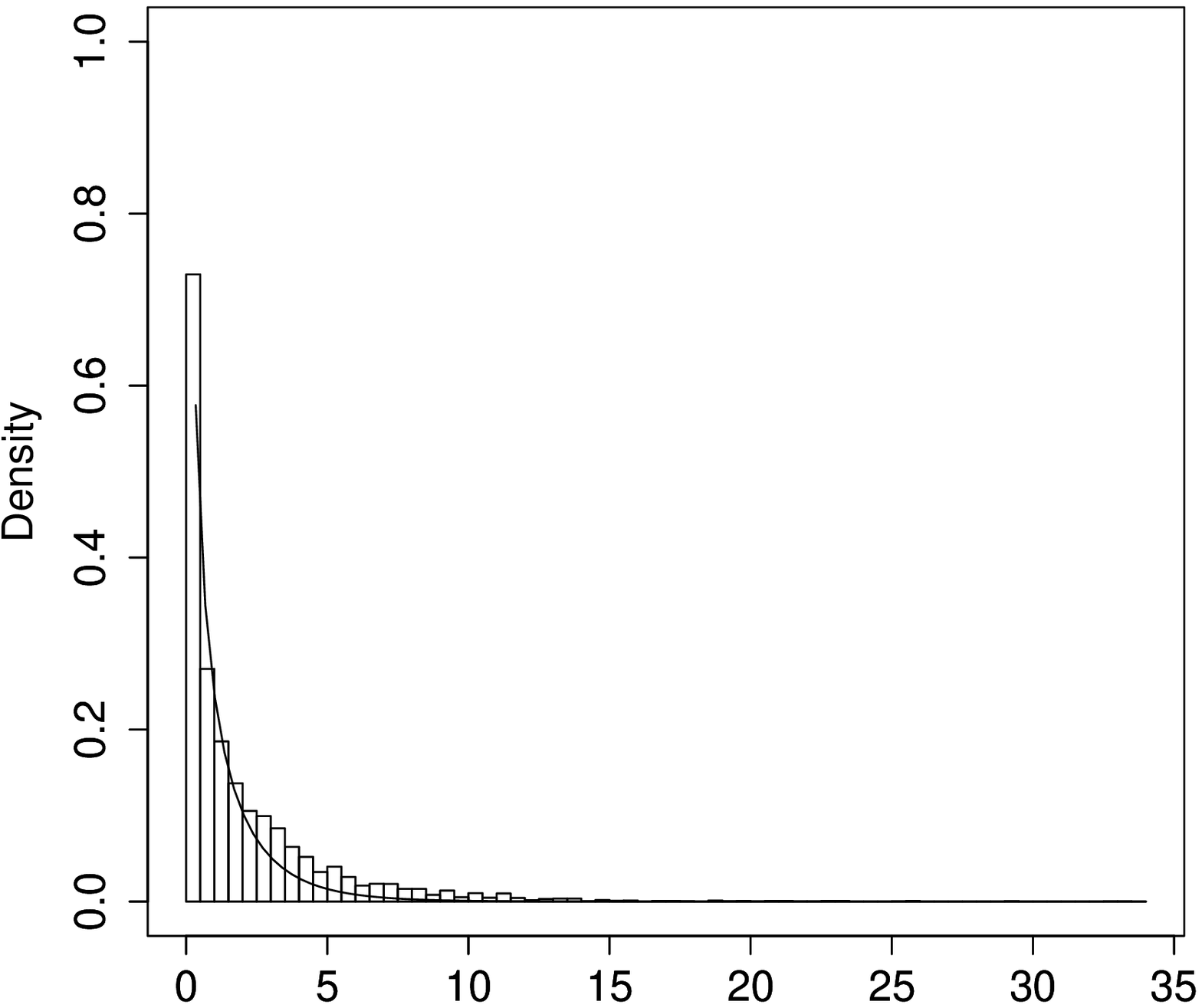}
\includegraphics[scale=0.20]{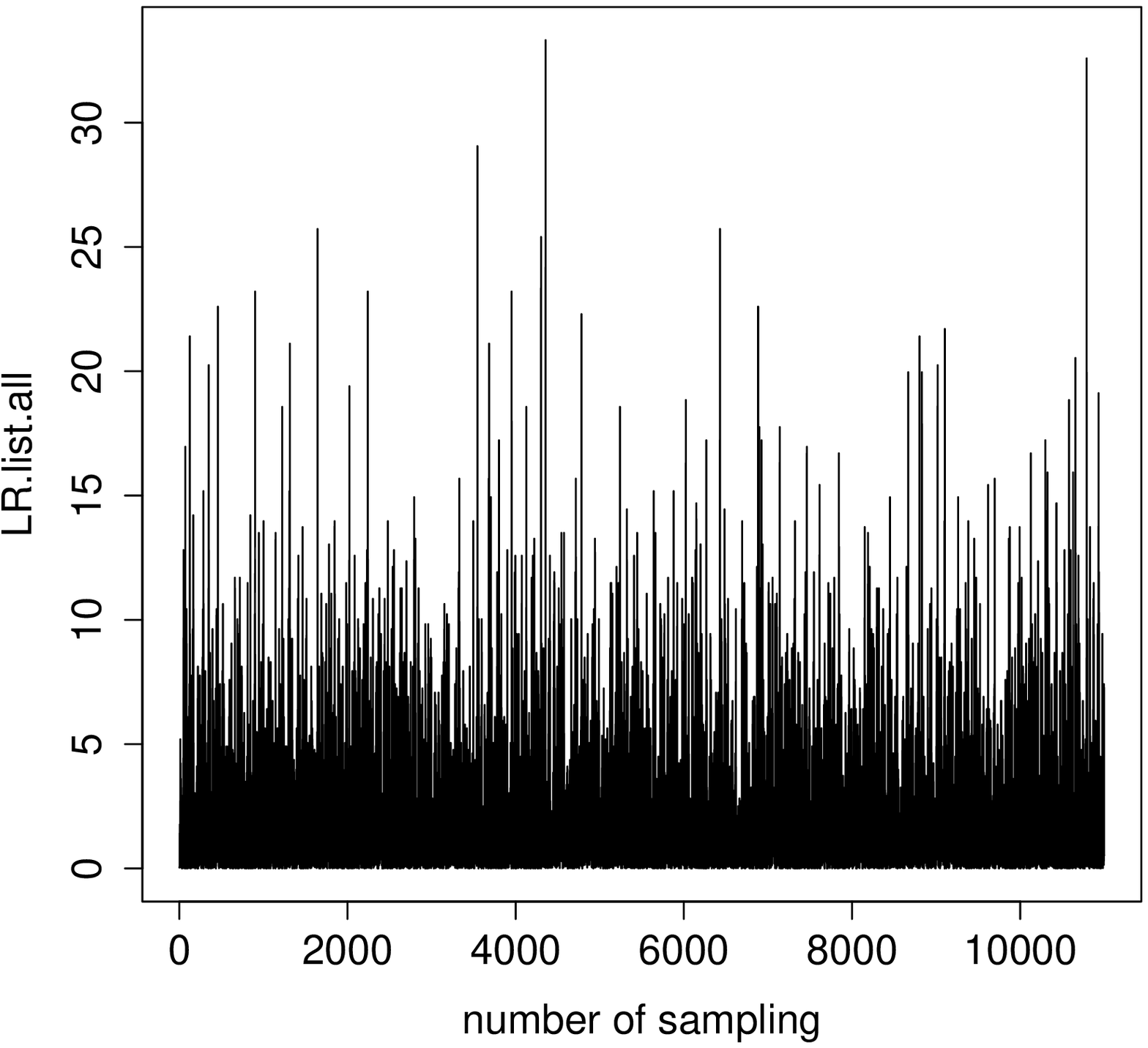}
\includegraphics[scale=0.20]{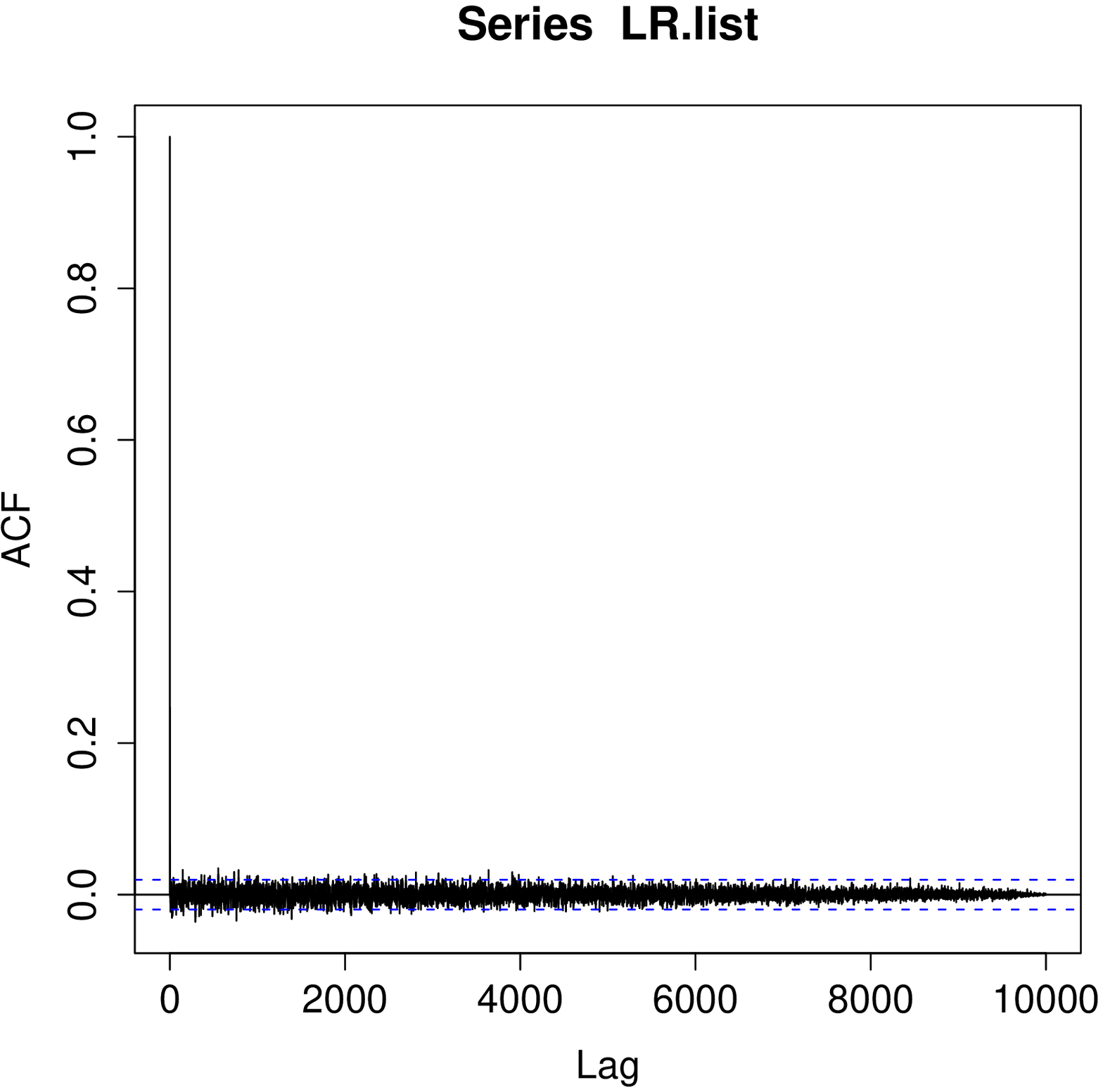}\\
 (c) a lattice basis with $Po(10)$\\
\includegraphics[scale=0.20]{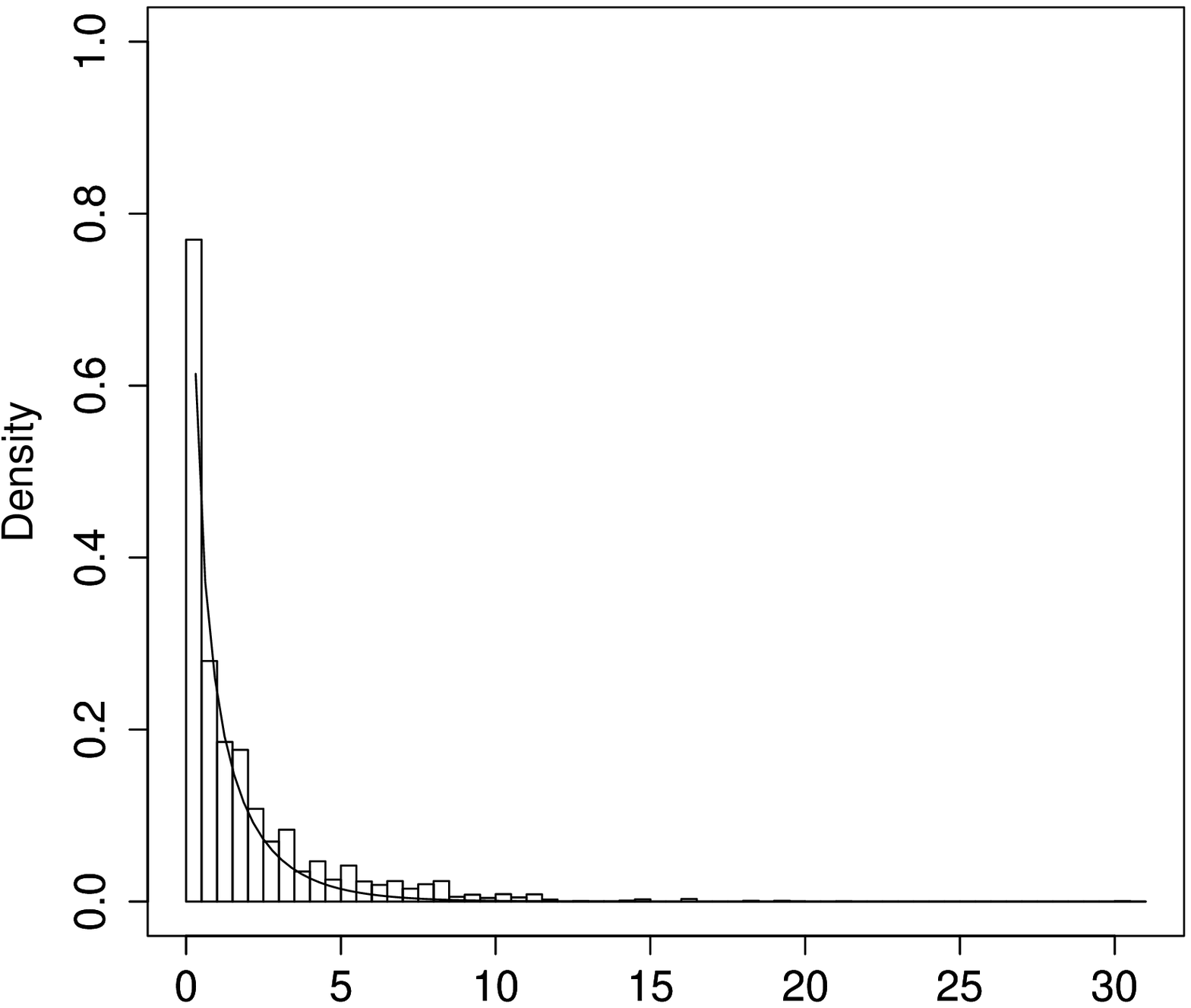}
\includegraphics[scale=0.20]{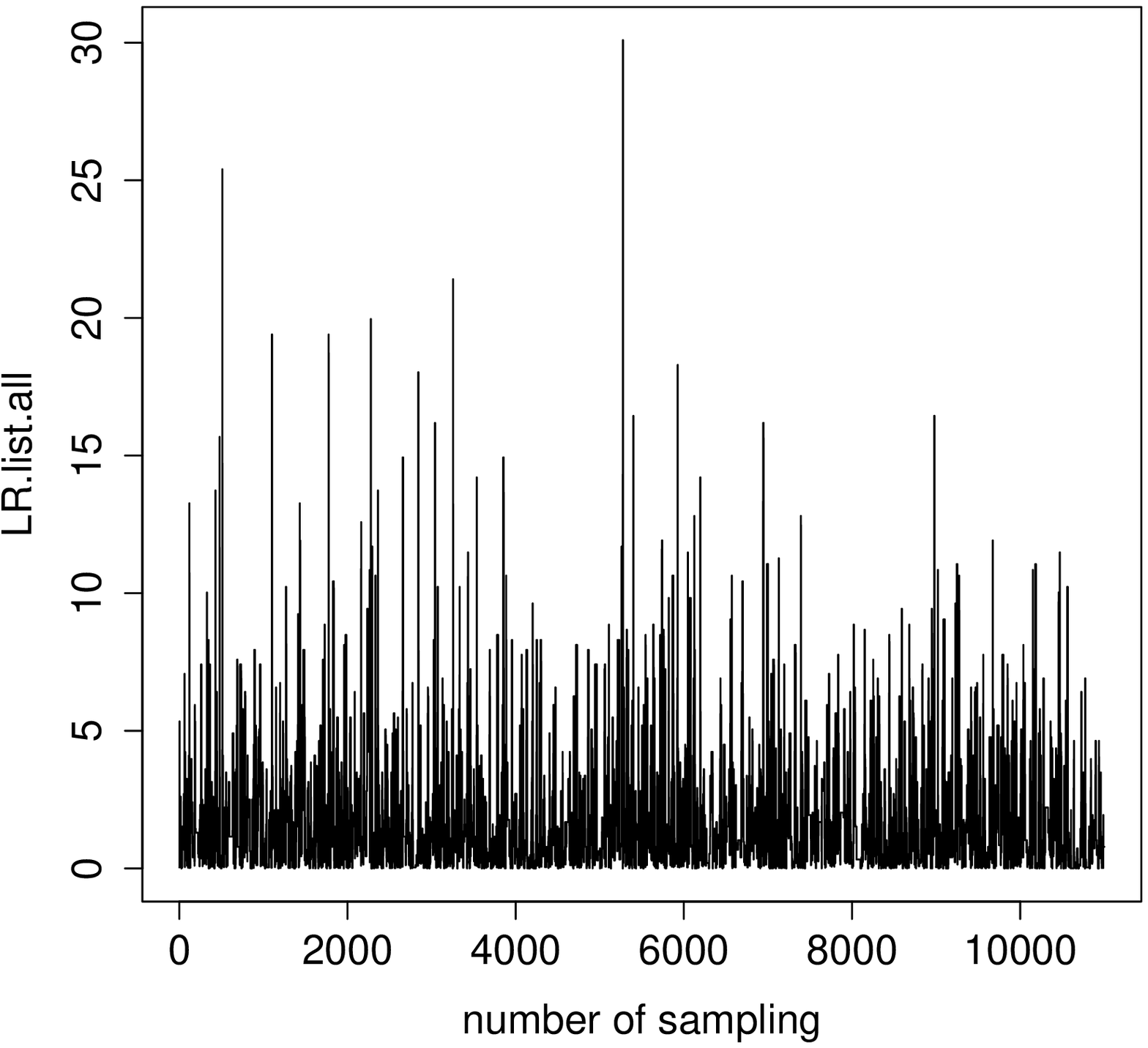}
\includegraphics[scale=0.20]{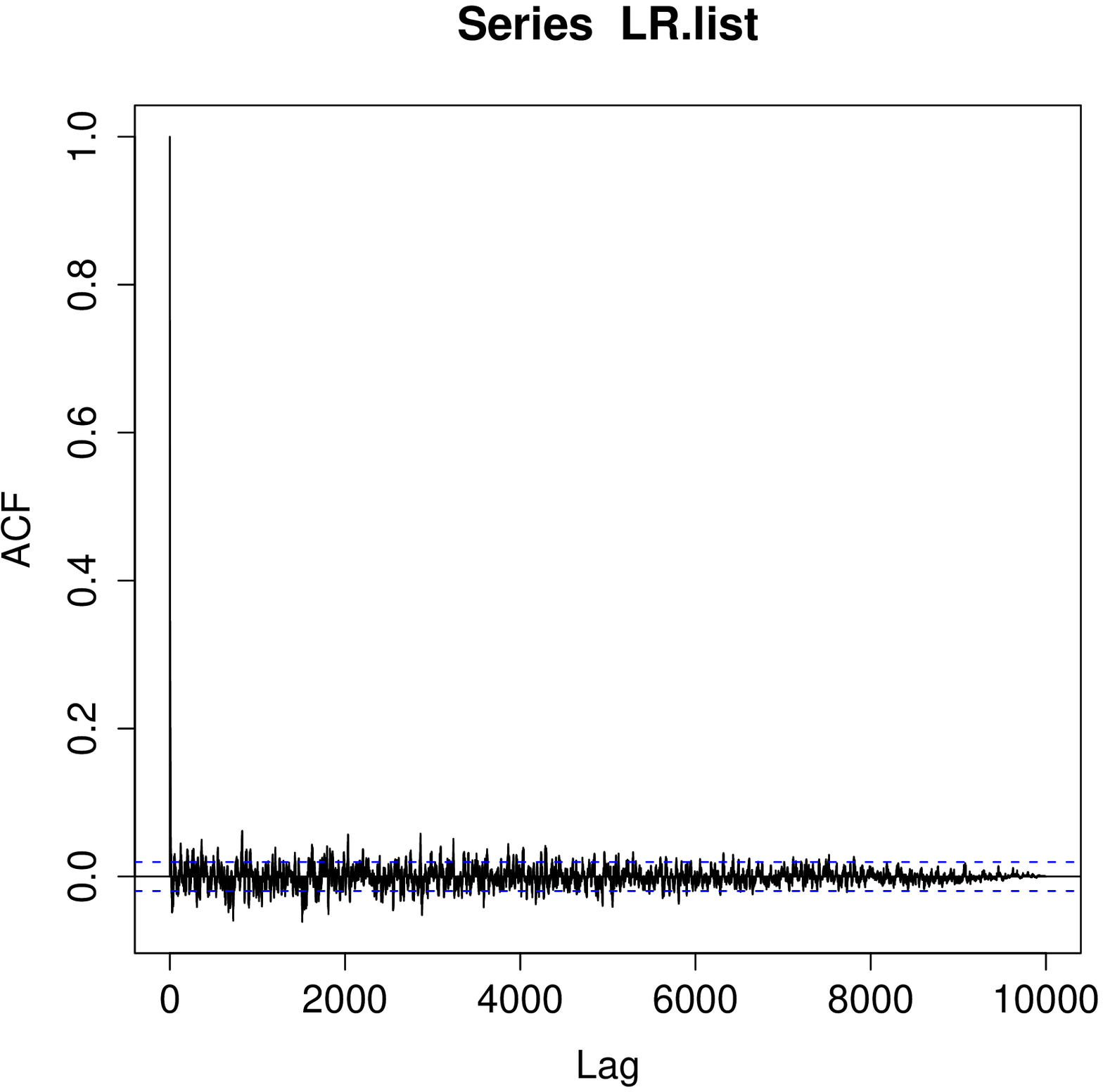}\\
 (d) a lattice basis with $Po(50)$\\
 \caption{Histograms, paths of LR statistic and correlograms of paths
 for discrete logistic regression model ((burn in,iteration) $=
 (1000,10000)$)} 
 \label{fig:logit1}
\end{figure}

\begin{figure}[h]
\centering
\noindent
\includegraphics[scale=0.20]{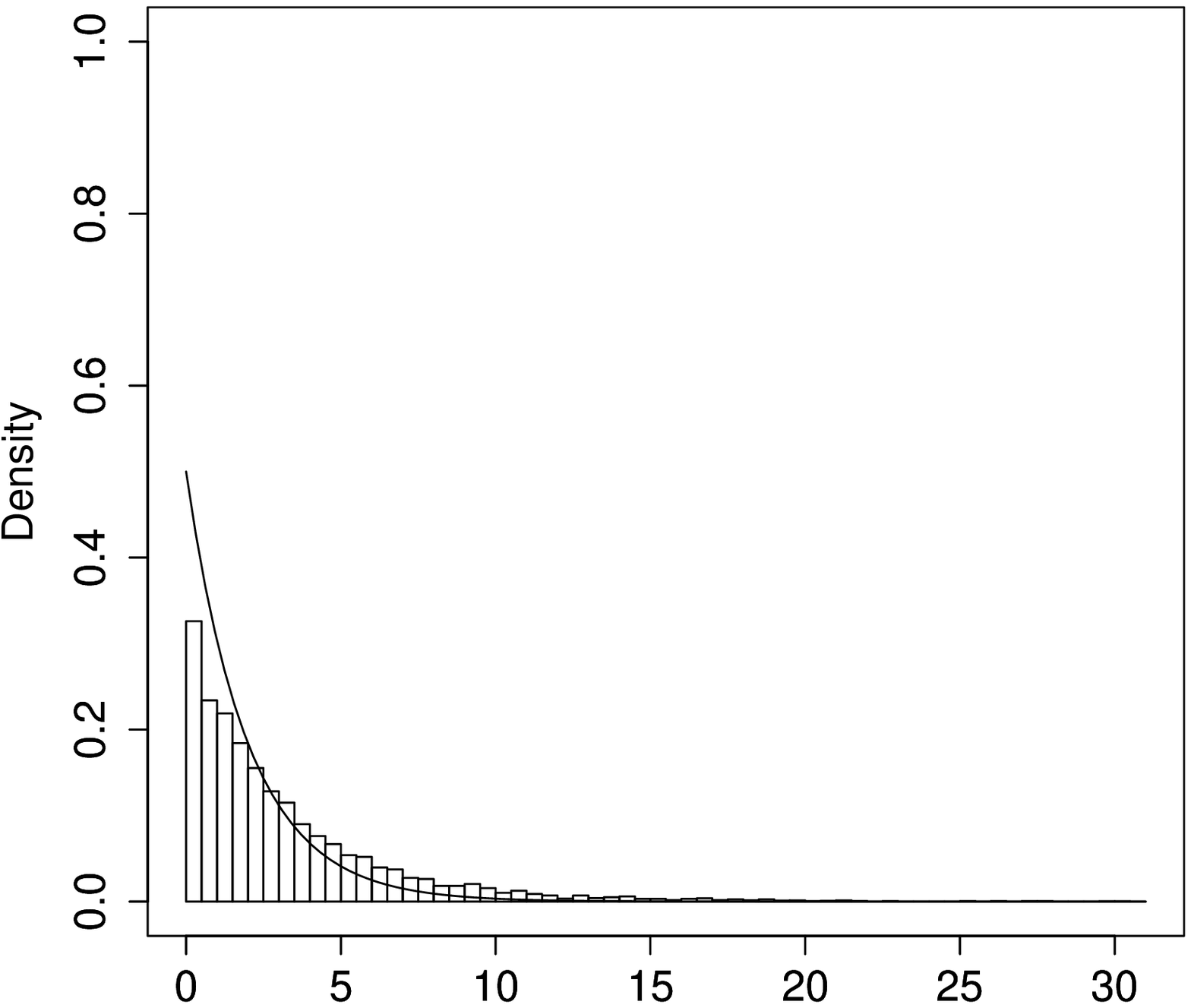}
\includegraphics[scale=0.20]{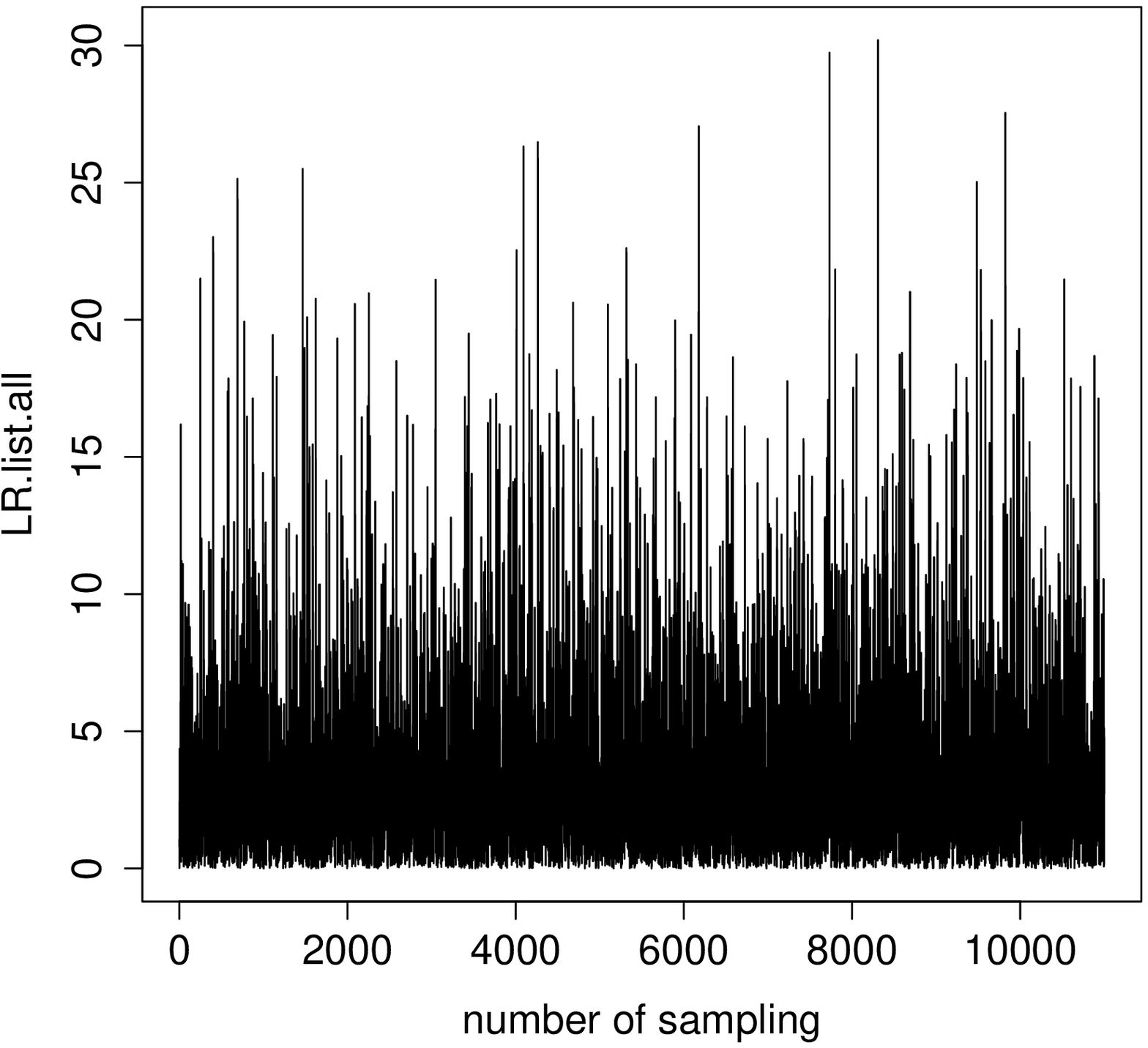}
\includegraphics[scale=0.20]{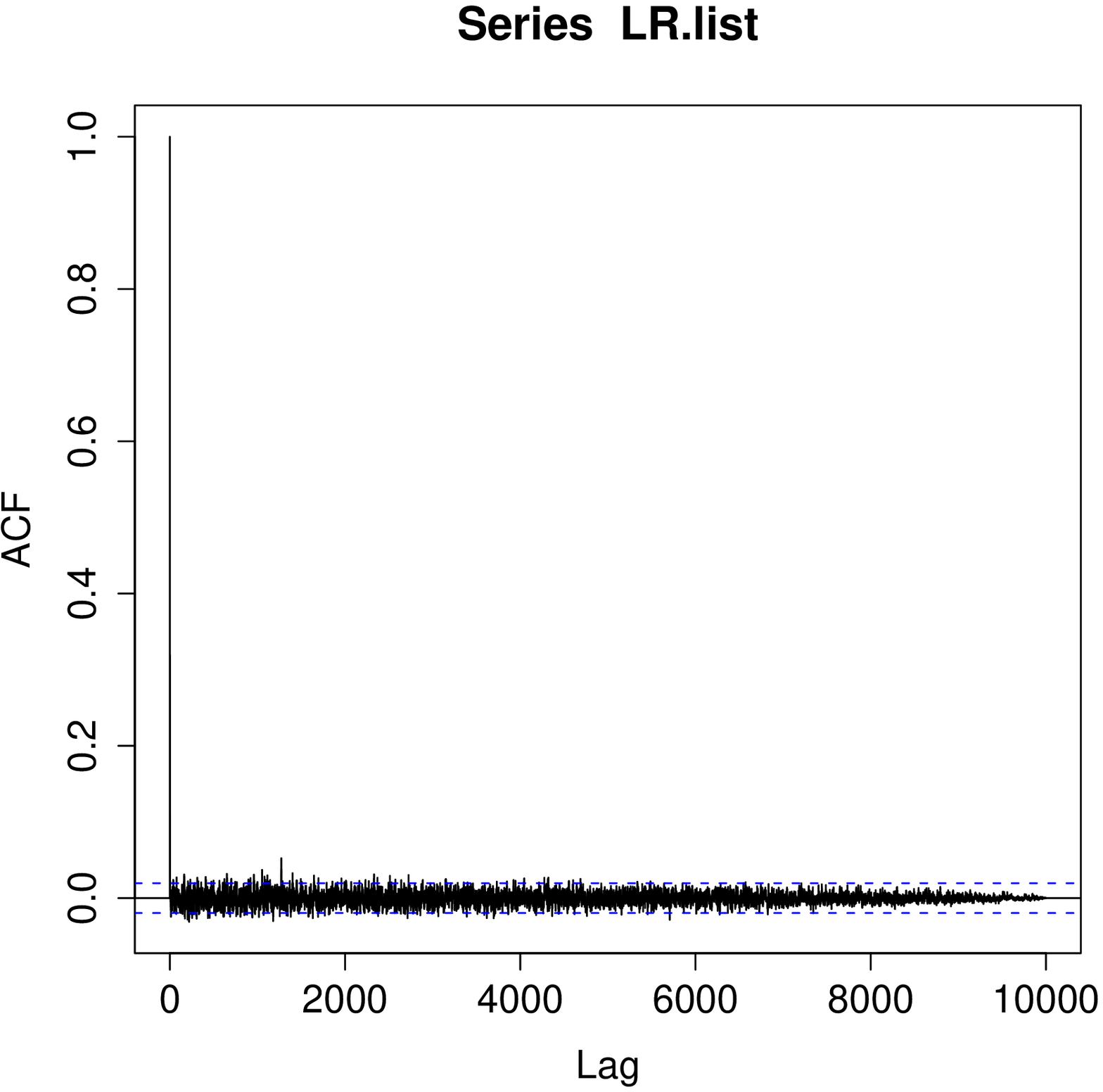}\\
 (a) a Markov basis\\
\includegraphics[scale=0.20]{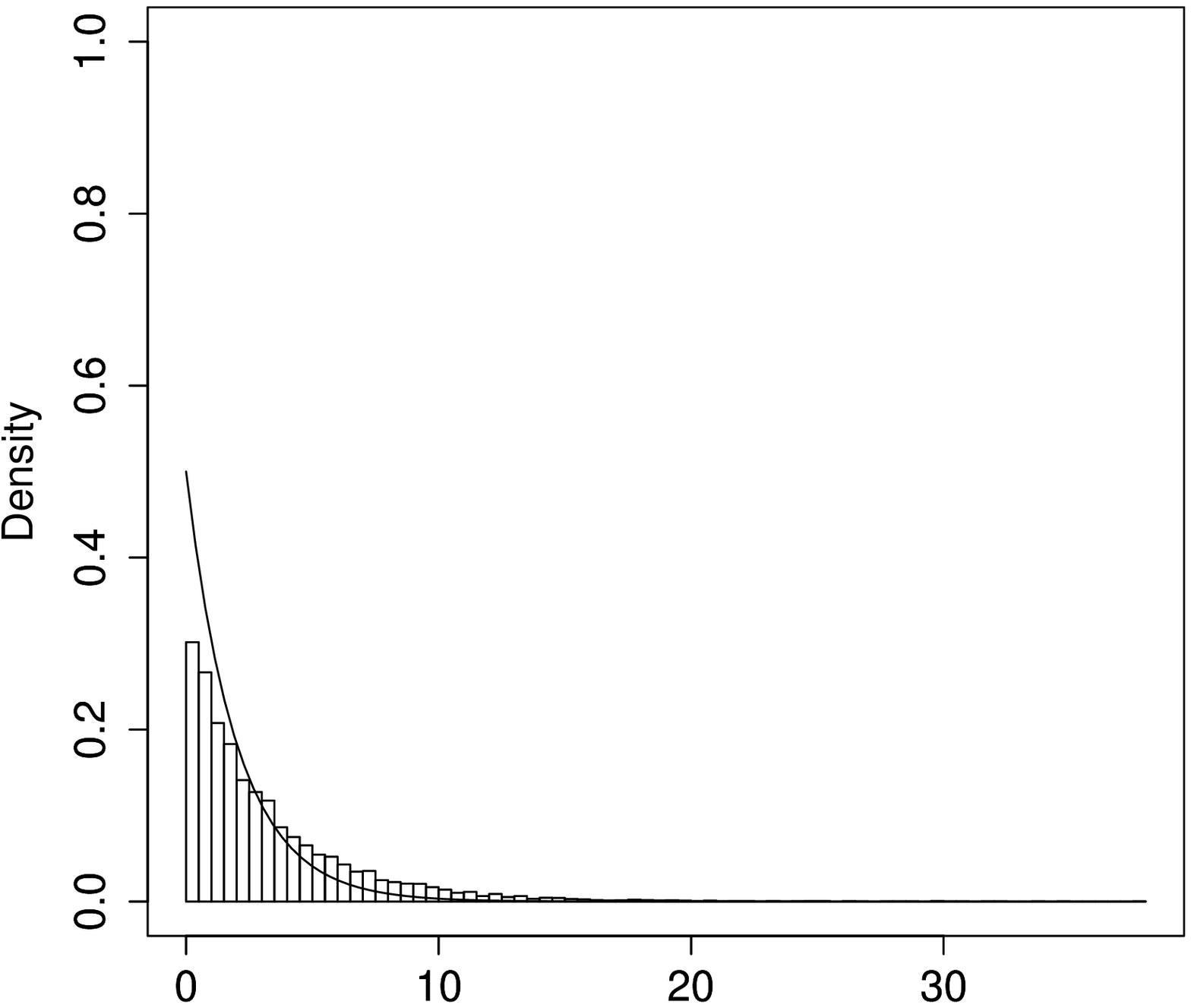}
\includegraphics[scale=0.20]{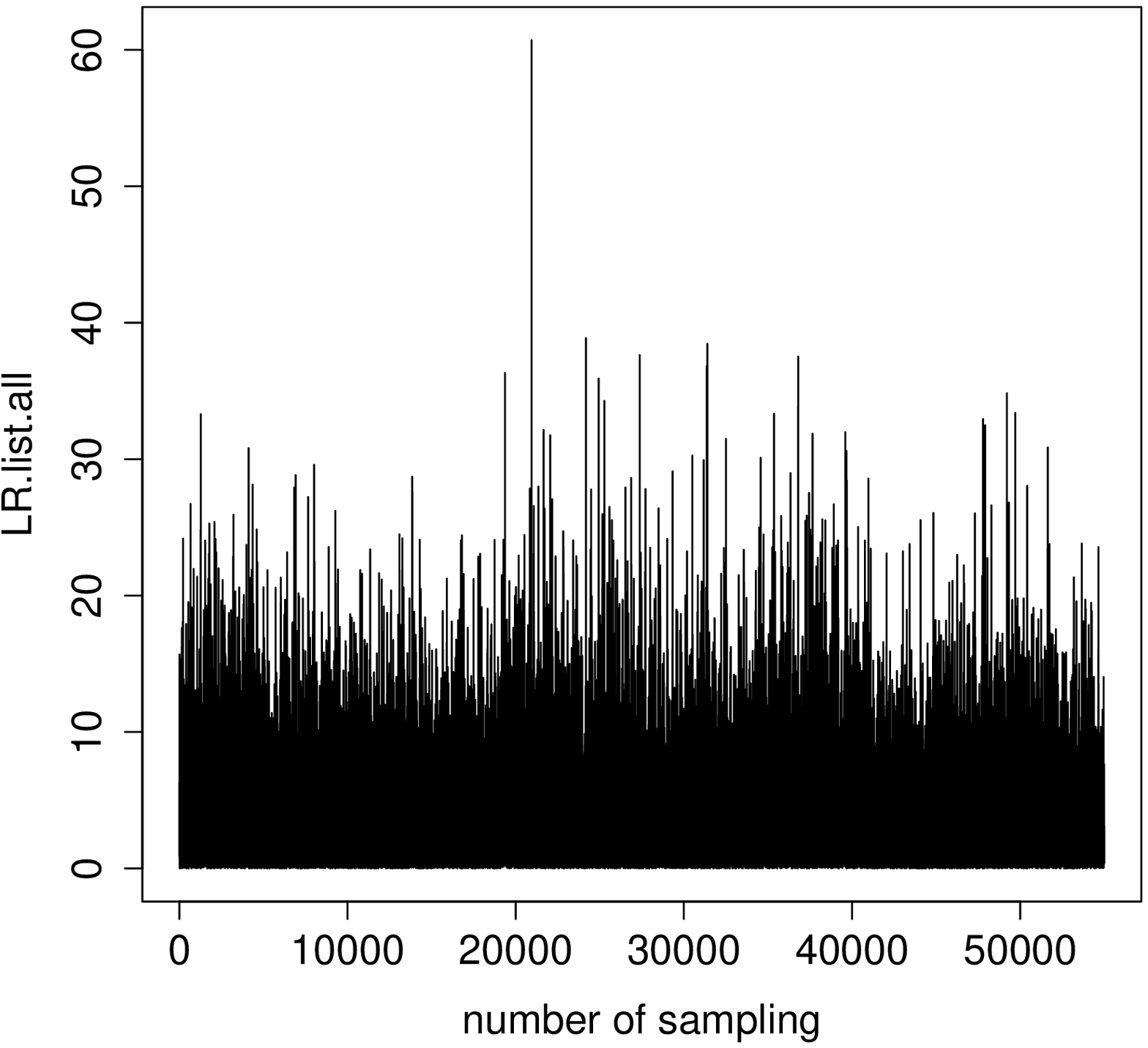}
\includegraphics[scale=0.20]{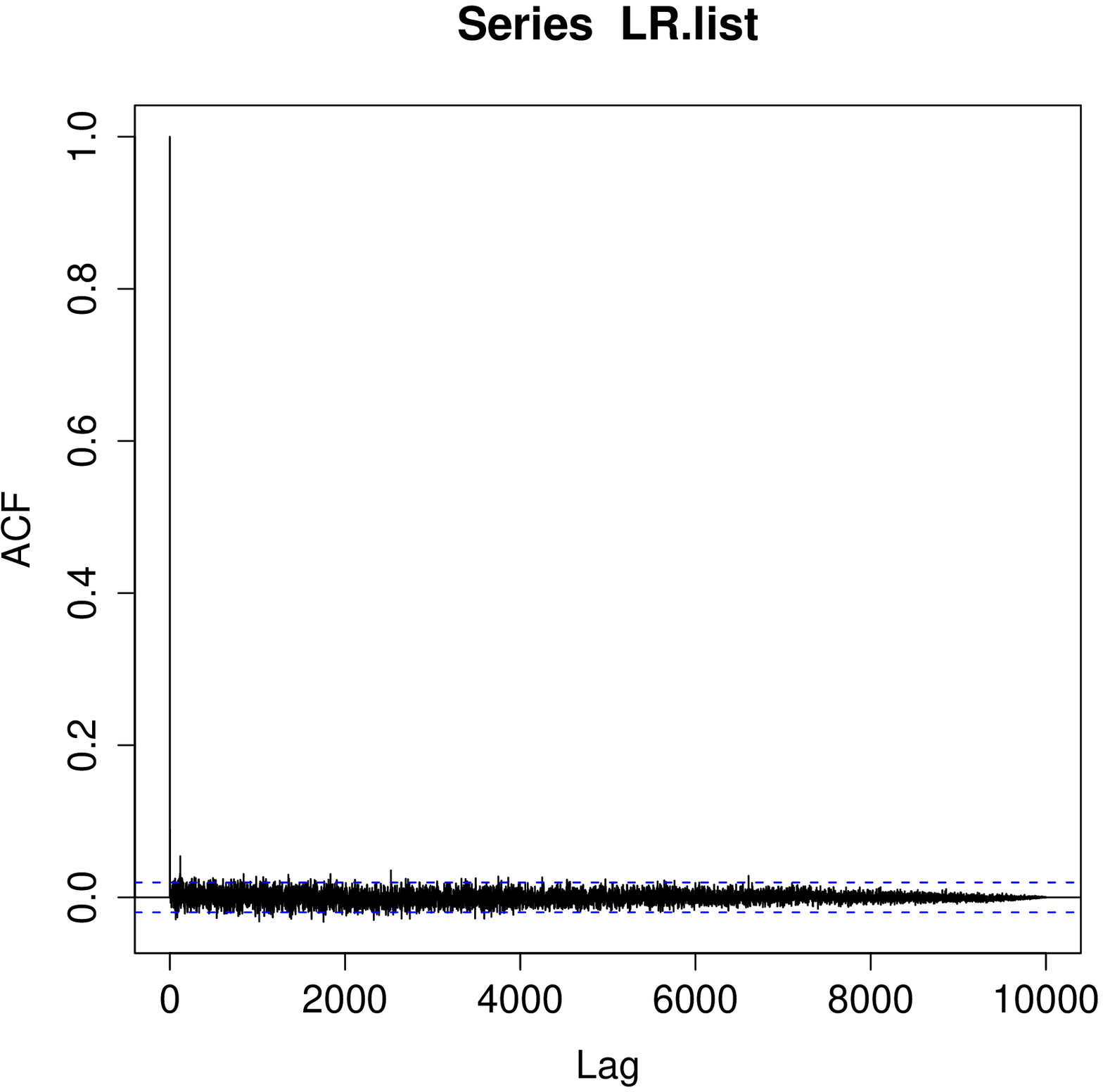}\\
 (b) a lattice basis with $Po(1)$\\
\includegraphics[scale=0.20]{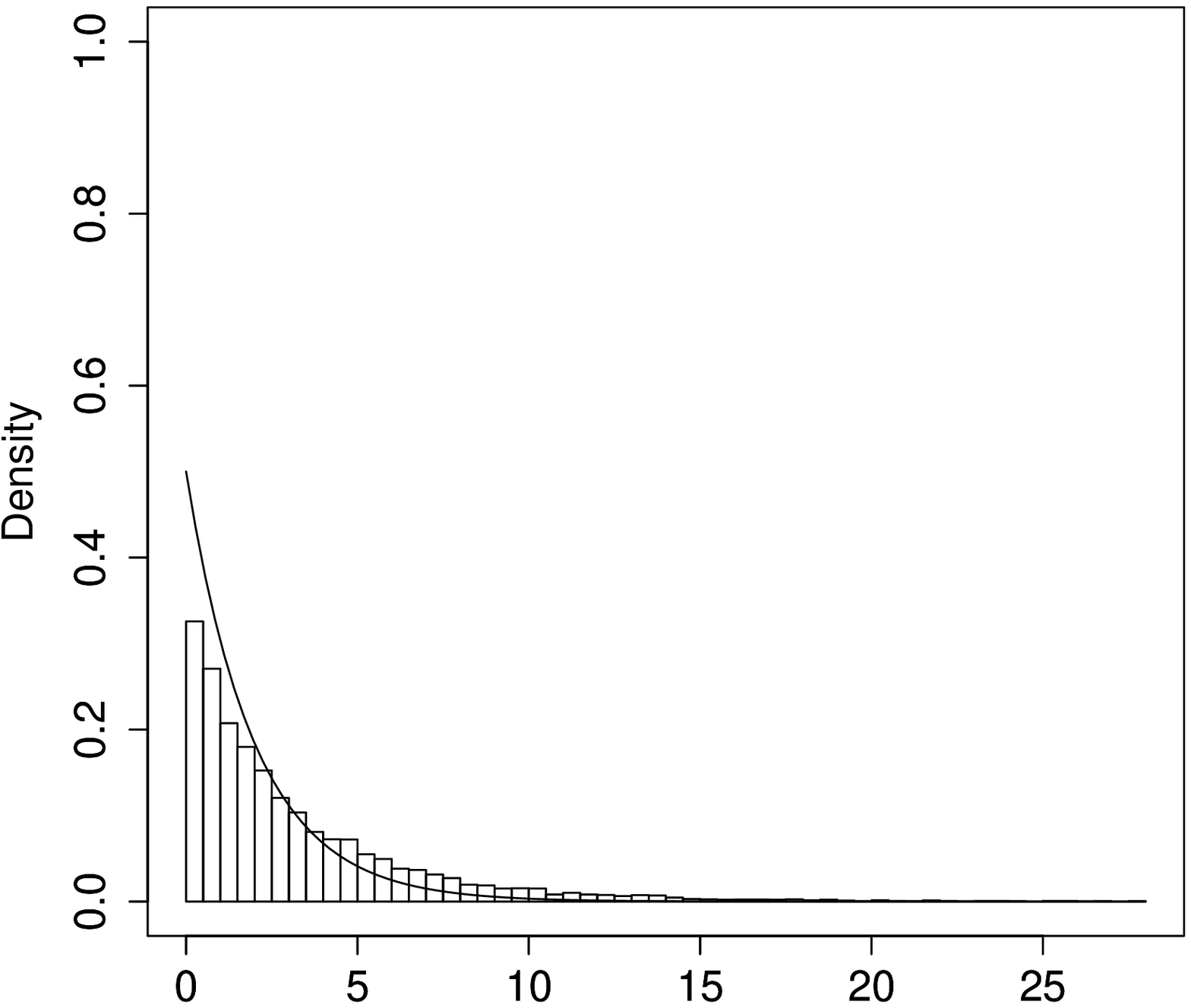}
\includegraphics[scale=0.20]{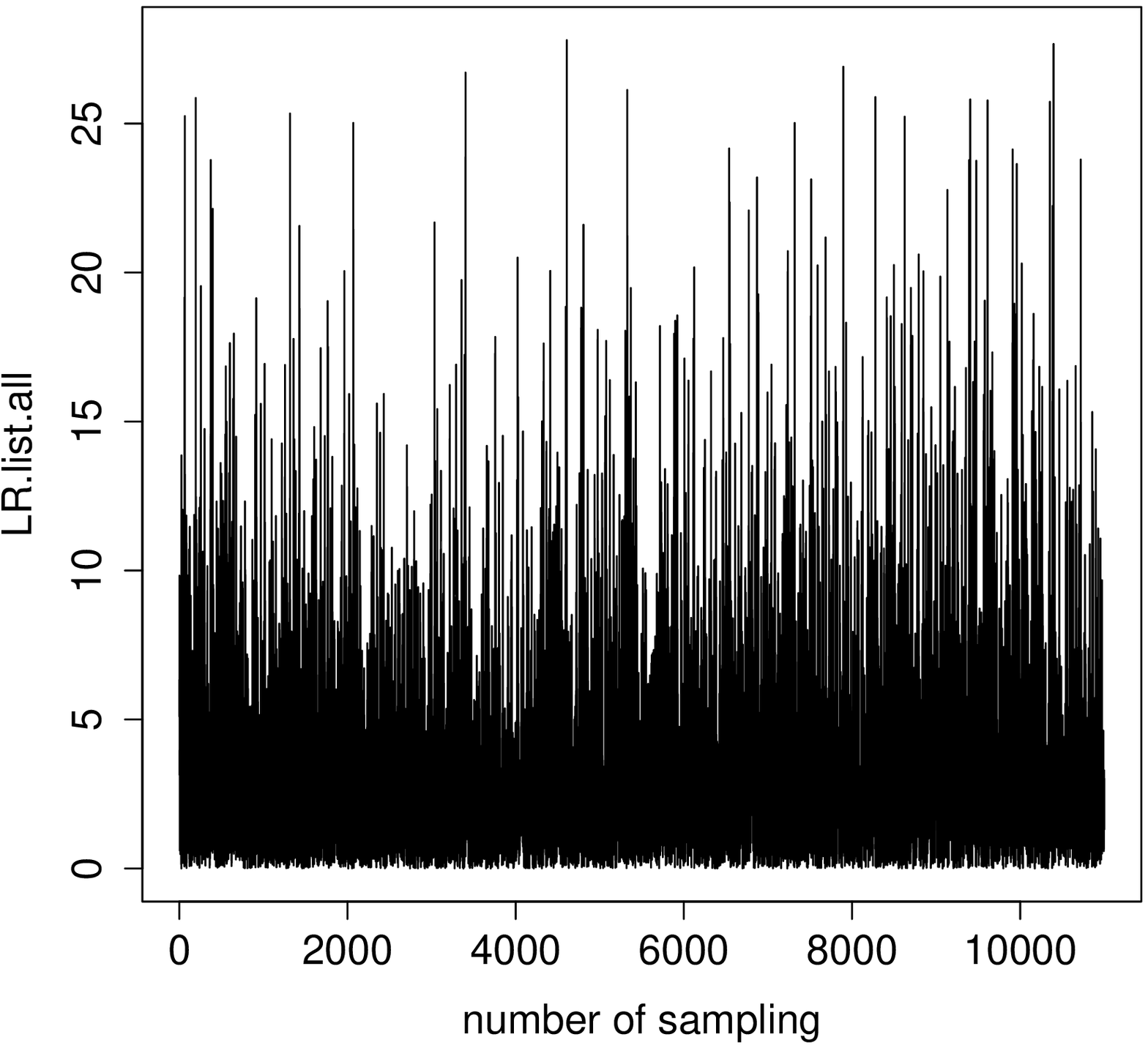}
\includegraphics[scale=0.20]{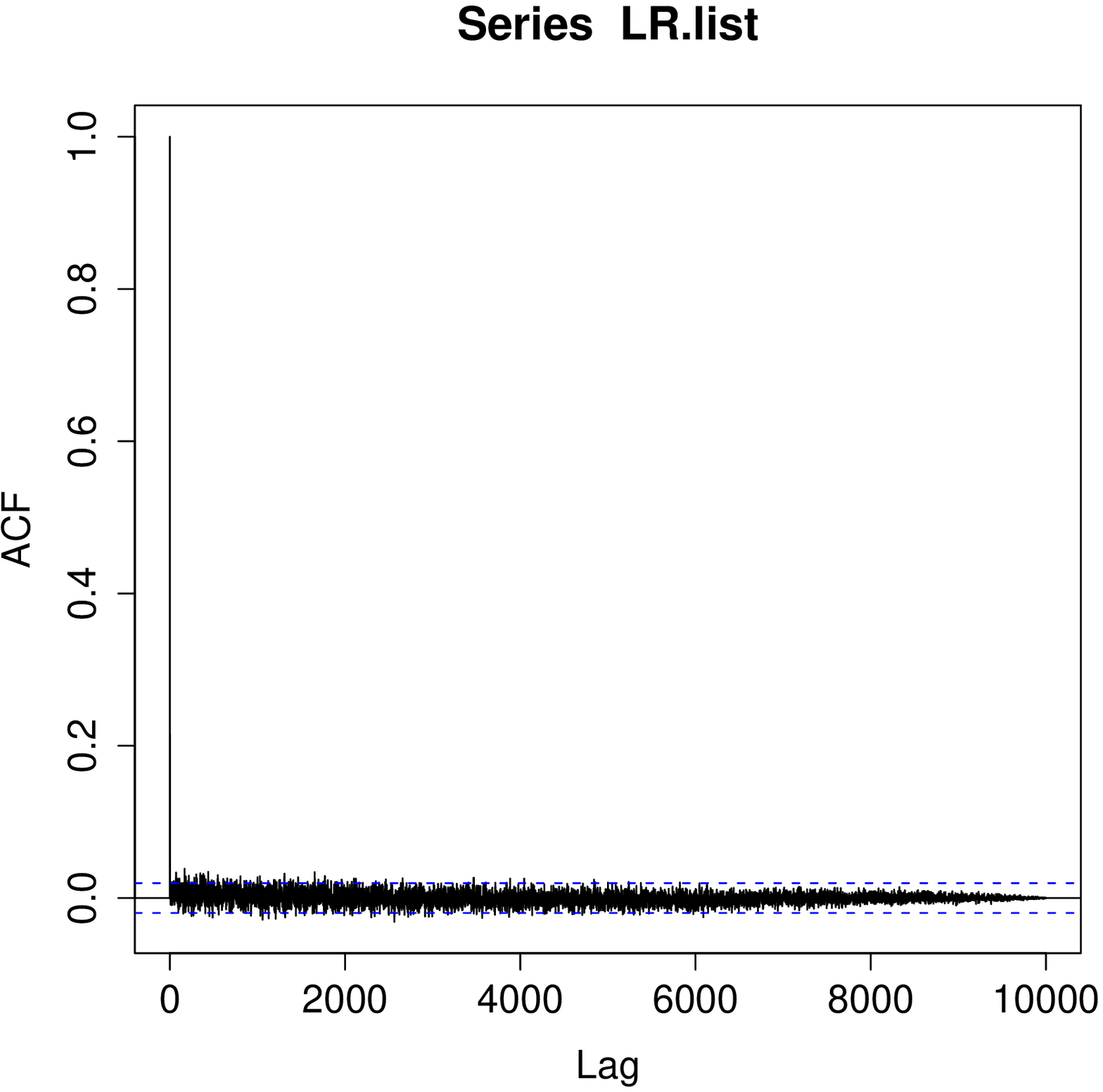}\\
 (c) a lattice basis with $Po(10)$\\
\includegraphics[scale=0.20]{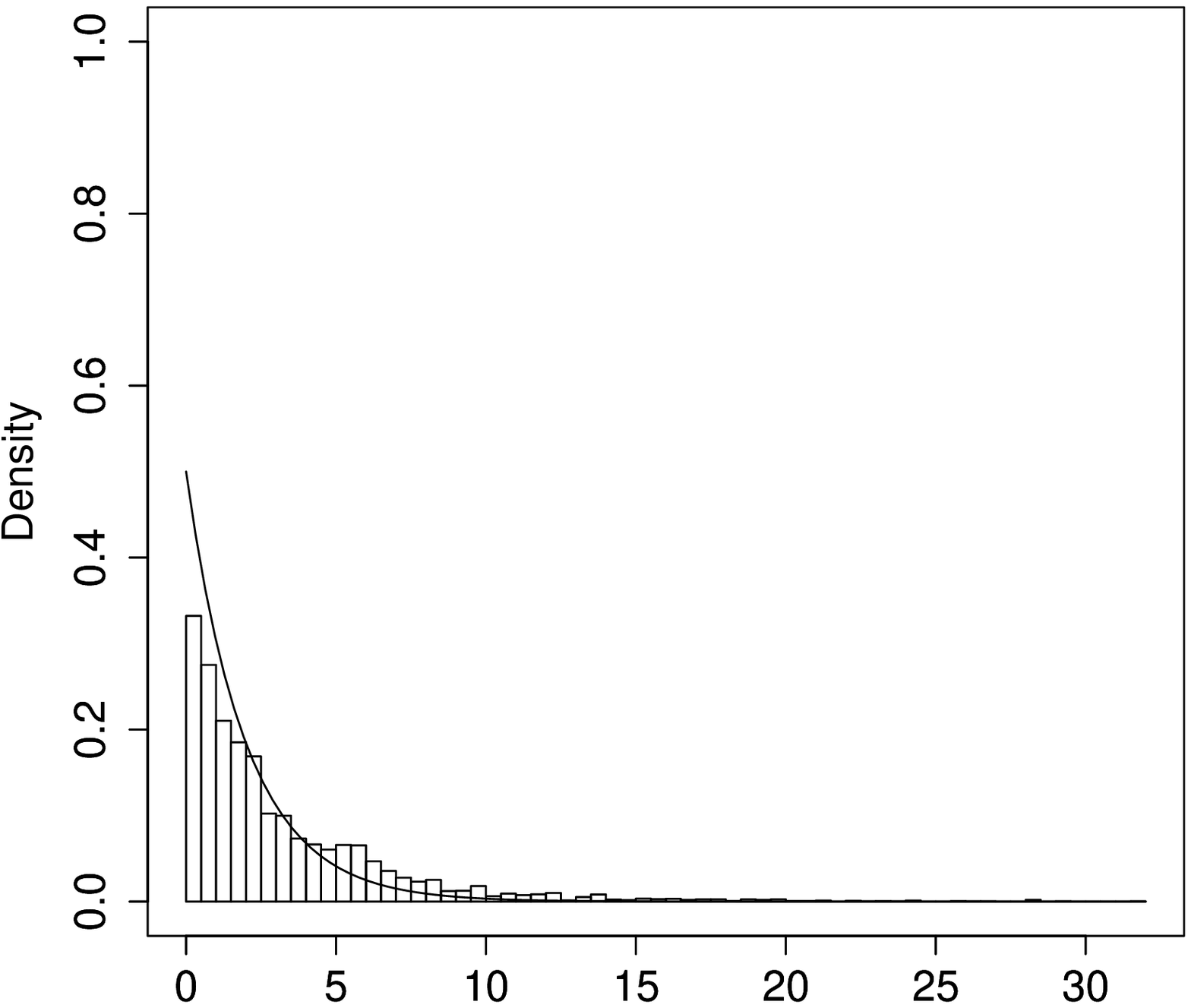}
\includegraphics[scale=0.20]{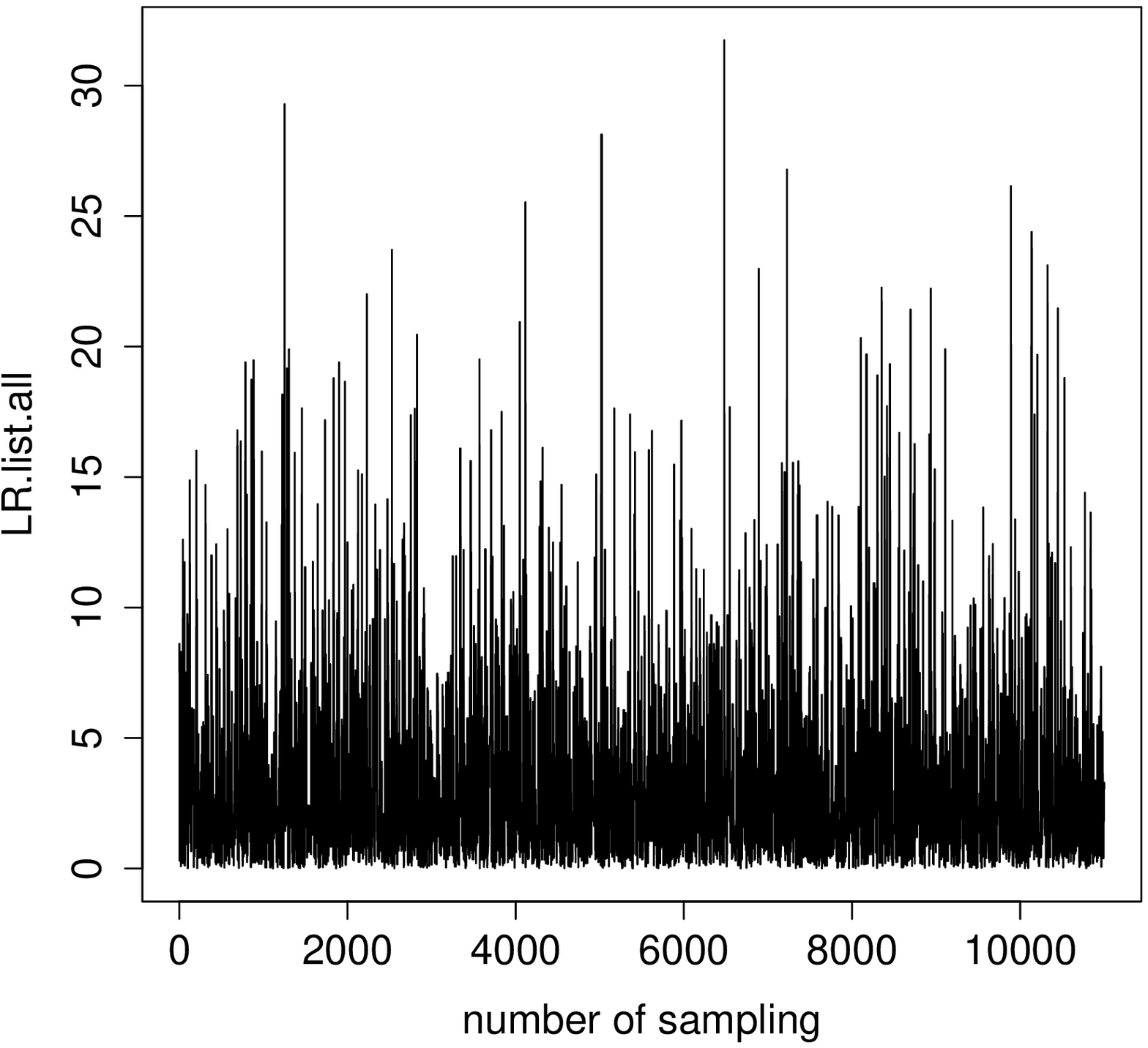}
\includegraphics[scale=0.20]{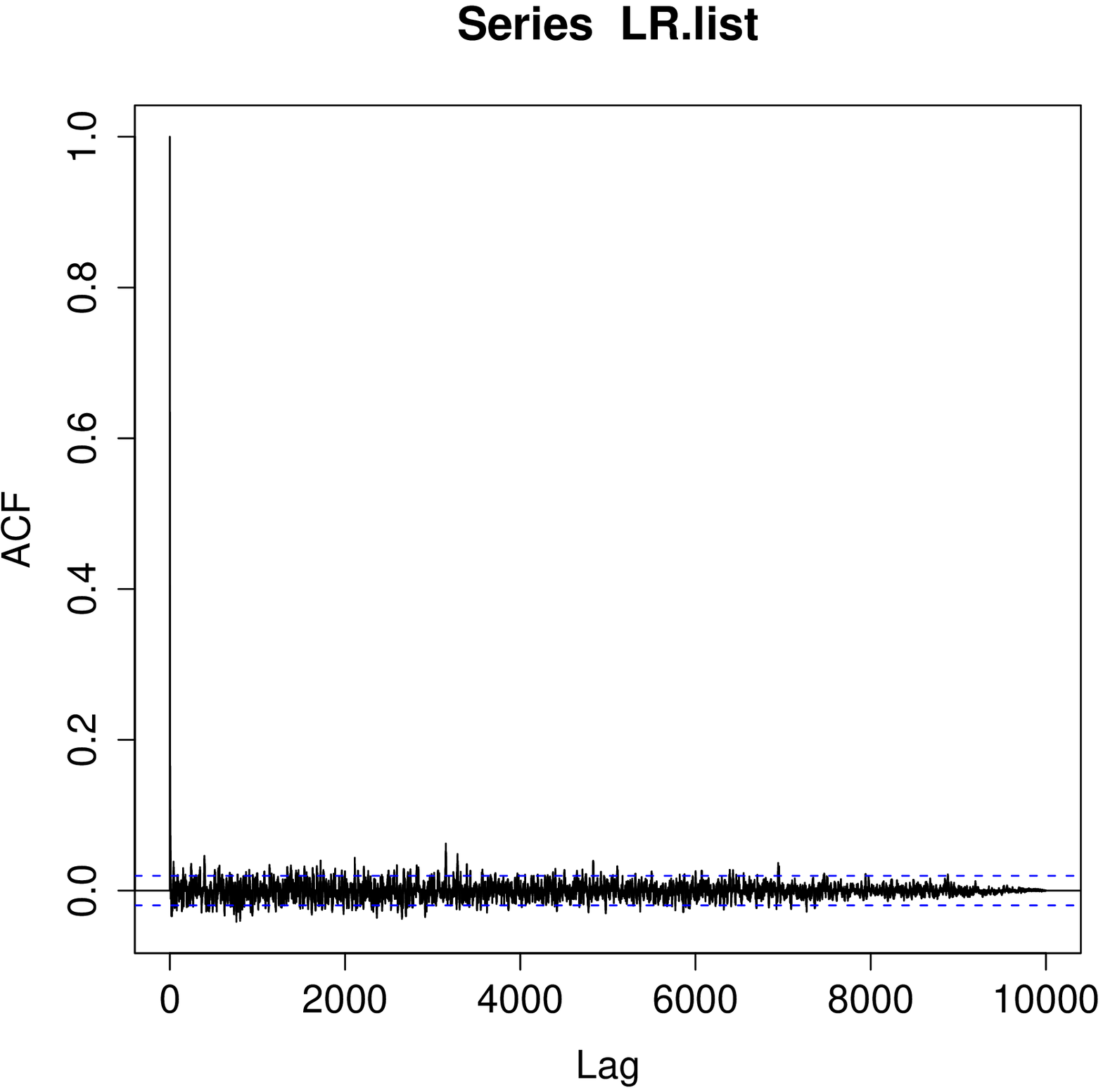}\\
 (d) a lattice basis with $Po(50)$\\
 \caption{Histograms, paths of LR statistic and correlograms of paths
 for trinomial discrete logit model ((burn in,iteration) $=
 (1000,10000)$)} 
 \label{fig:logit2}
\end{figure}

\begin{figure}[h]
\centering
\noindent
\includegraphics[scale=0.20]{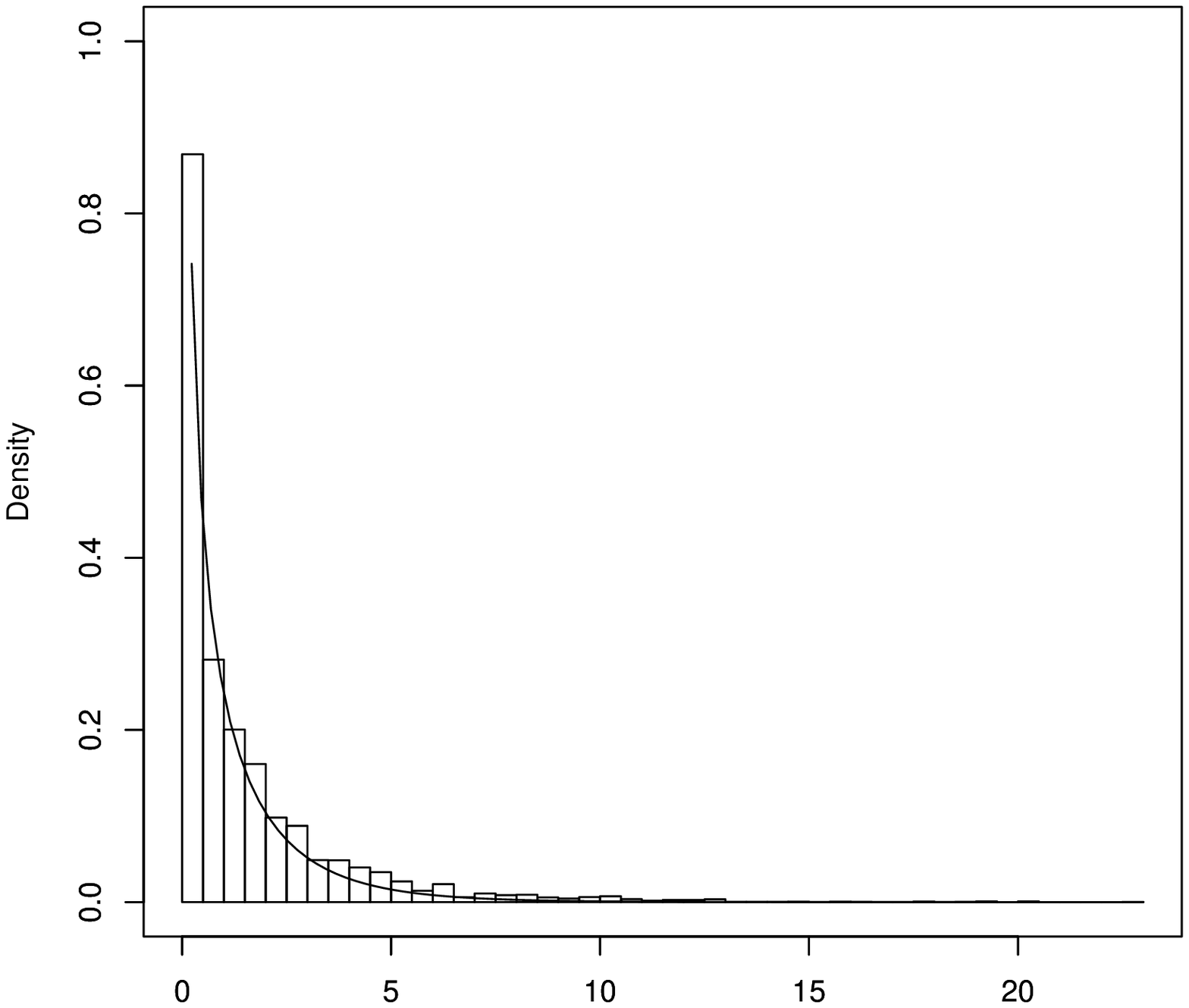}
\includegraphics[scale=0.20]{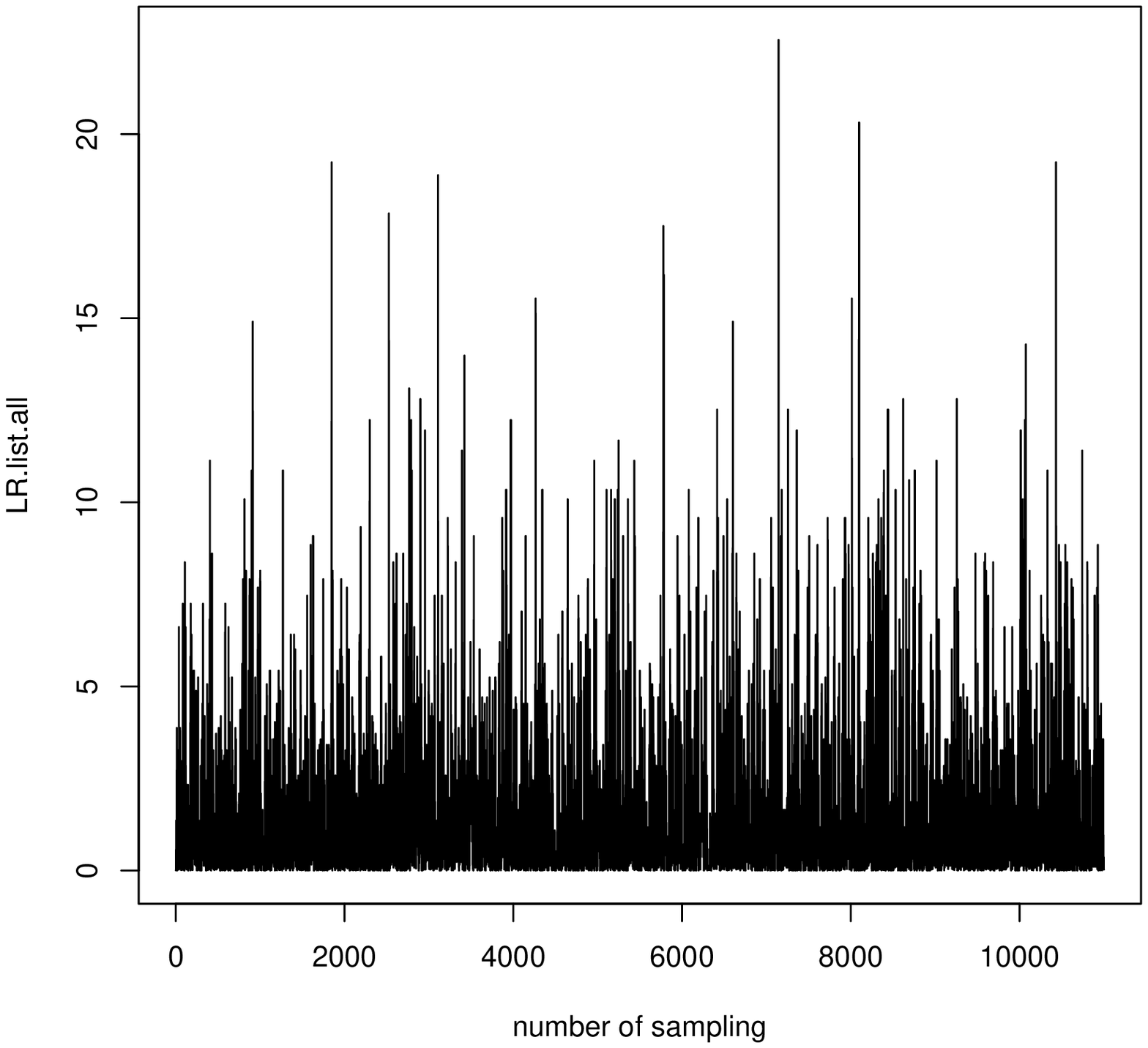}
\includegraphics[scale=0.20]{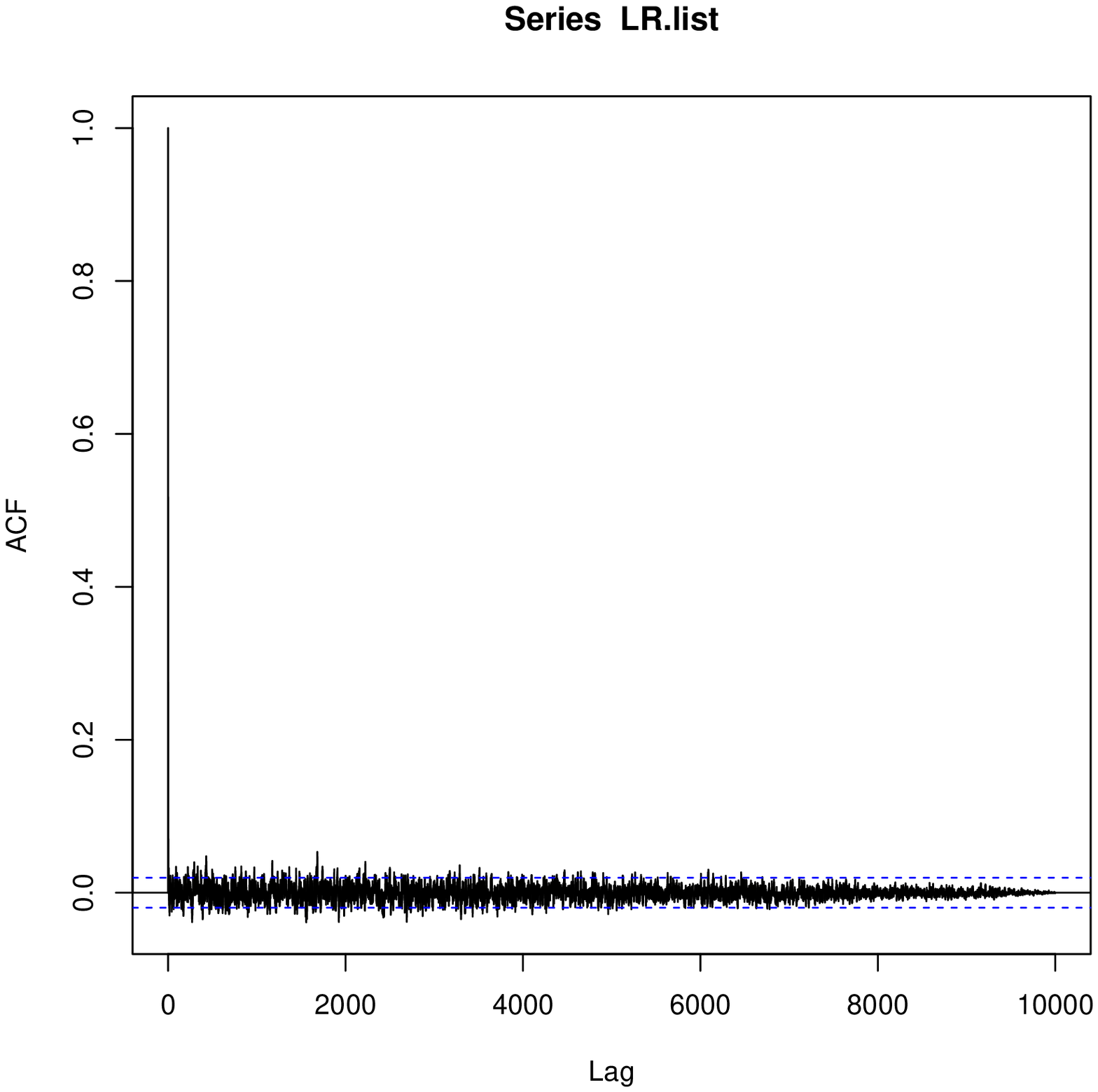}\\
 (a) binomial, a lattice basis with $Geom(0.1)$\\
\includegraphics[scale=0.20]{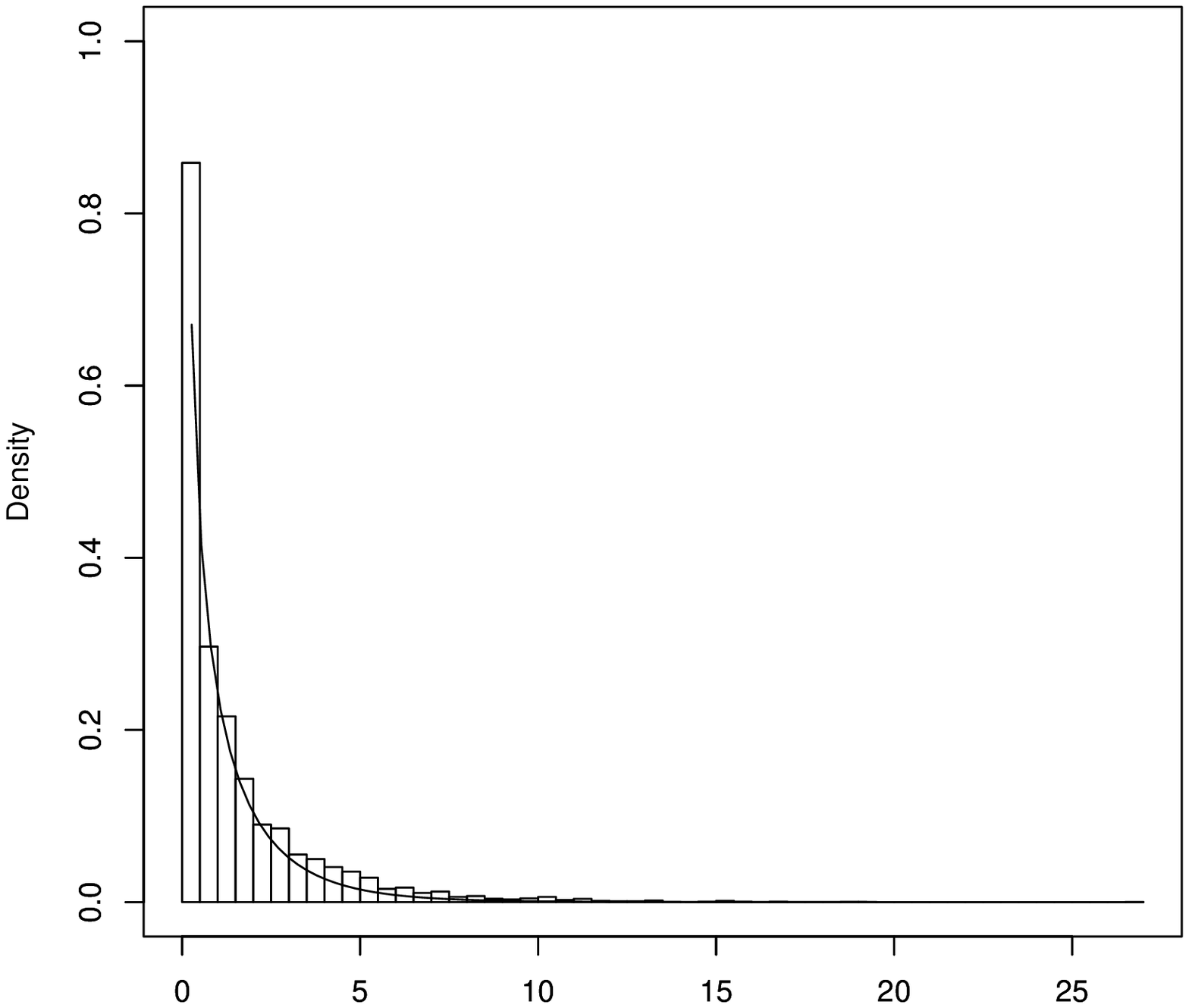}
\includegraphics[scale=0.20]{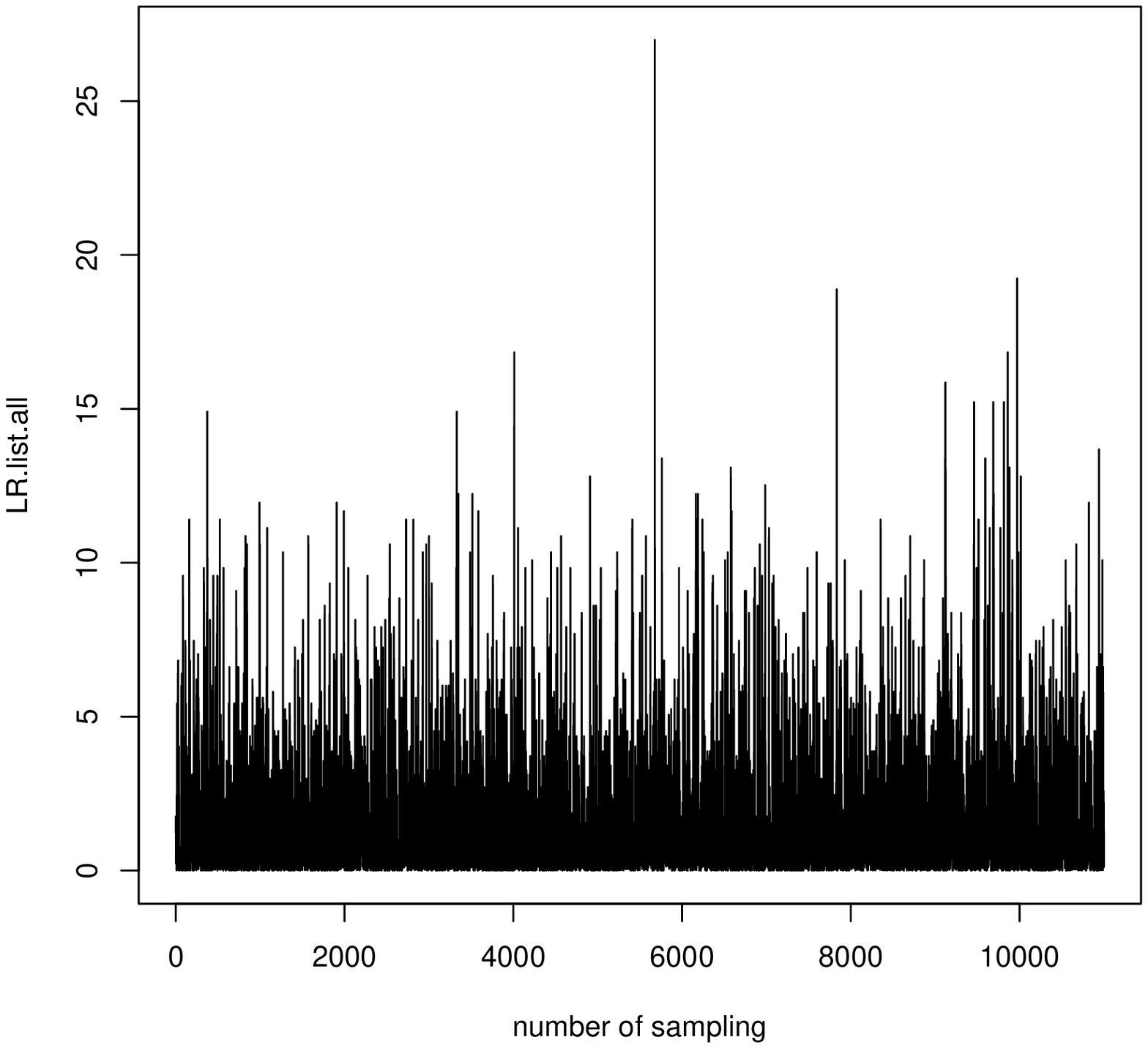}
\includegraphics[scale=0.20]{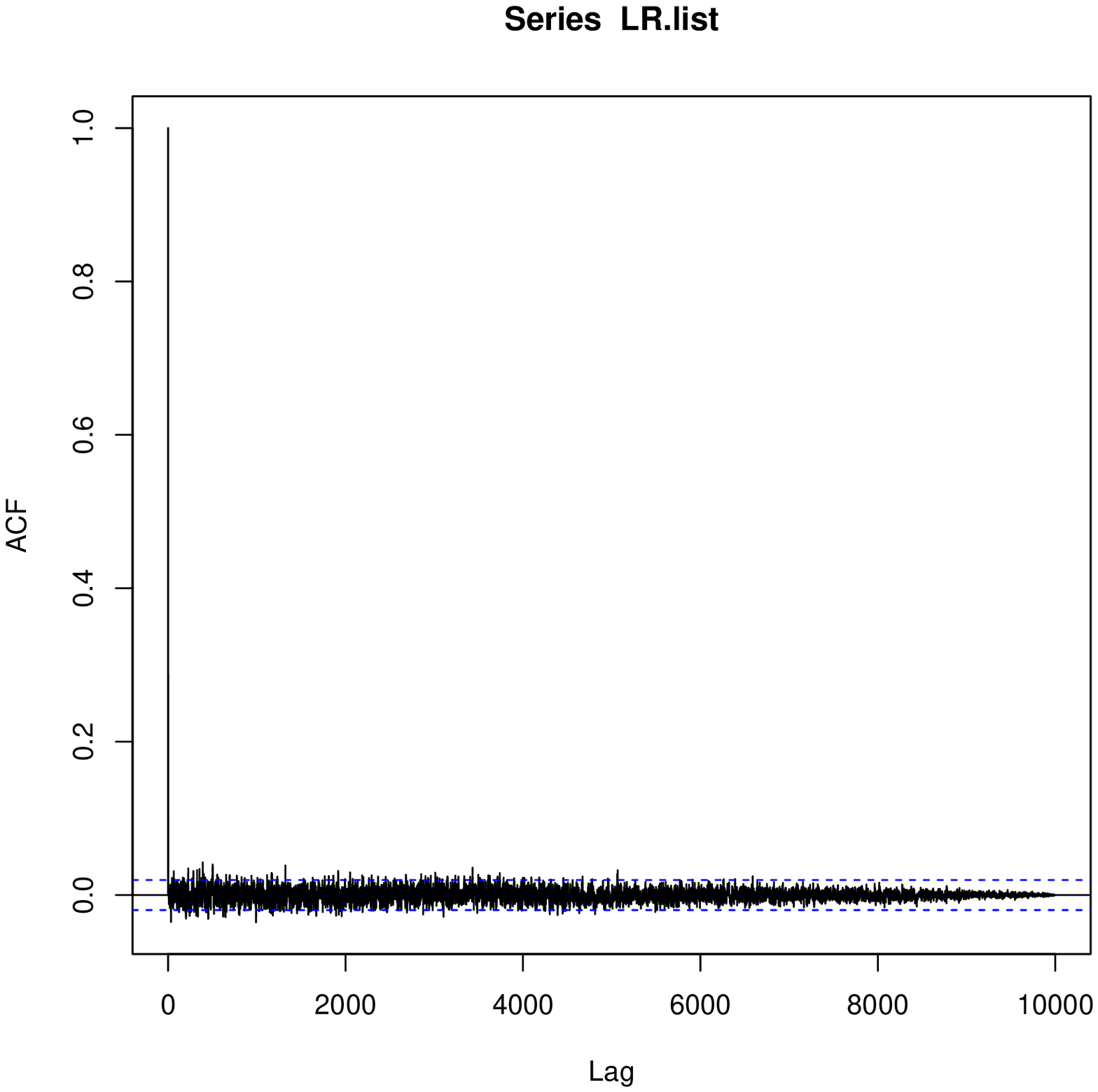}\\
 (b) binomial, a lattice basis with $Geom(0.5)$\\
\includegraphics[scale=0.20]{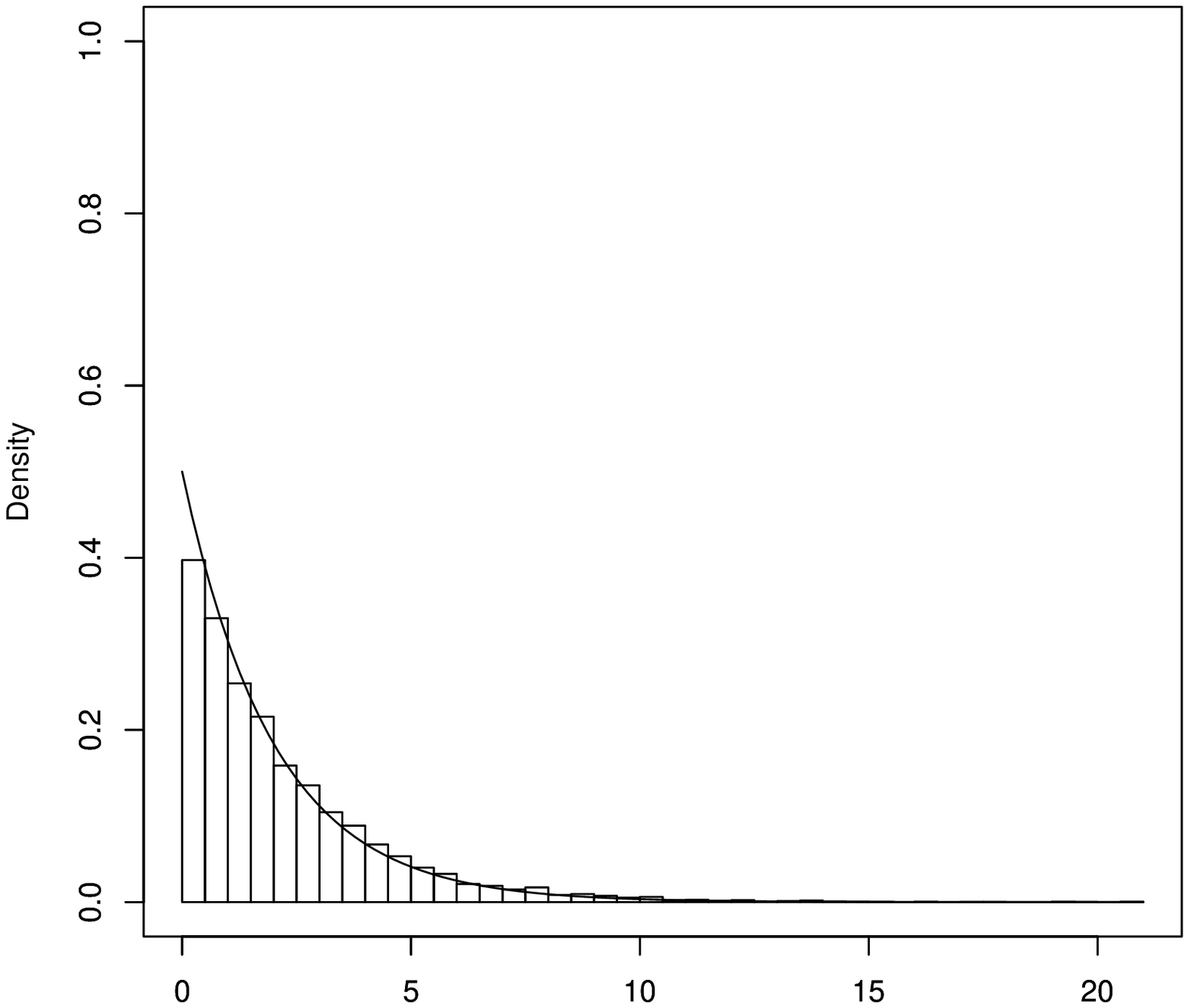}
\includegraphics[scale=0.20]{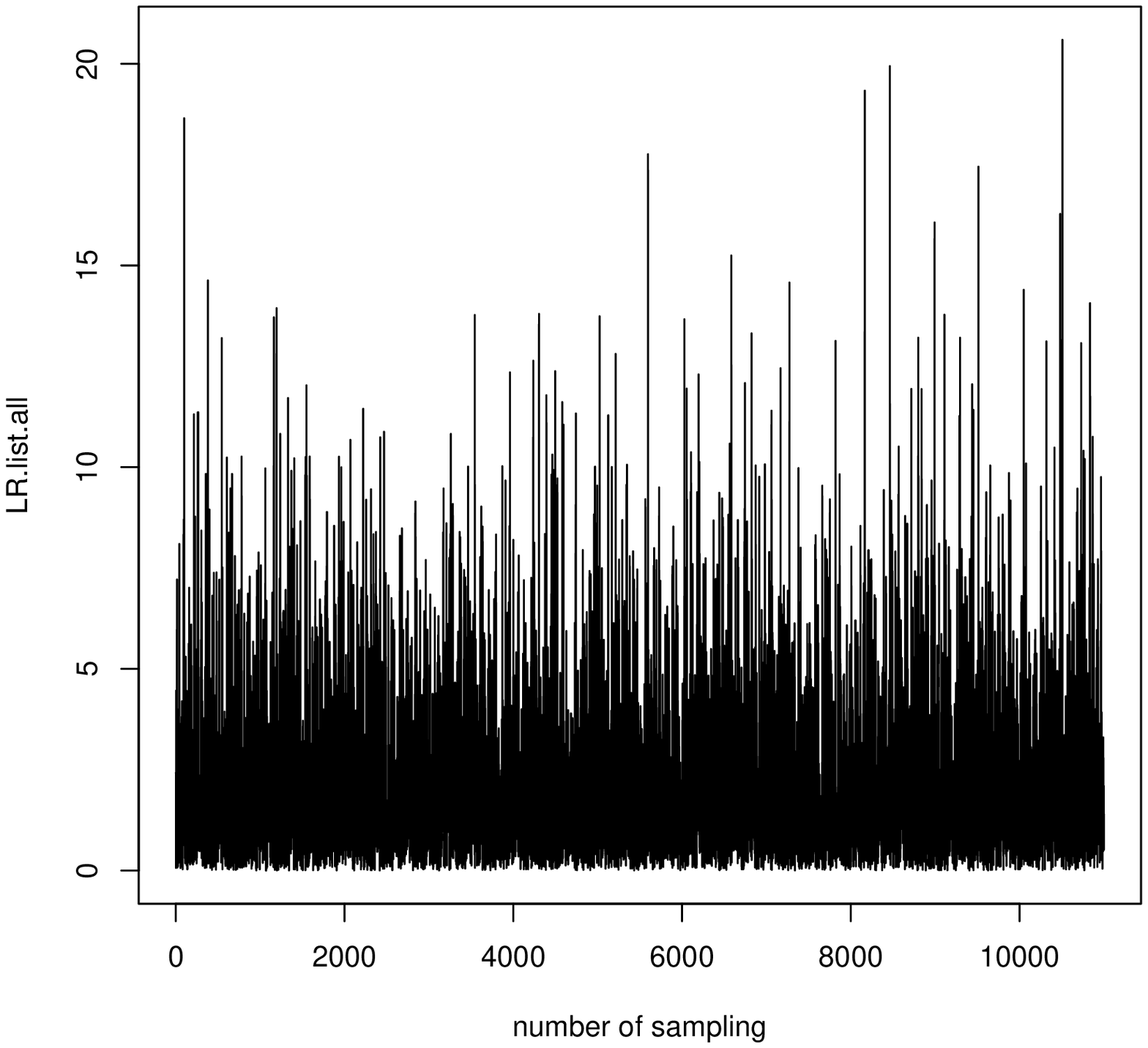}
\includegraphics[scale=0.20]{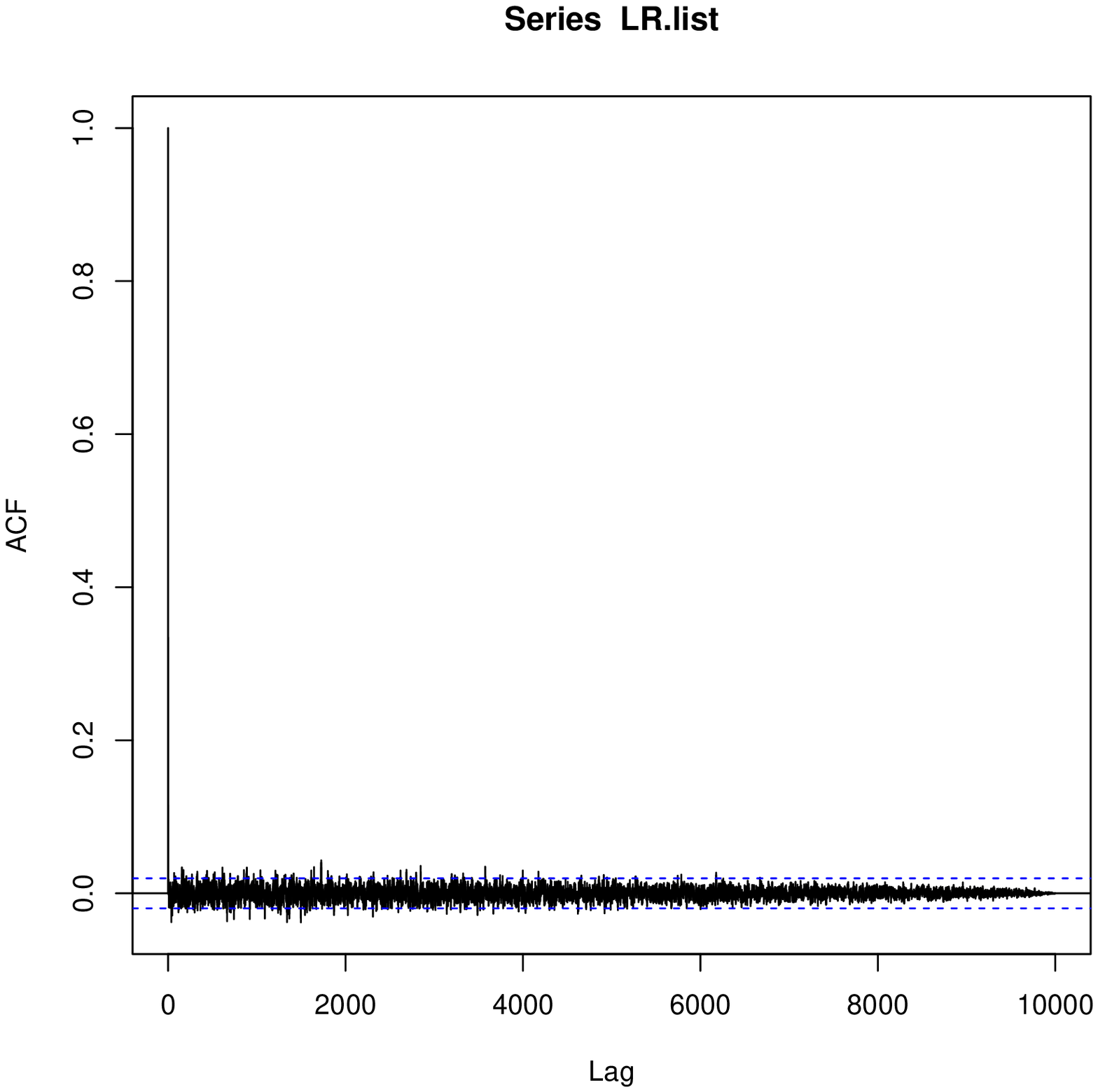}\\
 (c) trinomial, a lattice basis with $Geom(0.1)$\\
\includegraphics[scale=0.20]{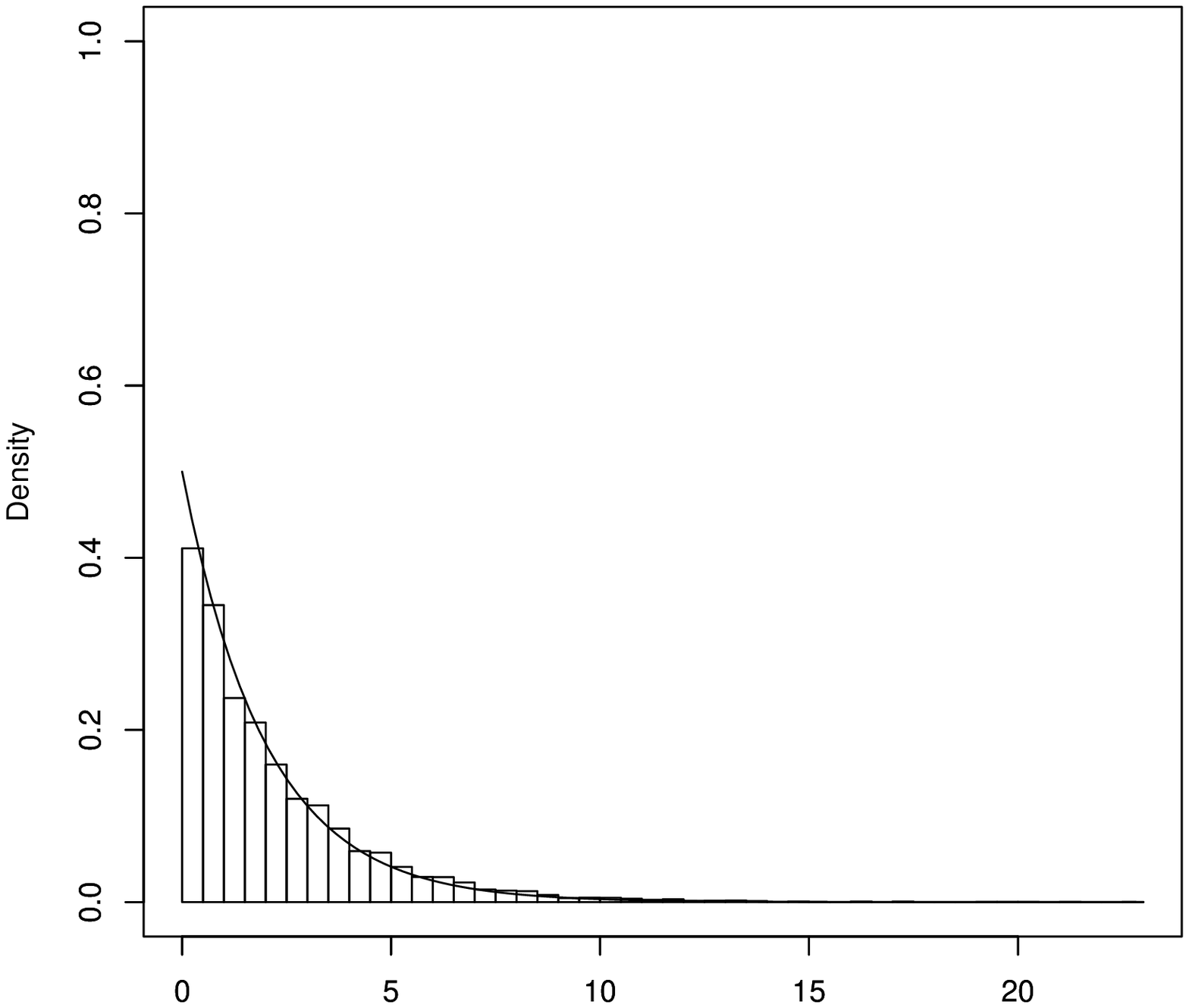}
\includegraphics[scale=0.20]{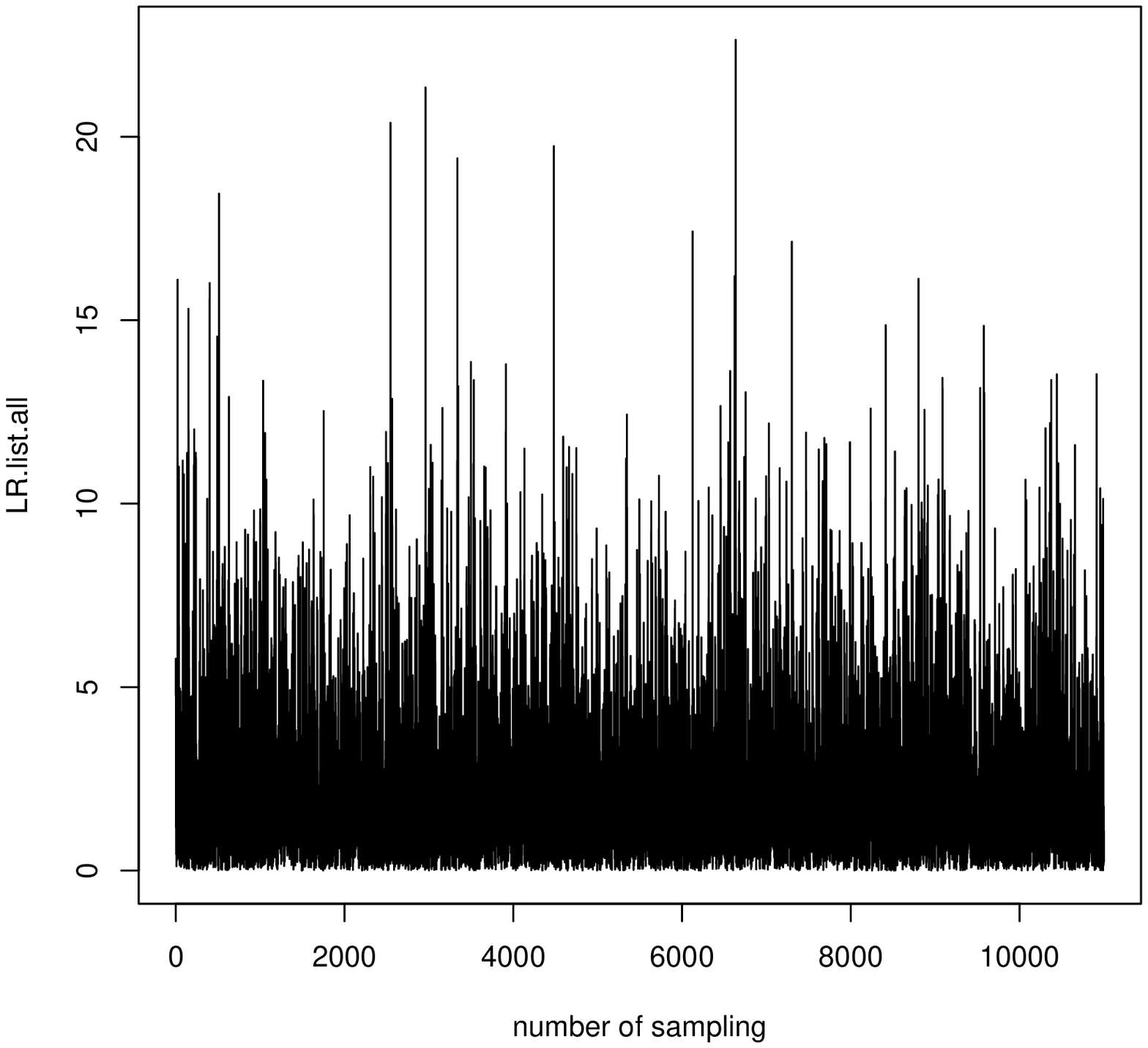}
\includegraphics[scale=0.20]{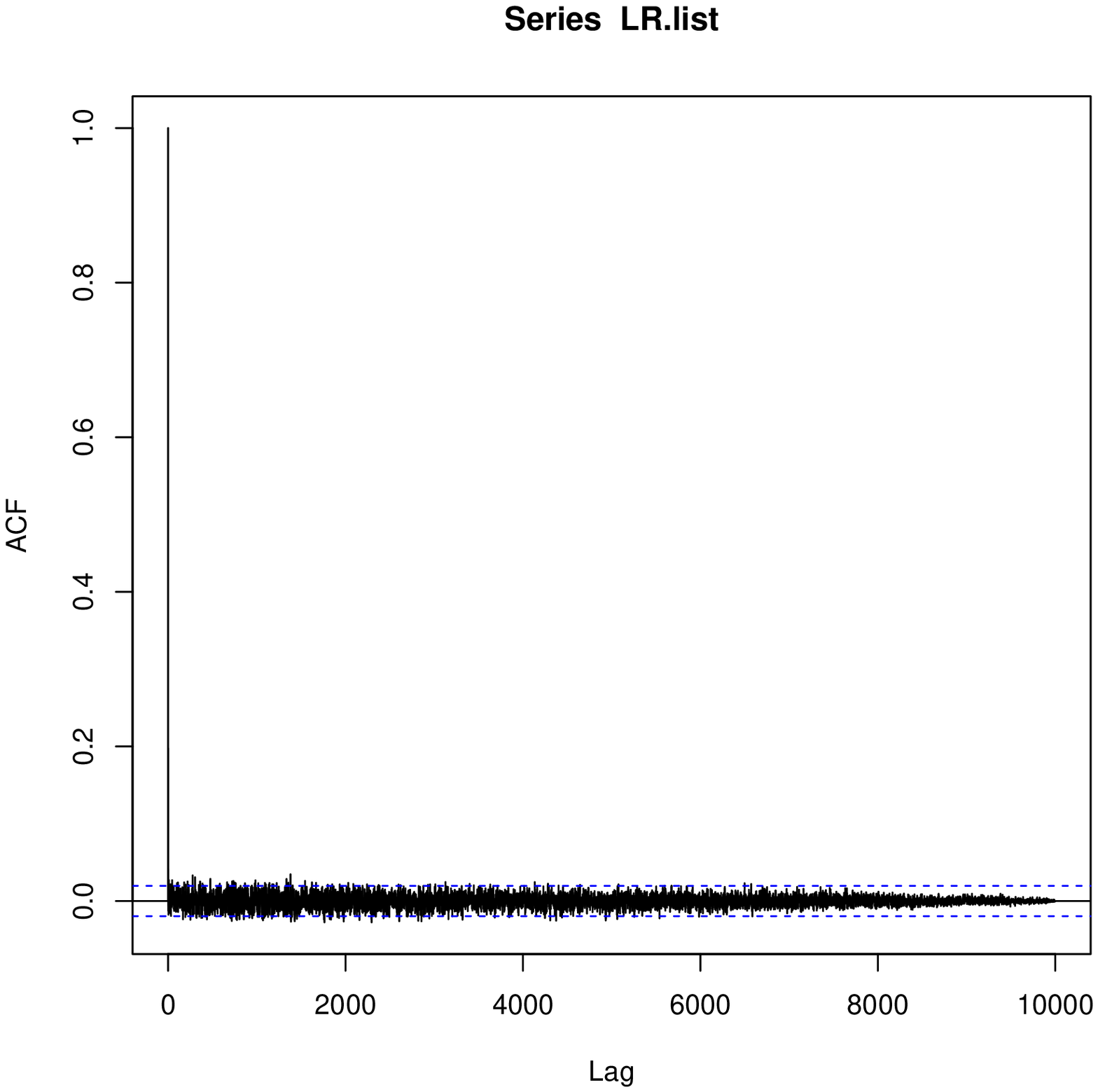}\\
 (d) trinomial, a lattice basis with $Geom(0.1)$\\
 \caption{Histograms, paths of LR statistic and correlograms of paths
 for discrete logistic regression model ((burn in,iteration) $= (1000,10000)$)} 
 \label{fig:logit3}
\end{figure}

\bibliographystyle{ws-procs9x6}
\bibliography{Hara-Aoki-Takemura-LB-arxiv}

\end{document}